\documentclass[preprint,11pt]{elsarticle}
\usepackage[utf8]{inputenc}
\usepackage[english]{babel}
\makeatletter
\def\ps@pprintTitle{%
 \let\@oddhead\@empty
 \let\@evenhead\@empty
 \def\@oddfoot{}%
 \let\@evenfoot\@oddfoot}
\makeatother
\usepackage[a4paper, margin=1.1in]{geometry}
\usepackage{mathtools,amsmath,amsfonts,amssymb,ntheorem}
\usepackage{lineno}
\usepackage[font=small,labelfont=bf]{caption}
\usepackage{bbm}
\usepackage{tikz,tikz-cd}
\usetikzlibrary{calc}
\usepackage{float}
\usepackage{subfig}
\usepackage{bbm}
\usepackage{subfiles}

\usepackage{operators}
\usepackage{nomencl}
\makenomenclature

\usepackage{multicol}

\usepackage{etoolbox}
\renewcommand\nomgroup[1]{%
  \item[\bfseries
  \ifstrequal{#1}{L}{Latin symbols}{%
  \ifstrequal{#1}{G}{Greek symbols}{%
  \ifstrequal{#1}{O}{Other}{}}}%
]}

\numberwithin{equation}{section}

\newtheorem{thm}{THEOREM}[section]
\newtheorem{lem}[thm]{LEMMA}

\newtheorem{prop}[thm]{PROPOSITION}

\newtheorem*{definition}{{\bf DEFINITION.}}
\newtheorem{remark}{{\bf Remark}}

\newenvironment{proof}{\paragraph{\bf Proof}}{\hfill$\square$\newline}

\begin{document}
\begin{frontmatter}
\title{Energy stable finite element scheme for simulating flow dynamics of droplets on non--homogeneous surfaces}
\author[1]{Filip Ivan\v ci\' c}
\ead{fivanci@nhri.edu.tw}

\author[1,2]{Maxim Solovchuk\corref{cor1}}
\ead{solovchuk@gmail.com}
\cortext[cor1]{Corresponding author}

\address[1]{Institute of Biomedical Engineering and Nanomedicine, National Health Research Institutes, No. 35, Keyan Road, Zhunan 35053, Taiwan}
\address[2]{Department of Engineering Science and Ocean Engineering, National Taiwan University, No. 1, Sec. 4, Roosevelt Road, Taipei 10617, Taiwan}

\begin{abstract}
An energy stable finite element scheme within {\it arbitrary Lagrangian Eulerian} (ALE) framework is derived for simulating the dynamics of millimetric droplets in contact with solid surfaces. Supporting surfaces considered may exhibit non--homogeneous properties which are incorporated into system through {\it generalized Navier boundary conditions} (GNBC). Numerical scheme is constructed such that the counterpart of (continuous) energy balance holds on the discrete level. This ensures that no spurious energy is introduced into the discrete system, i.e. the discrete formulation is stable in the energy norm. The newly proposed scheme is numerically validated to confirm the theoretical predictions. Of a particular interest is the case of droplet on a non--homogeneous inclined surface. This case shows the capabilities of the scheme to capture the complex droplet dynamics (sliding and rolling) while maintaining stability during the long time simulation.
\end{abstract}

\begin{keyword}
arbitrary Lagrangian--Eulerian\sep dynamic contact line\sep energy stability\sep discrete energy balance\sep millimetric droplets
\end{keyword}

\end{frontmatter}

\section{Introduction}
When a liquid is geometrically constrained to a small scale, surface forces become comparable to volume forces and greatly influence the dynamics of such flows. For example, the dynamics of a small droplet in contact with a rigid surface is governed by surface tension, wetting and gravity forces. Hydrophobic and hydrophilic surfaces are often seen in nature and lead to many interesting phenomena such as self--cleaning properties (e.g., see \cite{blossey03} and references therein). A good comprehensive review on microfluidics and applications in technical fields may be found in \cite{venkatesan20,malinowski20}. 

The aim of this paper is to derive an energy stable {\it finite element method} (FEM) for simulating the dynamics of small droplets on solid, possibly non--homogeneous and inclined surfaces. The main challenges in the numerical framework for such problems come from the description of the evolving geometry. Surface tension, which plays an essential role in the mathematical model for droplet dynamics, is a function of the mean curvature which depends on the geometry of fluid--fluid interface. A recent review on the numerical models for surface tension is given in \cite{popinet2018}. Two essentially different approaches are widely used in the literature to describe the moving interface: implicit and explicit. In the implicit approach, a time--dependent scalar field is introduced on a fixed computational mesh to track the evolution of the interface. This is often referred to as the {\it interface capturing} approach and it is in a tight relationship with mesh adaptation methods. Indeed, mesh has to be sufficiently dense in the neighborhood of the interface in order for the position of the interface to be credibly captured on the discrete level. {\it Level set method} (see \cite{sethian03} and references therein) and {\it volume of fluid method} (see \cite{hirt81} and references therein) are representatives of the interface capturing approach. In {\it interface tracking} approach, the interface is described explicitly with an aligned mesh, i.e. the mesh {\it fits} the interface. Within this framework, if the interface moves, the mesh must also be moved with it accordingly. {\it Lagrangian} and {\it arbitrary Lagrangian Eulerian} (ALE) methods fall within the latter approach. Detailed description of the ALE methods can be found in \cite{donea04} while a concise review is given in Section~\ref{sec:ALE} in this paper. Interface tracking methods are in general favored over interface capturing methods when a precise position of the interface is required. Their main disadvantage is the non--ability to capture possible topological changes in the domain.

A great challenge within ALE FEM is to ensure energy stability on the discrete level. If a scheme is stable in the energy norm, then there is no spurious energy brought into the system due to discretization. Such schemes are relatively easy to construct for the problems on fixed meshes, but significantly more challenging to derive on moving meshes. Energy stable schemes are often studied in context of {\it fluid--structure interaction} (FSI) problems modeled within ALE framework (see, e.g., \cite{lozovskiy15,hecht17,formaggia09} and references therein). For FSI problems, energy balance of ALE FEM scheme is usually based on two main ingredients (within {\it monolithic} approach): {\it space conservation law} (SCL) and proper coupling of fluid and structure (\cite{formaggia09}). In case discrete SCL is not satisfied, artificial energy sinks and/or sources appear in the numerical scheme, and spurious energy is brought into the system. More generally, SCL is closely intertwined with the stability of moving mesh methods -- it roots from the Reynolds transport theorem when transiting from the continuous to the discrete level. More details on SCL and stability issues may be found in \cite{formaggia99,ivancic19} and references therein. In context of simulating the free surface flows within ALE framework, SCL plays the same role for numerical stability as in the case of FSI. For the case of the free surface flows, however, sources of spurious energies are also known to appear on fluid--fluid interface, fluid--solid interface, and triple contact line. 

Energy stability and construction of energy stable schemes for a special case of free surface flows was already studied in \cite{gerbeau09}. There, they studied a simplified case where mesh velocity has only the vertical component. Such cases appear, for example, when studying a container filled with water and assuming there is no wave breaking. They derived a simplified energy balance of the governing system (ignoring the wetting energy) and investigated the energy stability of few numerical schemes. Their energy stability is based on satisfying the discrete SCL and its generalization for curved two--dimensional surfaces embedded in three--dimensional space. Essentially, they used a generalization of Reynolds transport theorem for manifolds of co--dimension one to derive the corresponding generalization of SCL. Furthermore, a scheme was proposed which is energy stable under the assumption of sufficiently small time step. They performed numerical validation of the theoretical results for the two--dimensional scenario and investigated implicit versus explicit treatments of interface displacement. It was shown that the explicit treatment of the interface displacement may introduce spurious energy into the discrete system when gravity and surface tension are taken into the account. In term, such spurious energy may cause a scheme blowup.

In this work an energy stable ALE FEM scheme for general free surface flows  is derived. Our focus is on liquid droplets in contact with solid surfaces, where wetting phenomenon may be observed. The framework we develop is quite general and may be applied to various setups, such as non--homogeneous or inclined supporting surfaces. By non--homogeneous surface, a surface whose physical and/or chemical properties may vary in space and in time is implied. For such surfaces, droplets in contact exhibit different wetting properties in different configurations. Hence, their behavior may be potentially controlled by surface properties (see, e.g.,  \cite{sadullah2020,mazaltarim2021}). Indeed, we investigate a couple of interesting examples in Section~\ref{sec:numerical-validation}, including droplets on non--homogeneous inclined surfaces exhibiting both sliding and rolling kinematics. Mathematical model describing the kinematics of droplets on solid surfaces consists of Navier--Stokes equations governing the fluid flow, coupled with equations for the mesh velocity governing the geometry evolution. Fluid equations have to be subjected to boundary conditions which credibly describe the situation on the interface. This proved to be quite challenging for viscous flows in contact with both solid and gaseous phases. Classical {\it no--slip} boundary conditions cannot capture the physics of fluid near the triple contact line (solid--gas--fluid interface). It has been observed that the contact line is able to move despite the apparent {\it no--slip} of the fluid at the wall. So called {\it generalized Navier boundary conditions} (GNBC) derived from the {\it molecular dynamics} (MD) theory for the Navier--Stokes equations have been proven to capture the dynamics of the triple contact line adequately (see \cite{qian03,qian06,qian06var,ren07}). GNBC have already been successfully applied for simulating free surface flows with ALE FEM for various multiphysics phenomena such as free surface flows in \cite{gerbeau09}, sliding/rolling droplets kinematics in \cite{ganesan09,willassen14}, droplets kinematics on non--isothermal surfaces in \cite{venkatesan16}, chemotaxis phenomena in \cite{ivancic19-chemo,ivancic20-chemo}, to name a few. The main goal of this paper is the derivation of the energy stable numerical scheme based on the mathematical model incorporating GNBC and not the investigation of GNBC themselves. However, we do perform some numerical validation of GNBC within the ALE FEM since they weren't originally developed within this framework.

The main ingredients for the energy stability of the proposed scheme are discrete space conservation laws for manifolds of co--dimension zero (traditional SCL, see \cite{formaggia99,ivancic19}) and of co--dimension one (generalized SCL introduced in \cite{gerbeau09}). Furthermore, an accurate evaluation of the curvature of fluid--fluid interface is essential to avoid birth of spurious velocities. This is particularly challenging in context of polygonal meshes since curvature is, roughly, a function of second derivatives of the interface. Some insights on this manner are given in \cite{ivancic20,ganesan07,cenanovic20}. The newly proposed energy stable scheme is demonstrated for a couple of different scenarios. Particular emphasis is on the three dimensional case. This is important since certain physics simply cannot be captured in two dimensional case. For example, non--homogeneous surfaces only become interesting in three dimensions. A full three dimensional model is significantly more challenging than its two--dimensional counterpart. Apart from the obvious increase in computational cost, challenges also arise due to curvature reconstruction. Reconstructing the curvature from discrete data is much simpler in the two--dimensional setup, while in three--dimensions it may exhibit instabilities even for fairly simple geometries (see \cite{cenanovic20}). Curvature plays an essential role in capillary flows.

This paper is organized as follows. In Section~\ref{sec:math-model}, mathematical model governing the flow kinematics of a droplet on solid surface is reviewed. In Section~\ref{sec:ALE}, the arbitrary Lagrangian Eulerian framework is reviewed and the notation is introduced. Some of the main challenges of ALE framework are highlighted. In Section~\ref{sec:varf}, variational formulation of the governing system and total energy balance are derived. In Section~\ref{sec:energy-stable-scheme}, a novel energy stable ALE FEM formulation is proposed. Theoretical estimates for the total discrete energy are presented. In Section~\ref{sec:numerical-validation}, numerical validation of theoretical results is performed on a few practical examples. Finally, in Section~\ref{sec:conclusion}, a short summary of this work is presented and the conclusions are drawn. 

\section{Mathematical modeling}\label{sec:math-model}
In this section we concisely describe the mathematical model governing the kinematics of small droplets on solid surfaces in a gravity field. By small, we mean that the length scale is close to the {\it capillary length} so surface tension and gravity forces are comparable. The mathematical model for free surface flows reviewed below has already been successfully applied in \cite{gerbeau09,ganesan09,willassen14}. We start with the mathematical description of geometry. Nomenclature is summarized in \ref{apx:nomenclature}.

\subsection{Geometry description}
Let $\Omega\subset\bbmR^d$, $d=2,3$, denote the domain occupied by the fluid at time $t\geq 0$. Note that the domain is a function of time, $\Omega=\Omega(t)$. Furthermore, assume that $\partial\Omega$ is a disjoint union of $\Gamma$ and $\Sigma$, $\partial\Omega=\Gamma\cup\Sigma$, where  $\Sigma$ and $\Gamma$ have strictly positive measures, $\vert\Gamma\vert>0$ and $\vert\Sigma\vert>0$. In what follows, we denote the gas--liquid interface by $\Sigma$ and the solid--liquid interface by $\Gamma$. In this context, the triple contact line, i.e. the solid--liquid--gas interface, is denoted by $\eta$, $\eta=\Gamma\cap\Sigma$. Note that, in aforementioned notation, $\eta=\partial\Sigma=\partial\Gamma$. The initial setup and the geometry notation are illustrated in Figure~\ref{fig:setup}.

\begin{figure}[h]
\centering
\subfloat[]{\includegraphics[width=0.4\linewidth]{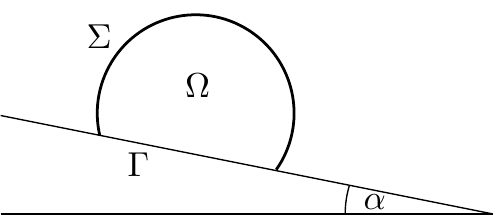}\label{fig:setup-2d}}\hspace{0.1\linewidth}\subfloat[]{\includegraphics[width=0.4\linewidth]{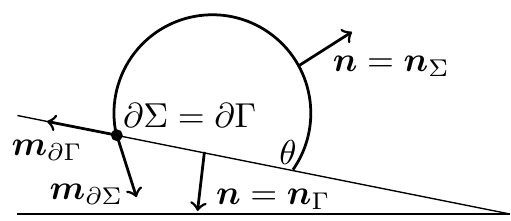}\label{fig:2dnormals}}

\vspace{0.01\linewidth}

\subfloat[]{\includegraphics[width=0.4\linewidth]{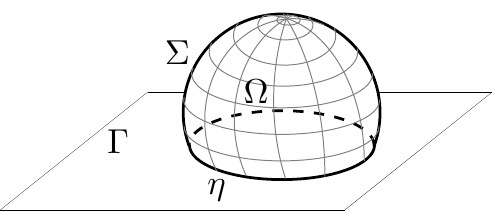}\label{fig:setup-3d}}\hspace{0.1\linewidth}\subfloat[]{\includegraphics[width=0.4\linewidth]{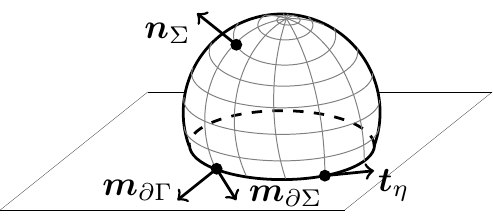}\label{fig:3dnormals}}

\caption{Illustration of the initial setup and the notation in two and three dimensions. }
\label{fig:setup}
\end{figure}

Let us denote by $\bfn$ the outer unit normal to the boundary $\partial\Omega$. When more detail information is needed, a subscript $\Sigma$ or $\Gamma$ is added which specifies the outer unit normal to $\Sigma\subset\partial\Omega$ or $\Gamma\subset\partial\Omega$, respectively. Furthermore, define the unit tangent vector to $\eta$ by $\bft_\eta=\bfn_\Sigma\times\bfn_\Gamma$. Within present notation, $\bft_\eta=\bft_{\partial\Gamma}=\bft_{\partial\Sigma}$. Then, the unit co--normal vectors to $\partial\Sigma$ and $\partial\Gamma$ are given by
\[
\bfm_{\partial\Sigma} = \bft_\eta\times\bfn_\Sigma\textrm{ and } \bfm_{\partial\Gamma} = \bft_\eta\times\bfn_\Gamma,
\]
and the contact angle $\theta$ between $\Gamma$ and $\Sigma$ is defined via the relation
\begin{equation}\label{eq:dyn_cont_angle_def}
\cos\theta = \bfm_{\partial\Sigma}\cdot\bfm_{\partial\Gamma}.
\end{equation}
Triples $(\bft_\eta,\bfn_\Sigma,\bfm_{\partial\Sigma})$ and $(\bft_\eta,\bfm_{\partial\Gamma},\bfn_\Gamma)$ define positively oriented orthonormal basis (on \(\eta\)). See Figure~\ref{fig:setup} for illustration.

By $\bfx_\Sigma$ we denote the inclusion from $\Sigma$ to $\bbmR^d$,
\[
\bfx_\Sigma\colon\Sigma\to\bbmR^d\;,\:\Sigma\ni\bfx\mapsto\bfx\in\bbmR^d,
\]
and by $\sT_{\bfx}\Sigma$ the tangential plane to $\Sigma$ at $\bfx\in\Sigma$. The tangent bundle of $\Sigma$ is denoted by $\bfT\Sigma$ and the orthogonal projection to $\bfT\Sigma$, $\PS$, is defined as
\[
\PS = \bbmI - \bfn\otimes\bfn,
\]
where $\bbmI$ is the identity ($d\times d$) matrix and $\otimes$ denotes the standard tensor product between vectors. For a scalar field $\phi\colon\Sigma\subset\bbmR^d\to\bbmR$, tangential (to $\Sigma$) gradient and divergence are then defined as
\[
\nabla_\Sigma\phi = P_\Sigma\nabla\phi\textrm{ and }\div_\Sigma\phi = \trace\nabla_\Sigma\phi.
\]
Corresponding maps and operators on $\Gamma$ are defined analogously. Laplace--Beltrami operator is then defined as $\Delta_\Sigma = \div_\Sigma\nabla_\Sigma$, and mean curvature vector $\bfh$ may be expressed as
\begin{equation}\label{eq:curvature}
\bfh\colon\Sigma\to\bbmR^d\;,\: \bfh=-\Delta_\Sigma\bfx_\Sigma.
\end{equation}
Mean curvature $H$ is then given by the identity $\bfh = 2H\bfn$. If by $\kappa_1$ and $\kappa_2$ we denote the principal curvatures of $\Sigma$, then
\[
H(\bfx_\Sigma) = \frac{1}{2}\Big( \kappa_1(\bfx_\Sigma) + \kappa_2(\bfx_\Sigma) \Big)\;,\:\bfx_\Sigma\in\Sigma.
\]

The coordinate system is chosen such that $\Gamma$ lies in $xy$--plane (or $x$--axis for the 2D case). Hence, direction opposite to direction of the gravity, $\bfk$, may be defined as
\begin{equation}\label{eq:inclination-k}
\bfk = [-\sin\alpha,\cos\alpha]^T\textrm{ if }d=2\textrm{ and }\bfk = [-\sin\alpha,0,\cos\alpha]^T\textrm{ if }d=3,
\end{equation}
where $\alpha$ is the inclination angle of the supporting plane $\Gamma$. This choice of coordinate system significantly simplifies the implementation of Dirichlet boundary conditions on $\Gamma$ for the velocity field (see Remark~\ref{rem:dirichlet-bc-impl}). 

\subsection{Governing equations}
The mathematical model governing the kinematics of a droplet in contact with a rigid surface consists of incompressible Navier--Stokes equations subjected to cautiously chosen boundary conditions. The main problem are the boundary conditions which describe the moving triple contact line $\eta$. Within ALE framework, the phenomenon of moving contact line is credibly captured by imposing the {\it generalized Navier boundary conditions} (GNBC).

Let us denote the fluid velocity by $\bfv$ and the pressure by $p$. Furthermore, $\varrho$ denotes the fluid density, $\mu$ the dynamic viscosity, $g$ the magnitude of the gravity acceleration, $\gamma$ the surface tension, $\varsigma$ the friction slip coefficient and $\theta_s$ the static contact angle. Some details on the static contact angle are given in Remark~\ref{rem:theta}. The governing system of equations in dimensional form then reads: in $\Omega(t)$, for $t>0$,
\begin{equation}\label{sys:NSdim}
\begin{split}
\varrho\left( \frac{\partial}{\partial t} \bfv + [\bfv\cdot\nabla]\bfv \right) - \div\bfsigma &= -\varrho g\bfk,\\
\div\bfv &= 0,
\end{split}
\end{equation} 
where
\begin{equation}
\begin{split}
\bfsigma &= -p\bbmI + \mu\bbmD(\bfv)\textrm{, with }\bbmD(\bfv)=[\nabla\bfv]^T+[\nabla\bfv].
\end{split}
\end{equation}
Note that Cauchy stress tensor $\bfsigma$ characterizes a Newtonian fluid. System (\ref{sys:NSdim}) is subjected to the following boundary conditions: for $t>0$
\begin{equation}\label{sys:NSdim-bc}
\begin{split}
\bfsigma\bfn &= \gamma\Delta_\Sigma\bfx_\Sigma\textrm{ on }\Sigma(t),\\
\bfv\cdot\bfn &= 0 \textrm{ on }\Gamma(t),\\
\bfsigma\bfn\cdot\bftau &= - \varsigma(\bfv-\bfu)\cdot\bftau - \gamma(\cos\theta - \cos\theta_s)\bfm_{\partial\Gamma}\cdot\bftau \delta_\eta \textrm{ on }\Gamma(t),
\end{split}
\end{equation}
where $\bftau\in \bfT\Gamma$ is arbitrary, $\delta_\eta$ is {\it Dirac's $\delta$--distribution} which localizes the last term in (\ref{sys:NSdim-bc})$_3$ to curve $\eta$, and $\bfu$ is the velocity of $\Gamma$ ($\bfu=0$ for fixed $\Gamma$). Note, (\ref{sys:NSdim-bc})$_1$ may be rewritten in terms of curvature using identity (\ref{eq:curvature}): on $\Sigma$
\[
\bfsigma\bfn = \gamma\Delta_\Sigma\bfx_\Sigma = -\gamma 2H\bfn.
\]
\begin{remark}
We assume that surface tension $\gamma$ is constant at all times. This is a reasonable and obvious assumption since we only study the isothermal case. In practice, $\gamma$ may be a function of temperature and/or surfactants (surface active agents). The derivation of energy stable schemes of such more general cases may be studied in future.
\end{remark}
\begin{remark}
Throughout the rest of the paper we write the gravity forcing term as the gradient of gravity potential $\Phi=\Phi(\bfx)$ where
\begin{equation}
\Phi(\bfx)=\varrho g(\bfk\cdot\bfx).
\end{equation}
Then, $\varrho g \bfk = \nabla\Phi(\bfx)$. 

All analyses performed in this paper may be applied to more general cases with arbitrary conservative forces.
\end{remark}
\begin{remark}\label{rem:theta}
Static contact angle $\theta_s$ and friction slip coefficient $\varsigma$ typically depend on the properties of supporting surface $\Gamma$ and the fluid of interest. Properties of the supporting solid surfaces may be controlled by, e.g., applying {\it hydrophobic coatings} (see \citep{nakajima11} and references therein) or by alternating their micro--structures (see \cite{yin14} and references therein).

Throughout this paper, we assume that $\theta_s$ and $\varsigma$ are given functions of $\Gamma$, i.e. for $\bfx_\Gamma\in\Gamma$
\[
\theta_s=\theta_s(\bfx_\Gamma)\;,\:\varsigma=\varsigma(\bfx_\Gamma).
\]
Varying $\theta_s$ and $\varsigma$ in space corresponds to non--homogeneous properties of the surface $\Gamma$.
\end{remark}

\begin{remark}\label{rem:dirichlet-bc-impl}
Taking into account the choice of coordinate system, Dirichlet boundary condition (\ref{sys:NSdim-bc})$_2$ states that the vertical component of velocity equals to zero. This is significantly simpler to implement than general condition $\bfv\cdot\bfn=0$ which requires an involvement of Lagrange multipliers or Nietsche's method (see, e.g., \cite{stenberg95}).
\end{remark}

\subsection{Dimensionless system}
In order to perform numerical simulations, we first derive the dimensionless form of system (\ref{sys:NSdim})--(\ref{sys:NSdim-bc}). It is also more convenient, from the practical point of view, to work with dimensionless equations during the analysis of the system. 

Let us for a moment emphasize the dimensionless quantities with the {\it overbar} symbol. Then, physical variables are non--dimensionalized as follows
\[
\overline{\bfx} = \frac{\bfx}{L}\;,\: \overline{t} = \frac{t}{t_c}\;,\: \overline{\bfv}=\frac{\bfv}{U},\:\overline{p}=\frac{p}{p_c}\;,
\]
where $p_c$ denotes the characteristic pressure, $t_c$ the characteristic time, $U$ the characteristic velocity and $L$ the characteristic length. Characteristic pressure and time may be defined in terms of characteristic length and velocity as
\[
p_c=\varrho U^2 \textrm{ and } t_c=\frac{L}{U}\;.
\]
Characteristic length and velocity are to be defined later; their choice depends on the problem considered. Dimensional analysis then gives rise to the following dimensionless numbers
\[
\noFr=\frac{U}{\sqrt{gL}}\;,\:\noRe=\frac{\varrho U L}{\mu}\;,\:\noCa=\frac{\mu U}{\gamma}\;,\:\overline{\varsigma} = \frac{\varsigma}{\varrho U}\;,
\]
where $\noFr$ is the Froude number, $\noRe$ is the Reynolds number and $\noCa$ is the capillary number. Furthermore, the dimensionless friction slip coefficient and the static contact angle may be rewritten as functions of $\overline{\bfx_\Gamma}=\bfx_\Gamma/L$ if the functional dependence on $\bfx_\Gamma$ is known. Then, slightly abusing the notation by dropping the {\it overbar} for the dimensionless quantities, system (\ref{sys:NSdim}) is rewritten in its dimensionless form as: in $\Omega(t)$, for $t>0$,
\begin{equation}\label{sys:NSnondim}
\begin{split}
\frac{\partial}{\partial t} \bfv + [\bfv\cdot\nabla]\bfv - \div\bfsigma &= -\nabla\Phi\;,\\
\div\bfv &= 0\;,
\end{split}
\end{equation} 
where the dimensionless form of the stress tensor $\bfsigma$ reads
\begin{equation}
\bfsigma = -p\bbmI + \frac{1}{\noRe}\bbmD(\bfv)\;,
\end{equation}
and dimensionless form of the gravity potential is
\begin{equation}
\Phi(\bfx) = \frac{1}{\noFr^2}(\bfk\cdot\bfx).
\end{equation}
System (\ref{sys:NSnondim}) is completed with dimensionless boundary conditions obtained from (\ref{sys:NSdim-bc}): for $t>0$ and $\bftau\in\bfT\Gamma$ arbitrary,
\begin{equation}\label{sys:NSnondimbc}
\begin{split}
\bfsigma\bfn &= \frac{1}{\noCa\noRe}\Delta_\Sigma\bfx_\Sigma\textrm{ on }\Sigma(t),\\
\bfv\cdot\bfn &= 0 \textrm{ on }\Gamma(t),\\
\bfsigma\bfn\cdot\bftau &= - \varsigma(\bfv-\bfu)\cdot\bftau - \frac{1}{\noCa\noRe}(\cos\theta - \cos\theta_s)\bfm_{\partial\Gamma}\cdot\bftau \delta_\eta \textrm{ on }\Gamma(t),
\end{split}
\end{equation} 
where we dropped the {\it overline} symbol in dimensionless Navier slip coefficient $\varsigma$. Characteristic length and velocity have to be chosen according to the specific problem we consider. For example, in context of millimetric droplets where surface tension is among the dominant forces, characteristic length should be somehow related to the droplet curvature radius.

\section{Arbitrary Lagrangian--Eulerian framework}\label{sec:ALE}
The general idea behind ALE framework is the interplay between the reference and physical configurations. By physical configuration we mean the current, instantaneous configuration occupied by the fluid. Physical configuration is in one--to--one correspondence with stationary reference configuration through the ALE map. More specifically, let $\kOmega$ denotes the reference configuration and let $\kALE\colon\kOmega\times[0,T]\to\bbmR^d$ describe the evolution of $\Omega$, i.e.
\[
\Omega(t) = \image(\kALE(\cdot,t)),
\] 
where $\image(\kALE(\cdot,t))$ denotes the image of $\kALE(\cdot,t)$. For practical purposes, it may be assumed that reference configuration coincides with the initial configuration $\kOmega=\Omega(0)$. Regarding the regularity, it is assumed that, for $t\in[0,T]$, $\kOmega$ and $\image(\kALE(\cdot,t))$ are bounded domains with Lipschitz boundary, and 
\[
\kALE(\cdot,t)\in\sW^{1,\infty}(\kOmega;\bbmR^d)\;,\:\kALEinv(\cdot,t)\in\sW^{1,\infty}(\Omega(t);\bbmR^d).
\]
This is sufficient for the following condition to hold:
\[
\phi\in\sH^1(\Omega(t)) \: \textrm{if and only if} \: \widehat{\phi}=\phi\circ\kALE(\cdot,t)\in\sH^1(\kOmega).
\]
For more details on necessary regularity of the ALE map we refer to \cite{formaggia99}. Furthermore, let us denote the space--time domain of interest as
\begin{equation}
Q_T = \bigcup\limits_{t\in(0,T)}\Big( \Omega(t)\times\{t\} \Big),
\end{equation}
where $T$ denotes the final time. Then, in terms of the previously introduced ALE map $\kALE$, we may write
\begin{equation}\label{id:ale}
Q_T\ni(\bfx,t) = (\kALE(\kbfx,t),t)\textrm{, for }\kbfx\in\kOmega\textrm{ and }t\in(0,T).
\end{equation}
Now, for any field $f$ (scalar, vector or tensor) defined on physical configuration $Q_T$, $f\colon Q_T\to\bbmR^m$, $m\in\mathbb{N}$, its ALE counterpart $\widehat{f}$ is simply a composition with $\kALEinv$, i.e.
\[
\widehat{f}(\widehat{\bfx},t)=f(\bfx,t) \textrm{ for }  \widehat{\bfx}\in\kOmega\;,\:\bfx=\kALE(\widehat{\bfx},t)\in\Omega(t).
\]
In this sense, the domain velocity $\bfw$ is defined as 
\begin{equation}\label{eq:ALEw}
\begin{split}
& \widehat{\bfw}(\widehat{\bfx},t) = \frac{\partial}{\partial t}\kALE(\widehat{\bfx},t)\textrm{ in }\kOmega\textrm{, and}\\
& \bfw(\bfx,t) = \widehat{\bfw}(\kALEinv(\bfx,t),t)\textrm{ for }\bfx\in\Omega(t).
\end{split}
\end{equation}
Employing the ALE framework results in replacing the classic (Eulerian) temporal derivative with its ALE counterpart, namely
\[
\frac{\partial}{\partial t}\bfv = \partALEt\bfv - [\bfw\cdot\nabla]\bfv,
\]
where $\partALEt$ denotes temporal derivative with respect to the ALE map, i.e. for $\phi\colon Q_T \to \bbmR$
\[
\partALEt\phi = \frac{\partial}{\partial t}\phi + \bfw\cdot\nabla\phi.
\]

\subsection{Discrete ALE framework}
Notation for a fully discretized equations in ALE framework can become quite messy and confusing, and sometimes even ambiguous with multiple possible interpretations. It is desirable that the notation can capture both discretization in space (with finite element method) and in time (typically with finite difference scheme). Furthermore, occasionally there is a need to use the same expression (e.g. function) on two different domains in time -- e.g. during the linearization step. Hence, the underlying configuration has to be clear from the notation (e.g. $\Omega_h^n$ or $\Omega_h^{n+1}$). With that in mind, we introduce certain rules on subscripts and superscripts which enable us to capture all of this information within a single expression. 

In this context, we distinguish three sets of expressions: the space--time domain, ALE related fields describing the evolution of geometry, and fields defined on the space--time domain. Within this notation, fields intrinsic only to physical configuration such as normal and ALE velocity fall within the third category. Details are given below.

\subsubsection{Discretization of space--time domain}
Time interval $[0,T]$ is partitioned in a finite number of segments: for $N\in\mathbb{N}$
\[
0=t_0<t_1<\dots<t_N=T \textrm{ and } \Delta t=t_{n+1}-t_n\;,n=0,\dots,N\;,
\]
where we assumed that the time step $\Delta t$ is uniform. Spatial domain $\Omega$ at time instant $t_n$ is denoted by $\Omega^n$ for short, i.e. $\Omega^n=\Omega(t_n)$. Spatial discretization of the domain is denoted by subscript $h$, i.e. $\Omega_h$ denotes spatial domain discretized in space. Naturally, $\Omega_h^n$ denotes discrete spatial domain at time instant $t_n$. Triangulations of $\Omega_h$, $\Omega_h^n$ and $\kOmega_h$ are denoted by $\calT_h$, $\calT_h^n$ and $\kcalT_h$, respectively. We often identify discrete domain $\Omega_h$ with its triangulation $\calT_h$.

\subsubsection{ALE related fields}
Discrete ALE map $\kALE$, its gradient $\kALEF=\widehat{\nabla}\kALE$ and Jacobian $\kjac=\det\kALEF$ may have up to two subscripts: $h$ is reserved for the spatial discretization, and $n$ denotes that field is evaluated at time instant $t_n$. {\it Hat} over the {\it nabla} in expression $\widehat{\nabla}$ denotes that derivatives are taken with respect to $\kbfx$ variable. For example,
\[
\begin{split}
&\kALE_{h,n}\in\sA_h(\kOmega)\;,\:\kALE_{h,n}=\kALE_{h}(\cdot,t_n)\;,\textrm{ is discretized in space and time},\\
&\kALE_{h}\in\sA_h(\kOmega)\;,\:\kALE_{h}=\kALE_{h}(\cdot,t)\;,\textrm{ is discretized in space and continuous in time.}
\end{split}
\]
Above, $\sA_h(\kOmega)$ denotes the finite element space to which ALE map belongs to. Note that, if $\kOmega_h$ is polygonal and we wish for the $\Omega_h(t)$ to be polygonal at each time instant, then $\sA(\kOmega)=[\fespaceP1(\kcalT_h)]^d$. $\fespaceP1(\kcalT_h)$ denotes the finite element space of piecewise linear functions.

In practice (during implementation), ALE map is typically constructed piecewise, i.e. per time segment $[t_n,t_{n+1}]$. The following notation comes in handy:
\[
\begin{split}
&\ALE_h^{[n+1,n]}\colon\Omega_h^{n+1}\to\bbmR^d\;,\: \ALE_h^{[n+1,n]}(\cdot,t) = \kALE_h(\cdot,t)\circ\kALE_{h,n+1}^{\;-1}\;,\:t\in[t_n,t_{n+1}],\\
&\Omega_h^{n+1}\ni\bfx^{n+1}\mapsto\bfx(t)\in\Omega_h\;,\:t\in[t_n,t_{n+1}].
\end{split}
\]
Both $\ALE_h^{[n+1,n]}$ and $\ALE_h^{[n,n+1]}$ are defined on the same time segment $[t_n,t_{n+1}]$ but the spatial domain of $\ALE_h^{[n,n+1]}$ is $\Omega_h^{n}$. Adding a second subscript denotes evaluation in time, i.e.
\[
\begin{split}
&\ALE_{h,n+1}^{[n,n+1]}(\cdot)=\ALE_h^{[n,n+1]}(\cdot,t_{n+1}).
\end{split}
\]
The idea behind $\ALE_h^{[n+1,n]}$ and $\ALE_h^{[n,n+1]}$ is to map $\Omega_h^{n+1}$ to $\Omega_h^n$ and {\it vice versa}, e.g.
\[
\ALE_{h,n+1}^{[n,n+1]}(\bfx^n)=\bfx^{n+1}\;,\:\ALE_{h,n}^{[n,n+1]}(\bfx^n)=\bfx^{n}.
\]
\subsubsection{General fields defined on space--time domain}
As a general rule, discretization in space is denoted by a subscript ($h$) and discretization in time by a superscript (typically $n$ or $n+1$). Let $\phi$ be an arbitrary scalar, vector or tensor field defined on $Q_T$. Then, since $\phi$ is defined on moving domain, when evaluating $\phi$ at some time $t$, one has to keep in mind the underlying domain $\Omega(t)$ on which $\phi$ is defined.

In this context, subscript $h$ denotes that $\phi_h$ belongs to some finite element space over $\Omega_h$, and a superscript $n+1$ denotes that $\phi_h^{n+1}$ is evaluated at time $t_{n+1}$ (and implicitly, on $\Omega_h^{n+1}$). Formally, $\phi_h^{n+1}\colon\Omega_h^{n+1}\to\bbmR^m$, $m\in\mathbb{N}$. It is sometimes necessary to {\it pull back} $\phi_h^{n+1}$ to $\Omega_h^n$ or {\it push forward} $\phi_h^n$ to $\Omega_h^{n+1}$ (e.g. during the linearization step). Hence, the second subscript is added to emphasize the domain where $\phi_h^{n+1}$ is defined on, i.e.
\[
\phi_{h,n}^{n+1}\colon\Omega_h^n\to\bbmR^k\;,\:
\phi_{h,n}^{n+1} = \phi_h^{n+1}\circ\ALE_{h,n}^{[n+1,n]}.
\]
For consistency, we allow $\phi_h^{n+1}=\phi_{h,n+1}^{n+1}$, but, by definition, this is implicitly implied.

Finally, for the cases where spatial domain is implied by contexts, e.g. if $\phi_h^n$ appears under the integral sign, the second (temporal) subscript may be dropped. For example, it is implied
\[
\int\limits_{\partial\Omega_h^{n+1}} \phi_h^n\bfn\dbfx = \int\limits_{\partial\Omega_h^{n+1}} \phi_{h,n+1}^n\bfn^{n+1}\dbfx.
\]

\subsubsection{Spatial derivatives on space--time domain}
Unless otherwise stated, the spatial derivative is always taken with respect to the domain of definition of the function. For example, since $\phi_h^{n+1}=\phi_{h,n+1}^{n+1}$ is defined on $\Omega_h^{n+1}$,
\[
\nabla\phi_h^{n+1}\textrm{ corresponds to }\frac{\partial}{\partial\bfx^{n+1}}\phi_h^{n+1}.
\]
By similar reasoning,
\[
\begin{split}
& \nabla\phi_{h,n}^{n+1}\textrm{ corresponds to }\frac{\partial}{\partial\bfx^{n}}\phi_{h,n}^{n+1}\;\textrm{, and}\\
& \Big( \nabla\phi_h^{n+1}\Big)_n\textrm{ corresponds to }\Big(\frac{\partial}{\partial\bfx^{n+1}}\phi_h^{n+1}\Big)\circ\ALE_{h,n}^{[n+1,n]}.
\end{split}
\]

\subsection{Discrete normal to the boundary}
Throughout this paper, by $\bfn$ we denote the exact normal to the boundary. Specially, for the case of discrete boundary $\partial\Omega_h$, $\bfn$ is defined almost everywhere. If $\partial\Omega_h$ is piecewise linear (polygonal mesh), then $\bfn$ is piecewise constant, i.e. belongs to space $\P0fespace(\partial\Omega_h)$ -- $\bfn$ is not defined in the mesh vertices. It is often of interest to define normal which is defined everywhere on the boundary $\partial\Omega_h$. We denote such normal by $\bfnu_h$. 

In practice, $\bfnu_h$ is reconstructed from the exact, almost everywhere defined, normal $\bfn$. A good and concise review on this manner is given in \cite{cenanovic20}. In context of finite elements for incompressible flow, an interesting approach for construction $\bfnu_h$ is given in \cite{engelman82}. Within this approach, Stokes formula for the integration by parts holds on the discrete level. $\bfnu_h$ is important in some approaches for the mesh velocity construction (see the following subsection).

\subsection{On the construction of the ALE velocity}
Construction of the ALE velocity is central for ALE framework. Although ALE velocity is derived from the ALE map in the description of the ALE framework, in practice the process is usually converse. ALE velocity is constructed first from the fluid velocity, and then ALE map is defined via ordinary differential equation (\ref{eq:ALEw}). The following compatibility condition has to be satisfied on continuous level between fluid and domain velocity: ALE and fluid velocities must have the same normal components, i.e.
\begin{equation}\label{eq:ALEw-n}
\bfw\cdot\bfn = \bfv\cdot\bfn\;,\:\textrm{ on }\partial\Omega\;,\:t\in[0,T],
\end{equation}
while the ALE velocity may be chosen arbitrary in the domain interior. Compatibility condition (\ref{eq:ALEw-n}) ensures the volume preservation on the continuous level. Arbitrary extension of the velocity into the interior is often realized as harmonic or linear elasticity extension (see \cite{formaggia99}). In practice, ALE velocity may be constructed solving the following problem: for $t>0$ and $\bfv_h$ fluid velocity,
\begin{equation}
\begin{split}
-\Delta\bfw_h & = 0\textrm{ in }\Omega_h,\\
\bfw_h & = \frac{\bfv_h\cdot\bfnu_h}{\bfk_h\cdot\bfnu_h}\bfk_h\textrm{ on }\partial\Omega_h,
\end{split}
\end{equation} 
where $\bfk_h\colon\partial\Omega_h\to\bbmR^d$ is some user defined vector and depends on the physics of the problem. The trivial choice of $\bfk_h$ is $\bfk_h=\bfv_h$ which, formally, yields $\bfw_h=\bfv_h$ on $\partial\Omega_h$. This is a very common choice for simulating the kinematics of droplets with FEM (see, e.g., \cite{venkatesan16,ganesan09,willassen14}). 

In this paper, we use implicit discretization in time for the domain evolution. More specifically, given $\bfv_h^{n+1}$ and $\bfw_h^{n+1}$ such that
\begin{equation}\label{eq:discre-mesh-vel}
\begin{split}
-\Delta\bfw_h^{n+1} & = 0\textrm{ in }\Omega_h^{n+1},\\
\bfw_h^{n+1} & = \frac{\bfv_h^{n+1}\cdot\bfnu_h^{n+1}}{\bfk_h^{n+1}\cdot\bfnu_h^{n+1}}\bfk_h^{n+1}\textrm{ on }\partial\Omega_h^{n+1},
\end{split}
\end{equation}
$\Omega_h^{n+1}$ is obtained from $\Omega_h^n$ according to the following updating scheme:
\begin{equation}\label{id:discrete-ALE-map}
\Omega_h^{n+1}\ni\bfx^{n+1} = \bfx^n + \Delta t\;\bfw_{h,n}^{n+1}(\bfx^n)
\end{equation}
Clearly, an iterative procedure has to be employed for such approach -- $\Omega_h^{n+1}$ is obtained by updating $\Omega_h^n$ with velocity $\bfw_{h,n}^{n+1}$ evaluated on $\Omega_h^{n+1}$. Details on linearization and iterative updating strategy are given in Section~\ref{sec:numerical-validation}. 

Notice that for the case of polygonal mesh $\calT_h$, $\bfw_h$ has to be chosen from the space $[\fespaceP1(\calT_h)]^d$ in order for $\ALE_h^{[n,n+1]}$ to be in $[\fespaceP1(\calT_h)]^d$.
\begin{remark}\label{rem:div-vh}
It is important to note that discrete ALE map $\ALE_h^{[n,n+1]}$ constructed via identity (\ref{id:discrete-ALE-map}) does not necessarily preserves the volume of $\Omega_h$. Indeed, 
\begin{equation}
\div\bfv_h^{n+1}=0\textrm{ in }\Omega_h^{n+1}\textrm{ does not imply }\div\bfv_{h,n}^{n+1}=0\textrm{ in }\Omega_h^{n},
\end{equation}
where divergence is taken with respect to the domain of definition (i.e. w.r.t. $\bfx^{n+1}$ on $\Omega_h^{n+1}$ and w.r.t. $\bfx^{n}$ on $\Omega_h^{n}$). In practice, this issue is often unattended and circumvented by choosing time step $\Delta t$ sufficiently small so that there is no significant volume gain or loss. In a special case where mesh velocity $\bfw_{h,n}^{n+1}$ has a single component, exact mass/volume conservation may be obtained on the discrete level under some minor, but reasonable, assumptions. See \cite{gerbeau09} for details.
\end{remark}
\begin{remark}\label{rem:velocity-consistency}
By similar reasoning as in Remark~\ref{rem:div-vh}, it can be seen that, in general, condition (\ref{eq:discre-mesh-vel})$_2$ and mesh updating scheme (\ref{id:discrete-ALE-map}) do not necessary imply
\begin{equation}\label{eq:vel-normals-discrete}
\bfw_h^{n+1}\cdot\bfn = \bfv_h^{n+1}\cdot\bfn \textrm{ on }\partial\Omega_h(t)\textrm{, for }t\in(t_n,t_{n+1}).
\end{equation}
Identity (\ref{eq:vel-normals-discrete}) also depend on the definition of discrete normal $\bfnu_h$ in (\ref{eq:discre-mesh-vel})$_2$. Hence, in general, the following desirable identity derived from Reynolds transport theorem cannot be employed straightforwardly on the discrete level: for $t\in(t_n,t_{n+1})$
\begin{equation}\label{eq:discrete-Rey-ex}
\frac{\td}{\dt}\int\limits_{\Omega_h(t)}\dbfx = \int\limits_{\partial\Omega_h(t)}\bfw_h^{n+1}\cdot\bfn\dS = \int\limits_{\partial\Omega_h(t)}\bfv_h^{n+1}\cdot\bfn\dS.
\end{equation}
The last equality in the above equation is only true if (\ref{eq:vel-normals-discrete}) holds. Identity (\ref{eq:discrete-Rey-ex}) is crucial in derivation of energy estimates on a discrete level.
\end{remark}

\section{Variational formulation and energy balance}\label{sec:varf}
In this section the variational formulation of (dimensionless) system (\ref{sys:NSnondim})--(\ref{sys:NSnondimbc}) is derived and studied. Specially, we investigate in depth the energy balance of system (\ref{sys:NSnondim})--(\ref{sys:NSnondimbc}), which we later use as a guideline to derive the energy stable discretization.
\subsection{Variational formulation in ALE framework}
Employing identity (\ref{id:ale}), we introduce function spaces necessary for the derivation of the weak formulation:
\[
\begin{split}
& \sV(\Omega) = \left\{ \bfvarphi\colon Q_T\to\bbmR^d \mid \bfvarphi(\bfx,t)=\kbfvarphi(\kALEinv(\bfx,t))\;,\:\kbfvarphi\in\skV(\kOmega) \right\} \textrm{ where}\\
& \skV(\kOmega) = \sH^1_{\bfn,\Gamma}(\kOmega;\bbmR^d)=\left\{ \kbfvarphi\in\sH^1(\kOmega;\bbmR^d)\mid \kbfvarphi\cdot\widehat{\bfn_\Gamma}=0 \textrm{ on }\widehat{\Gamma} \right\}\;,
\end{split}
\]
and
\[
\sQ(\Omega) = \{\chi\colon Q_T\to\bbmR\mid \chi(\bfx,t)=\kchi(\kALEinv(\bfx,t)) \;,\: \kchi\in\sL^2(\kOmega)\}.
\]
Note that Dirichlet boundary condition (\ref{sys:NSnondimbc})$_2$ is embedded into the function space $\sV(\Omega)$ and that the test functions in $\sV(\Omega)$ are time independent within ALE framework, i.e.
\begin{equation}\label{id:time-indep}
\partALEt \bfvarphi = 0\;,\:\bfvarphi\in\sV(\Omega).
\end{equation} 

Weak formulation of system (\ref{sys:NSnondim})--(\ref{sys:NSnondimbc}) is obtained in a standard way. We multiply system (\ref{sys:NSnondim}) with test functions $\bfvarphi\in\sV(\Omega)$ and $\chi\in\sQ(\Omega)$, and integrate by parts, while employing the boundary conditions (\ref{sys:NSnondimbc}) in the process. The variational formulation then reads: 
\begin{equation}\label{sys:vf-continuous}
\begin{split}
&\textrm{find }\bfv\in\sL^2(0,T;\sH^1_{\bfn,\Gamma}(\Omega))\;,\:p\in\sL^2(0,T;\sL^2(\Omega)) \textrm{ and }\kALE\in\sH^1(0,T;\sW^{1,\infty}(\kOmega;\bbmR^d))\\
&\textrm{such that } \forall\:(\bfvarphi,\chi)\in\sV(\Omega)\times\sQ(\Omega)\\
& \frac{\td}{\dt}\int\limits_\Omega \bfvarphi\cdot\bfv\dbfx - \int\limits_\Omega \bfvarphi\cdot\bfv\div\bfw\dbfx + \int\limits_\Omega \bfvarphi\cdot[(\bfv-\bfw)\cdot\nabla]\bfv\dbfx \\
&\hspace{0.2\linewidth} + \int\limits_\Omega \frac{1}{2}\bbmD(\bfvarphi)\colon\frac{1}{\noRe}\bbmD(\bfv)\dbfx -\int\limits_\Omega p\div\bfvarphi\dbfx\\
& \hspace{0.2\linewidth} +\int\limits_\Gamma\bfvarphi\cdot\varsigma(\bfv-\bfu)\dS - \int\limits_\eta \bfvarphi\cdot\frac{1}{\noCa\noRe}\cos\theta_s\bfm_{\partial\Gamma}\ds + \int\limits_\Sigma \frac{\div_\Sigma\bfvarphi}{\noCa\noRe}\dS\\
& \hspace{0.2\linewidth} + \int\limits_\Omega\bfvarphi\cdot\nabla\Phi\dbfx =0,\\
& \int\limits_\Omega \chi\div\bfv\dbfx =0,
\end{split}
\end{equation}
where $\bfw=\partial_t\kALE$ satisfies
\[
\bfw\cdot\bfn = \bfv\cdot\bfn\textrm{ on }\partial\Omega(t)\;,\:\Delta\bfw=0\textrm{ in }\Omega(t).
\]
Specially, ALE time derivative is extracted in front of the integral sign by employing Reynolds transport theorem in the process and using the property (\ref{id:time-indep}).

The following notation is introduced for a compact notation:
\begin{equation}
\begin{split}
& m(\bfvarphi,\bfv) = \int\limits_\Omega\bfvarphi\cdot\bfv\dbfx \;,\: d(\bfvarphi,\bfv;\bfu) = \int\limits_\Omega \bfvarphi\cdot\bfv\div\bfu\dbfx \;,\: c(\bfvarphi,\bfv;\bfu) = \int\limits_\Omega \bfvarphi\cdot[\bfu\cdot\nabla]\bfv\dbfx \;,\\
& a(\bfvarphi,\bfv) = \int\limits_\Omega \frac{1}{2}\bbmD(\bfvarphi)\colon\frac{1}{\noRe}\bbmD(\bfv)\dbfx \;,\: b(\bfvarphi,p) = \int\limits_\Omega p\div\bfvarphi \dbfx \;,\: r(\bfvarphi,\bfv) = \int\limits_\Gamma \bfvarphi\cdot\varsigma\bfv\dbfx \;,\\
& f_\textrm{cl}(\bfvarphi) = \int\limits_\eta\frac{\cos\theta_s}{\noCa\noRe}\bfm_{\partial\Gamma}\cdot\bfvarphi \dS \;,\: f_\textrm{st}(\bfvarphi) = \int\limits_\Sigma \frac{\div_\Sigma\bfvarphi}{\noCa\noRe}\dS \;,\: f_\textrm{g}(\bfvarphi) = \int\limits_\Omega \bfvarphi\cdot\nabla\Phi\dbfx \;,\\
& f_\textrm{gs}(\bfvarphi) = \int\limits_{\Sigma}\Phi\bfvarphi\cdot\bfn\dS \;.
\end{split}
\end{equation}

\subsection{Energy balance}
This subsection is devoted to the derivation of the energy balance of system (\ref{sys:NSnondim}) subjected to boundary conditions (\ref{sys:NSnondimbc}). We start by introducing the following notation: kinetic energy $E_k(t)$, viscous power $P_v(t)$, friction power $P_{fr}(t)$, free surface energy $E_{fs}(t)$, wetting energy $E_w(t)$ and potential energy $E_p(t)$, defined by
\begin{equation}\label{eq:energies-notation}
\begin{split}
& E_k(t)=\int\limits_\Omega\frac{1}{2}\vert\bfv\vert^2\dbfx\;,\: P_v(t)=\int\limits_\Omega\frac{1}{2\noRe}\vert\bbmD(\bfv)\vert^2\dbfx\;,\: P_{fr}(t)=\int\limits_\Gamma\varsigma(\bfv-\bfu)\cdot\bfv\dS\;,\\
& E_{fs}(t)=\int\limits_\Sigma\frac{1}{\noCa\noRe}\dS\;,\: E_w(t)=\int\limits_\Gamma\frac{-\cos\theta_s}{\noCa\noRe}\dS\;,\textrm{ and}\: E_p(t)=\int\limits_\Omega\Phi\dbfx.
\end{split}
\end{equation}
The main result of this subsection is the following proposition.
\begin{prop}\label{prop:ceb}
System (\ref{sys:NSnondim}) subjected to boundary conditions (\ref{sys:NSnondimbc}) satisfies the following energy equality:
\begin{equation}\label{eq:energy-balance}
\frac{\td}{\dt} E_k(t) + P_v(t) + \frac{\td}{\dt}E_{fs}(t) + \frac{\td}{\dt}E_{w}(t) + P_{fr}(t) + \frac{\td}{\dt}E_{p}(t) = 0.
\end{equation}
\end{prop}
Derivation of the energy balance stated in Proposition~\ref{prop:ceb} is fairly standard and straightforward. Moreover, it has already been derived in \cite{gerbeau09} and used as a base for a special case of problem (\ref{sys:NSnondim})--(\ref{sys:NSnondimbc}) in which ALE velocity $\bfw$ has only one non--vanishing component -- they studied the case where $\bfw=[0,0,w_z]^T$. Achieving energy stability on the discrete level for the case of arbitrary $\bfw$ becomes significantly more technical and demanding. Some properties, such as SCL, that are trivial on a continuous level, become hard to satisfy during the space and/or temporal discretization. This is the reason why we, at least concisely, summarize the derivation of energy balance equation (\ref{eq:energy-balance}) in the next few lemmas. We recall some of these steps later during the transition to discrete level.

\begin{lem}\label{lem:kinetic}
Change in kinetic energy satisfies
\begin{equation}\label{id:kinetic}
\frac{\td}{\dt}E_k(t) = \int\limits_\Omega\bfv\cdot\left( \partALEt\bfv +[(\bfv-\bfw)\cdot\nabla]\bfv \right)\dbfx.
\end{equation}
\end{lem}
\begin{proof}
Employing Reynolds transport theorem in the process, it is straightforward to obtain
\begin{equation}\label{id:transient}
\int\limits_\Omega\bfv\cdot\partALEt\bfv\dbfx = \int\limits_\Omega\partALEt\left( \frac{1}{2}\vert\bfv\vert^2\right)\dbfx = \frac{\td}{\dt}\int\limits_\Omega\frac{1}{2}\vert\bfv\vert^2\dbfx - \int\limits_\Omega \frac{1}{2}\vert\bfv\vert^2\div\bfw\dbfx.
\end{equation}
Furthermore, integration by parts in convective term gives us
\begin{equation}\label{id:convective}
\begin{split}
&\int\limits_\Omega\bfv\cdot[(\bfv-\bfw)\cdot\nabla]\bfv \dbfx = \int\limits_\Omega (\bfv-\bfw)\cdot\nabla\left( \frac{1}{2}\vert\bfv\vert^2 \right)\dbfx\\
&\hspace{0.1\linewidth} = \int\limits_{\partial\Omega}\frac{1}{2}\vert\bfv\vert^2(\bfv-\bfw)\cdot\bfn\dS - \int\limits_\Omega\frac{1}{2}\vert\bfv\vert^2\div\bfv\dbfx + \int\limits_\Omega\frac{1}{2}\vert\bfv\vert^2\div\bfw\dbfx.
\end{split}
\end{equation}
Summing up (\ref{id:transient}) and (\ref{id:convective}) while using (\ref{sys:NSnondim})$_2$ and (\ref{eq:ALEw}), we obtain (\ref{id:kinetic}).
\end{proof}

\begin{lem}\label{lem:fse}
Change in free surface energy satisfies
\begin{equation}\label{eq:fs-energy}
\frac{\td}{\dt}E_{fs}(t)=\int\limits_{\Sigma}\frac{\div_\Sigma\bfv}{\noCa\noRe}\dS 
\end{equation}
\end{lem}
\begin{proof}
Employing the generalization of Reynolds transport theorem for manifolds of co--dimension $1$ (\ref{thm:generalizedRTT}), while using identity (\ref{eq:ALEw-n}) and decomposition $\bfv=\PS\bfv + (\bfv\cdot\bfn)\bfn$, it is obtained
\begin{equation}
\int\limits_{\Sigma}\frac{\div_\Sigma\bfv}{\noCa\noRe}\dS = \int\limits_{\Sigma}\frac{\div_\Sigma\bfw}{\noCa\noRe}\dS = \frac{\td}{\dt} \int\limits_{\Sigma}\frac{1}{\noCa\noRe}\dS.
\end{equation}
\end{proof}

\begin{lem}\label{lem:wett}
Change in wetting energy satisfies
\begin{equation}\label{eq:wett-energy}
\frac{\td}{\dt}E_w(t) = \int\limits_{\partial\Gamma} \frac{-\cos\theta_s}{\noCa\noRe}\bfm_{\partial\Gamma}\cdot\bfv\ds.
\end{equation}
\end{lem}
\begin{proof}
Similarly as in Lemma~\ref{lem:fse}, we employ Reynolds transport theorem for manifolds of co--dimension $1$ while using $\theta_s=\theta_s(\bfx_\Gamma)$ and $\bfw\cdot\bfn=0$ on $\Gamma$ to obtain
\begin{equation}
\int\limits_{\partial\Gamma}\frac{\cos\theta_s}{\noCa\noRe}\bfm_{\partial\Gamma}\cdot\bfv\ds  = \int\limits_\Gamma\div_\Gamma\left(\frac{\cos\theta_s}{\noCa\noRe}\bfw\right)\dS = \frac{\td}{\dt}\int\limits_\Gamma\frac{\cos\theta_s}{\noCa\noRe}\dS.
\end{equation}
\end{proof}

\begin{lem}\label{lem:potential}
Change in potential energy satisfies
\begin{equation}\label{eq:pot-energy}
\frac{\td}{\dt}E_p(t) = \int\limits_\Omega \bfv\cdot\nabla\Phi\dbfx.
\end{equation}
\end{lem}
\begin{proof}
Employing Reynolds transport theorem and using $\Phi=\Phi(\bfx)$ and $\bfw\cdot\bfn=\bfv\cdot\bfn$ on $\partial\Omega$, it is straightforward to obtain
\begin{equation}
\begin{split}
& \int\limits_\Omega \bfv\cdot\nabla\Phi\dbfx = \int\limits_\Omega\Big( \div(\Phi\bfv) - \Phi\div\bfv \Big)\dbfx = \int\limits_{\partial\Omega}\Phi\bfv\cdot\bfn\dS - \int\limits_\Omega \Phi\div\bfv \dbfx\\
&\hspace{0.2\linewidth} = \int\limits_{\partial\Omega}\Phi\bfw\cdot\bfn\dS = \frac{\td}{\dt} \int\limits_\Omega \Phi\dbfx.
\end{split}
\end{equation}
\end{proof}

Finally, we have all the ingredients to prove Proposition~\ref{prop:ceb}.
\begin{proof}Proof of Proposition~\ref{prop:ceb}. We multiply system (\ref{sys:NSnondim}) by $(\bfv,p)$ and integrate by parts employing boundary conditions (\ref{sys:NSnondimbc}) in the process. Using the identities derived in Lemmas~\ref{lem:kinetic}--\ref{lem:potential}, the identity (\ref{eq:energy-balance}) follows by straightforward calculation.
\end{proof}

\section{FEM formulation and discrete energy stability}\label{sec:energy-stable-scheme}

\subsection{Spatial discretization}
Discretization in space is performed in a standard way characteristic for finite element method. First, domain $\Omega$ is replaced by its discrete couterpart $\Omega_h$ on which mesh $\calT_h$ is constructed. Then, function spaces $\sV(\Omega)$ and $\sQ(\Omega)$ are replaced by their finite element counterparts $\sV_h(\Omega_h)$ and $\sQ_h(\Omega_h)$ built over triangulation $\calT_h$. 

Discretization $\Omega\mapsto\Omega_h$ essentially means that the boundary $\partial\Omega=\Sigma\cup\Gamma$ is replaced by its discrete counterpart. Most commonly for FEM, the approximation of geometry is with piecewise linear functions. Higher order polynomial approximation (characteristic for isoparametric FEM) is also often applied. This paper focuses on the case of piecewise linear triangular discretization of the boundary $\partial\Omega$ and tetrahedral meshing of the domain interior. However, most of the analysis performed in this paper can be extended straightforwardly to the cases of higher order polynomial approximation of the geometry.

Space discretization of the boundary plays the central role in imposing the free surface boundary condition (\ref{sys:NSnondimbc})$_1$. The geometry of the boundary is embedded into this boundary condition through the curvature, which is rewritten in terms of surface Laplacian (\ref{eq:curvature}). This imposes a constraint on the order of the basis functions involved in the curvature evaluation term (for details see \cite{ivancic20}). We describe this issues concisely in the following paragraph.

Let us consider discretized counterpart of the free surface term in variational formulation (\ref{sys:vf-continuous}), namely
\begin{equation}\label{eq:discrete-curv}
f_{\textrm{fs},h}(\bfvarphi_h) = \int\limits_{\Sigma_h}\frac{\div_{\Sigma_h}\bfvarphi_h}{\noCa\noRe}\dS.
\end{equation}
It has been shown in \cite{ivancic20} that the polynomial order of $\bfvarphi_h$ restricted on $\Sigma_h$ cannot be higher than the order of basis function used for construction of $\Sigma_h$. Otherwise, spurious velocities may be generated in the neighborhood of the discrete boundary $\Sigma_h$. This condition restricts the choice of finite element space $\sV_h(\Omega_h)$ to spaces whose functions restricted to $\Sigma_h$ are of the same order as the basis functions for construction of $\Sigma_h$. Specially, if $\Sigma_h$ is polygonal, then restriction of $\bfvarphi_h$ to ${\Sigma_h}$ has to belong to space $[\fespaceP1(\Sigma_h)]^d$. This can be circumvented by decoupling the curvature evaluation from the main problem or by employing some stabilization techniques. Details can be found in \cite{ivancic20}.
\begin{remark}
It has been shown in \cite{hansbo15} that FEM for computing mean curvature vector on polygonal meshes via Laplace--Beltrami operator may fail due to instabilities. This issue is characteristic for 3D case only, i.e. for polygonal surfaces embedded in $\bbmR^3$.  They proposed simple but efficient stabilization procedure which consists of modifying the weak formulation for mean curvature vector. Such instabilities in curvature vector, in our experience, usually appear in saddle vertices of the mesh.

For the purpose of this work, we use directly formulation (\ref{eq:discrete-curv}) without stabilization, and experience no issues, as reported in Section~\ref{sec:numerical-validation}. However, curvature stabilization techniques should be considered in general. Curvature stabilization technique proposed in \cite{hansbo15} is straightforward to implement within the framework proposed in this work.
\end{remark}

\subsection{Choice of the finite element spaces}
It is well known that spaces for velocity and pressure cannot be chosen arbitrary but have to satisfy the so--called {\it Lady\v{z}enskaya--Babu\v{s}ka--Brezzi} (LBB) condition (see \cite{boffi13}). This is due to the {\it saddle--point} nature of the incompressible Navier--Stokes equations. The simplest choice for $\sV_h(\Omega_h)\times\sQ_h(\Omega_h)$ which satisfies LBB condition and condition imposed by discrete Laplace--Beltrami operator for curvature evaluation discussed above, is $[\fespaceminiP1(\Omega_h)]^d\times\fespaceP1(\Omega_h)$. Here, $\fespaceminiP1$ denotes the {\it mini}--elements, i.e. $\fespaceP1$ elements enriched with the cubic bubble function.

\subsection{Energy balance of the semi--discrete system}
In this subsection, we recall possible numerical sinks/sources of artificial energy coming solely due to the spatial discretization.

It is straightforward to derive the following property of the trilinear form $c(\dots)$:
\begin{equation}\label{property:trilinear}
c(\bfvarphi,\bfv;\bfu) = -c(\bfv,\bfvarphi;\bfu) + \int\limits_{\partial\Omega}(\bfvarphi\cdot\bfv)\bfu\cdot\bfn\dS - \int\limits_{\Omega} \bfvarphi\cdot\bfv\div\bfu\dbfx.
\end{equation}
Taking $\bfvarphi=\bfv$ and $\bfu=\bfv-\bfw$, and using $\div\bfv=0$ and $\bfw\cdot\bfn=\bfv\cdot\bfn$ in (\ref{property:trilinear}), it is easily obtained
\begin{equation}\label{property:trilinear-energy}
c(\bfv,\bfv;\bfv-\bfw) = \int\limits_\Omega \frac{1}{2}\vert\bfv\vert^2\div\bfw\dbfx.
\end{equation}
Property (\ref{property:trilinear-energy}) does not hold in general on the discrete level. Indeed, the assumption 
\[
\bfw_h\cdot\bfn=\bfv_h\cdot\bfn\textrm{ on }\Sigma_h
\]
holds in general only on semi--discrete level; temporal discretization may ruin this property (see Remark~\ref{rem:velocity-consistency} for some details). Furthermore, $\div\bfv_h=0$ holds only in variational sense, i.e.
\[
\int\limits_{\Omega_h}\chi_h\div\bfv_h\dbfx=0\;,\:\chi_h\in\sQ_h(\Omega_h)\;,\: \textrm{does not imply} \: \int\limits_{\Omega_h}\frac{1}{2}\vert\bfv_h\vert^2\div\bfv_h\dbfx=0
\]
since $\vert\bfv_h\vert^2\notin\sQ_h(\Omega_h)$ in general. Note that (\ref{property:trilinear-energy}) is a modification of the classic {\it skew symmetry} property of the trilinear form in Navier--Stokes equations posed on stationary domains. Consistent stabilization of the trilinear form is thus used in discrete variational formulation, namely
\begin{equation}
\begin{split}
& c_h(\bfvarphi_h,\bfv_h;\bfv_h-\bfw_h)\mapsto \tilde{c}_h(\bfvarphi_h,\bfv_h;\bfv_h-\bfw_h),\\
& \tilde{c}_h(\bfvarphi_h,\bfv_h;\bfv_h-\bfw_h) = \frac{1}{2}\Big( c_h(\bfvarphi_h,\bfv_h;\bfv_h-\bfw_h) - c_h(\bfv_h,\bfvarphi_h;\bfv_h-\bfw_h) \Big)\\
&\hspace{0.25\linewidth}+ \frac{1}{2}d_h(\bfvarphi_h,\bfv_h;\bfw_h).
\end{split}
\end{equation}
Trilinear form $\tilde{c}_h(\bfvarphi_h,\bfv_h;\bfv_h-\bfw_h)$ is consistent in sense that on the continuous level
\[
\tilde{c}(\bfvarphi,\bfv;\bfv-\bfw) = c(\bfvarphi,\bfv;\bfv-\bfw).
\]
We summarize the above discussion in the following lemma:
\begin{lem}
Consistent stabilization of discrete trilinear form $c_h(\bfvarphi_h,\bfv_h;\bfv_h-\bfw_h)$ denoted by $\tilde{c}_h(\bfvarphi_h,\bfv_h;\bfv_h-\bfw_h)$ and defined by
\begin{equation}
\begin{split}
\tilde{c}_h(\bfvarphi_h,\bfv_h;\bfv_h-\bfw_h) = & \frac{1}{2}\Big( c_h(\bfvarphi_h,\bfv_h;\bfv_h-\bfw_h) - c_h(\bfv_h,\bfvarphi_h;\bfv_h-\bfw_h) \Big) \\
& + \frac{1}{2}d_h(\bfvarphi_h,\bfv_h;\bfw_h),
\end{split}
\end{equation}
satisfies the discrete counterpart of property (\ref{property:trilinear-energy}), i.e.
\begin{equation}
\tilde{c}_h(\bfv_h,\bfv_h;\bfv_h-\bfw_h) = \int\limits_{\Omega_h}\frac{1}{2}\vert\bfv_h\vert^2\div\bfw_h\dbfx.
\end{equation}
\end{lem}

Artificial energy sinks/sources may appear in transition to the discrete level during derivation of potential, wetting and free surface energies in Lemmas~\ref{lem:fse}--\ref{lem:potential}. The semi--discrete counterparts of aforementioned lemmas are derived for the general case.
\begin{lem}
Change in semi--discrete potential energy satisfies
\begin{equation}
\int\limits_{\Omega_h}\bfv_h\cdot\nabla\Phi_h\dbfx = \frac{\td}{\dt}E_{p,h}(t) - \int\limits_{\Omega_h}\Phi_h\div\bfv_h\dbfx + \int\limits_{\partial\Omega_h}\Phi_h(\bfv_h-\bfw_h)\cdot\bfn\dS.
\end{equation}
Specially, if $\Phi_h\in \sQ_h(\Omega_h)$ and $\bfw_h\cdot\bfn=\bfv_h\cdot\bfn$ on $\Sigma_h$, then the discrete counterpart of (\ref{eq:pot-energy}) is satisfied, i.e.
\begin{equation}
\frac{\td}{\dt}E_{p,h}(t) = \int\limits_{\Omega_h}\bfv_h\cdot\nabla\Phi_h\dbfx.
\end{equation}
\end{lem}
\begin{proof}
By straightforward calculation, it follows
\[
\begin{split}
\int\limits_{\Omega_h}\bfv_h\cdot\nabla\Phi_h\dbfx & = \int\limits_{\Omega_h}\Big( \div(\Phi_h\bfv_h) - \Phi_h\div\bfv_h \Big)\dbfx\\
& = \int\limits_{\partial\Omega_h}\Big( \Phi_h\bfv_h\cdot\bfn \pm  \Phi_h\bfw_h\cdot\bfn\Big)\dS - \int\limits_{\Omega_h}\Phi_h\div\bfv_h\dbfx\\
& = \frac{\td}{\dt}\int\limits_{\Omega_h}\Phi_h\dbfx - \int\limits_{\Omega_h}\Phi_h\div\bfv_h\dbfx + \int\limits_{\partial\Omega_h}\Big( \Phi_h(\bfv_h-\bfw_h)\cdot\bfn\Big)\dS,
\end{split}
\]
where Reynolds transport theorem was employed in the process.
\end{proof}

\begin{lem}\label{lem:semidiscr-wett}
Change in semi--discrete wetting energy satisfies
\begin{equation}
\int\limits_{\eta_h}\frac{-\cos\theta_s}{\noCa\noRe}\bfm_{\partial\Gamma_h}\cdot\bfv_h\ds = \frac{\td}{\dt}E_{w,h}(t) + \int\limits_{\partial\Gamma_h}\div_{\Gamma_h}\left[\frac{-\cos\theta_s}{\noCa\noRe}\big(\PGh\bfv_h-\bfw_h\big)\right]\dS.
\end{equation}
Specially, if $\bfw_h\cdot\bfn=\bfv_h\cdot\bfn$ on $\partial\Omega_h$ and $\bfv_h\cdot\bfn=0$ on $\Gamma_h$, then the discrete counterpart of (\ref{eq:wett-energy}) is satisfied, i.e.
\begin{equation}
\frac{\td}{\dt}E_{w,h}(t) = \int\limits_{\eta_h}\frac{-\cos\theta_s}{\noCa\noRe}\bfm_{\partial\Gamma_h}\cdot\bfv_h\dS.
\end{equation}
\end{lem}
\begin{proof}
Employing the generalization of Reynolds transport theorem for manifolds of co--dimension $1$ while using $\theta_s=\theta_s(\bfx_\Gamma)$, by straightforward calculation it is obtained
\begin{equation}
\begin{split}
&\int\limits_{\eta_h}\frac{-\cos\theta_s}{\noCa\noRe}\bfm_{\partial\Gamma_h}\cdot\bfv_h\ds  = \int\limits_{\Gamma_h}\div_{\Gamma_h}\left[\frac{-\cos\theta_s}{\noCa\noRe}\big(\PGh\bfv_h \pm \bfw_h\big) \right]\dS\\
 &\hspace{0.2\linewidth} =\frac{\td}{\dt}\int\limits_{\Gamma_h}\frac{-\cos\theta_s}{\noCa\noRe}\dS + \int\limits_{\Gamma_h}\div_{\Gamma_h}\left[\frac{-\cos\theta_s}{\noCa\noRe}\big(\PGh\bfv_h-\bfw_h\big)\right]\dS.
\end{split}
\end{equation}
\end{proof}

\begin{lem}\label{lem:semidiscr-fs}
Change in semi--discrete free surface energy satisfies
\begin{equation}
\int\limits_{\Sigma_h}\frac{\div_{\Sigma_h}\bfv_h}{\noCa\noRe}\dS = \frac{\td}{\dt} E_{fs,h}(t) + \int\limits_{\Sigma_h}\frac{\div_{\Sigma_h}(\bfv_h-\bfw_h)}{\noCa\noRe}\dS.
\end{equation} 
Specially, if $\bfw_h\cdot\bfn=\bfv_h\cdot\bfn$ on $\Sigma_h$, then the discrete counterpart of (\ref{eq:fs-energy}) is satisfied, i.e.
\begin{equation}
\frac{\td}{\dt} E_{fs,h}(t) = \int\limits_{\Sigma_h}\frac{\div_{\Sigma_h}\bfv_h}{\noCa\noRe}\dS.
\end{equation}
\end{lem}
\begin{proof}
Proof is analogous to that of Lemma~\ref{lem:semidiscr-wett} and thus omitted.
\end{proof}

\subsection{Temporal discretization}
Semi--discrete counterpart (discrete in space) of variational formulation (\ref{sys:vf-continuous}) reads
\begin{equation}\label{sys:vf-semidiscrete}
\begin{split}
&\textrm{find }\bfv_h\in\sL^2(0,T;\sH^1_{\bfn,\Gamma_h}(\Omega_h))\textrm{ and }p_h\in\sL^2(0,T;\sL^2(\Omega_h))\textrm{ such that}\\
& \forall\:(\bfvarphi_h,\chi_h)\in\sV_h(\Omega_h)\times\sQ(\Omega_h)\\
& 0 = \frac{\td}{\dt}m_h(\bfvarphi_h,\bfv_h) - \frac{1}{2}d_h(\bfv_h,\bfvarphi_h;\bfw_h) \\
&\hspace{0.05\linewidth} + \frac{1}{2}c_h(\bfvarphi_h,\bfv_h;\bfv_h-\bfw_h) - \frac{1}{2}c_h(\bfv_h,\bfvarphi_h;\bfv_h-\bfw_h)\\
&\hspace{0.05\linewidth}+ a_h(\bfvarphi_h,\bfv_h) - b_h(\bfvarphi_h,p_h)\\
&\hspace{0.05\linewidth}+ r_h(\bfvarphi_h,\bfv_h) - r_h(\bfvarphi_h,\bfw_h) - f_{\textrm{cl},h}(\bfvarphi_h) + f_{\textrm{st},h}(\bfvarphi_h) + f_{\textrm{gs},h}(\bfvarphi)
\end{split}
\end{equation}  
In the above semi--discrete formulation we replaced the trilinear form $c_h(\dots)$ by its stabilized counterpart $\tilde{c}_h(\dots)$. Furthermore, we replaced the forcing term $f_{\textrm{g},h}(\dots)$ by $f_{\textrm{gs},h}(\dots)$. This last trick can be employed for an arbitrary conservative force. It proved to be very useful for energy stability after temporal discretization. Replacing $f_{\textrm{g},h}(\dots)$ by $f_{\textrm{gs},h}(\dots)$ corresponds to replacing total with dynamic pressure so that gravity appears only in the boundary conditions:
\begin{equation}
\begin{split}
& p_\textrm{dyn}  = p - \frac{1}{\noFr^2}\bfk\cdot\bfx,\\
& -p_\textrm{dyn}\bbmI + \frac{1}{\noRe}\bbmD(\bfv) = \frac{1}{\noCa\noRe}\Delta_\Sigma\bfx_\Sigma - \frac{1}{\noFr^2}(\bfk\cdot\bfx)\bfn\;,\textrm{ on }\Sigma\;,
\end{split}
\end{equation}
where $p_\textrm{dyn}$ denotes the dynamic pressure. In the rest of the paper, we simply write $p$ keeping in mind that it is actually the dynamic pressure.

For temporal discretization a variation of implicit Euler method is employed. Integrating (\ref{sys:vf-semidiscrete}) from $t_n$ to $t_{n+1}$, we obtain
\begin{equation}\label{sys:vf-discrete}
\begin{split}
0 =\; & m_h^{n+1}(\bfvarphi_h,\bfv_h^{n+1}) - m_h^{n}(\bfvarphi_h,\bfv_h^{n}) - \intInn1 \frac{1}{2}d_h(\bfv_h^{n+1},\bfvarphi_h;\bfw_h^{n+1})\\
&\hspace{0.025\linewidth} + \Delta t\;\frac{1}{2}c_h^{n+1}(\bfvarphi_h,\bfv_h^{n+1};\bfv_h^{n+1}-\bfw_h^{n+1}) - \Delta t\;\frac{1}{2}c_h^{n+1}(\bfv_h^{n+1},\bfvarphi_h;\bfv_h^{n+1}-\bfw_h^{n+1})\\
&\hspace{0.025\linewidth} + \Delta t\; a_h^{n+1}(\bfvarphi_h,\bfv_h^{n+1}) - \Delta t\; b_h^{n+1}(\bfvarphi_h,p_h^{n+1})\\
&\hspace{0.025\linewidth} + \Delta t\; r_h^{n+1}(\bfvarphi_h,\bfv_h^{n+1}) - \Delta t\; r_h^{n+1}(\bfvarphi_h,\bfw_h^{n+1}) - \intInn1 f_{\textrm{cl},h}(\bfvarphi_h) + \intInn1 f_{\textrm{st},h}(\bfvarphi_h)\\
&\hspace{0.025\linewidth} + \intInn1 f_{\textrm{gs},h}(\bfvarphi_h),
\end{split}
\end{equation}
where $\intInn1$ denotes a quadrature rule for $\int_{t_n}^{t_{n+1}}\dots\dt$ which has to be chosen carefully. Quadrature rule $\intInn1$ plays the central role in discrete energy estimates. Indeed, as it is shown in Proposition~\ref{prop:discrete-energy-balance}, $\intInn1$ should ideally satisfy the following discrete space conservation laws:
\begin{equation}\label{prop:assumptions1}
\begin{split}
& \intInn1 \int\limits_{\Omega_h}\frac{1}{2}\vert\bfv_h^{n+1}\vert^2\div\bfw_h^{n+1}\dbfx = \int\limits_{\Omega_h^{n+1}}\frac{1}{2}\vert\bfv_h^{n+1}\vert^2\dbfx - \int\limits_{\Omega_h^{n}}\frac{1}{2}\vert\bfv_{h,n}^{n+1}\vert^2\dbfx\;,\\
& \intInn1 \int\limits_{\Sigma_h} \frac{\div_{\Sigma_h}\bfw_h^{n+1}}{\noCa\noRe}\dS = \int\limits_{\Sigma_h^{n+1}} \frac{1}{\noCa\noRe}\dS - \int\limits_{\Sigma_h^{n}} \frac{1}{\noCa\noRe}\dS\;,\\
& \intInn1 \int\limits_{\eta_h}\frac{\cos\theta_s}{\noCa\noRe}\bfm_{\partial\Gamma_h}\cdot\bfw_h^{n+1}\ds = \int\limits_{\Gamma_h^{n+1}}\frac{\cos\theta_s}{\noCa\noRe}\dS - \int\limits_{\Gamma_h^{n}}\frac{\cos\theta_s}{\noCa\noRe}\dS\;,\\
& \intInn1 \int\limits_{\partial\Omega_h} \Phi_h\bfw_h^{n+1}\cdot\bfn\dS = \int\limits_{\Omega_h^{n+1}}\Phi_h\dS - \int\limits_{\Omega_h^{n}}\Phi_h\dS.
\end{split}
\end{equation} 
In practice, $\intInn1$ may be chosen differently for the different terms in (\ref{prop:assumptions1}) -- for more details see also (\ref{prop:assumptions1-odt}). Keeping that in mind, we slightly abuse the notation and use the same symbol for any case. Whether such quadrature rule is possible to chose is another manner, discussed in more detail in Remark~\ref{rem:questionable-terms}. The formal definition of SCL and some insights on how to satisfy it on the discrete level are given in \ref{apx:dSCL}.

\subsection{Discrete energy inequality}
Let us define the discrete counterpart of the energy balance (\ref{prop:ceb}) by
\begin{equation}\label{eq:discr-count-eb}
{\cal E}_{\Delta t,h} = \frac{E_{k,h}^{n+1} - E_{k,h}^{n}}{\Delta t} + P_{v,h}^{n+1} + \frac{E_{fs,h}^{n+1} - E_{fs,h}^{n}}{\Delta t} + \frac{E_{w,h}^{n+1} - E_{w,h}^{n}}{\Delta t} + P_{fr,h}^{n+1} + \frac{E_{p,h}^{n+1} - E_{p,h}^{n}}{\Delta t}.
\end{equation} 
Above, additional subscript $h$ and superscripts $n$ and $n+1$ indicate the discrete counterparts of quantities defined in (\ref{eq:energies-notation}).
\begin{prop}\label{prop:discrete-energy-balance}
Assume that quadrature rule $\intInn1$ in formulation (\ref{sys:vf-discrete}) is such that discrete SCLs of co--dimension $0$ and $1$ are satisfied, i.e. that assumptions (\ref{prop:assumptions1}) hold. Furthermore, assume 
\begin{equation}\label{prop:assumptions2}
\begin{split}
& \bfw_{h,t}^{n+1}\cdot\bfn=\bfv_{h,t}^{n+1}\cdot\bfn \textrm{ on } \partial\Omega_h(t)\;,\:t\in[t_n,t_{n+1}].
\end{split}
\end{equation}
Then, discrete variational formulation (\ref{sys:vf-discrete}) is stable in energy norm in sense that the following energy inequality holds:
\begin{equation}\label{eq:discrete-energy-inequality}
{\cal E}_{\Delta t,h} \leq 0.
\end{equation}
\end{prop}
\begin{proof}
The proof is rather technical but quite straightforward. We write it down concisely. 

Let us take $\bfvarphi_h=\bfv_h^{n+1}$ in discrete formulation (\ref{sys:vf-discrete}). Then, using inequality $ab\leq \frac{1}{2}a^2+\frac{1}{2}b^2$ and employing Reynolds transport theorem in the process, it may be estimated 
\begin{equation}\label{ineq:transient-term}
\begin{split}
 \int\limits_{\Omega_h^{n+1}}\vert\bfv_h^{n+1}\vert^2\dbfx - \int\limits_{\Omega_h^{n}}\bfv_{h,n}^{n+1}\cdot\bfv_h^{n}\dbfx & \geq \int\limits_{\Omega_h^{n+1}}\frac{1}{2}\vert\bfv_h^{n+1}\vert^2\dbfx - \int\limits_{\Omega_h^{n}}\frac{1}{2}\vert\bfv_h^{n}\vert^2\dbfx \\
& + \int\limits_{t_n}^{t_{n+1}} \left(\; \int\limits_{\Omega_h}\frac{1}{2}\vert\bfv_h^{n+1}\vert^2 \div\bfw_h^{n+1}\dbfx \right) \dt.
\end{split}
\end{equation}
Hence, the last term in above inequality (\ref{ineq:transient-term}) cancels out with $\intInn1 \frac{1}{2}d_h(\bfv_h^{n+1},\bfv_h^{n+1};\bfw_h^{n+1})$ due to assumption (\ref{prop:assumptions1})$_1$. Terms involving trilinear form $c_h^{n+1}(\dots)$, $a_h^{n+1}(\dots)$, $b_h^{n+1}(\dots)$ and $r_h^{n+1}(\dots)$ are obtained by simply inserting $\bfvarphi_h=\bfv_h^{n+1}$. Furthermore,
\begin{equation}
\intInn1 f_{\textrm{cl},h}(\bfv_h^{n+1}) = \int\limits_{\Gamma_h^{n+1}}\frac{\cos\theta_s}{\noCa\noRe}\dS - \int\limits_{\Gamma_h^{n}}\frac{\cos\theta_s}{\noCa\noRe}\dS
\end{equation}
due to assumptions (\ref{prop:assumptions1})$_3$ and (\ref{prop:assumptions1})$_2$ (formal derivation follows that of Lemma~\ref{lem:semidiscr-wett}). Similarly, due to assumptions (\ref{prop:assumptions1})$_2$ and (\ref{prop:assumptions1})$_2$, following the derivation from Lemma~\ref{lem:semidiscr-fs} we obtain
\begin{equation}
\intInn1 f_{\textrm{st},h}(\bfv_h^{n+1}) = \int\limits_{\Sigma^{n+1}}\frac{1}{\noCa\noRe}\dS - \int\limits_{\Sigma^{n}}\frac{1}{\noCa\noRe}\dS.
\end{equation}
Finally, once again applying Reynolds transport theorem and using assumptions (\ref{prop:assumptions1})$_4$ and (\ref{prop:assumptions2}), we obtain
\begin{equation}
\begin{split}
\intInn1 f_{\textrm{gs},h}(\bfv_h^{n+1}) =  \int\limits_{\Omega_h^{n+1}}\Phi_h\dS - \int\limits_{\Omega_h^{n}}\Phi_h\dS.
\end{split}
\end{equation}
Summing it all up, discrete energy estimate (\ref{eq:discrete-energy-inequality}) follows.
\end{proof}
\begin{remark}
Note that if we work with forcing term in form $f_{\textrm{g},h}(\dots)$ rather than incorporating it into the pressure term, we are left with term of spurious energy of order $\Delta t$ in Proposition~\ref{prop:discrete-energy-balance}. Indeed,
\begin{equation}
\begin{split}
\int\limits_{t_n}^{t_{n+1}} \int\limits_{\Omega_h} \nabla\Phi\cdot\bfv^{n+1}_h\dbfx  \dt = E_{p,h}^{n+1} - E_{p,h}^{n} - \int\limits_{t_n}^{t_{n+1}} \int\limits_{\Omega_h} \Phi_h\div\bfv^{n+1}_h\dbfx\dt
\end{split}
\end{equation}
and, in general (see Remark~\ref{rem:div-vh}), 
\[
\int\limits_{t_n}^{t_{n+1}} \int\limits_{\Omega_h} \Phi_h\div\bfv^{n+1}_h\dbfx\dt \neq 0.
\]
\end{remark} 

\begin{remark}
Energy inequality (\ref{eq:discrete-energy-inequality}) is the discrete counterpart of energy balance (\ref{eq:energy-balance}). It shows that, under the assumptions of Proposition~\ref{prop:discrete-energy-balance}, time discretization does not introduce spurious energy into the system. In other words, discrete formulation (\ref{sys:vf-discrete}) is stable in the energy norm. See also Remark~\ref{rem:questionable-terms} for more details on validity of the assumptions in Proposition~\ref{prop:discrete-energy-balance}.
\end{remark}

\begin{remark}\label{rem:questionable-terms}
An important question related to Proposition~\ref{prop:discrete-energy-balance} is whether it is possible to chose quadrature rule $\intInn1$ such that assumptions (\ref{prop:assumptions1}) to hold. It was shown in \cite{ivancic19} that this is indeed possible for identity (\ref{prop:assumptions1})$_1$. Strategy derived in \cite{ivancic19} can be straightforwardly extended for identity (\ref{prop:assumptions1})$_4$. A concise derivation is given in \ref{apx:dSCL}.

The other two identities in (\ref{prop:assumptions1}), namely, identities (\ref{prop:assumptions1})$_2$ and (\ref{prop:assumptions1})$_3$, cannot be satisfied exactly in a general case, at least to the best of our knowledge. However, the strategy derived in \cite{ivancic19} may still be employed to some extent. It allows us to conveniently employ arbitrary quadrature rules and, at least in theory, to satisfy identities (\ref{prop:assumptions1})$_2$ and (\ref{prop:assumptions1})$_3$ up to desired order in $\Delta t$. More precisely, we may assume
\begin{equation}\label{prop:assumptions1-odt}
\begin{split}
& \intInn1 \int\limits_{\Sigma_h} \frac{\div_{\Sigma_h}\bfw_h^{n+1}}{\noCa\noRe}\dS = \int\limits_{\Sigma_h^{n+1}} \frac{1}{\noCa\noRe}\dS - \int\limits_{\Sigma_h^{n}} \frac{1}{\noCa\noRe}\dS + \;o(\Delta t^k)\;,\\
& \intInn1 \int\limits_{\eta_h}\frac{-\cos\theta_s}{\noCa\noRe}\bfm_{\partial\Gamma_h}\cdot\bfw_h^{n+1}\ds = \int\limits_{\Gamma_h^{n+1}}\frac{\cos\theta_s}{\noCa\noRe}\dS - \int\limits_{\Gamma_h^{n}}\frac{\cos\theta_s}{\noCa\noRe}\dS+ \;o(\Delta t^k)\;,
\end{split}
\end{equation} 
where $k$ depends on the choice of quadrature formula $\intInn1$. More details are given in \ref{apx:dSCL}.
\end{remark}
Taking the discussion in Remark~\ref{rem:questionable-terms} into account, energy inequality (\ref{eq:discrete-energy-inequality}) in Proposition~\ref{prop:discrete-energy-balance} may be restated as
\begin{equation}
{\cal E}_{\Delta t, h} + o(\Delta t^k) \leq 0,
\end{equation}
where $k$ depends on the choice of quadrature rule $\intInn1$.

\subsection{Linearization and mesh updating procedure}
As mentioned in Section~\ref{sec:ALE}, an implicit approach is employed for the geometry evolution. The ALE map is constructed via the identity
\begin{equation}\label{eq:practical-ALE-construct}
\begin{split}
&\ALE_h^{[n,n+1]}(\bfx^n) = \bfx^n + t\bfw_{h,n}^{n+1}(\bfx^n)\;,\:t\in[0,\Delta t]\;,\\
&\Omega_h^{n+1}\ni\bfx^{n+1} = \bfx^n + \Delta t\;\bfw_{h,n}^{n+1}(\bfx^n)\;,
\end{split}
\end{equation}
where mesh velocity $\bfw_h^{n+1}$ is constructed by solving
\begin{equation}\label{eq:practical-w-construct}
\begin{split}
-\Delta\bfw_h^{n+1} & = 0\textrm{ in }\Omega_h^{n+1},\\
\bfw_h^{n+1} & = \bfv_h^{n+1}\textrm{ on }\partial\Omega_h^{n+1}.
\end{split}
\end{equation}
Clearly, such approach results in a highly non--linear system. For linearization of the fluid equations we use Newton's linearization technique, which is standard. For linearization of the geometrical coupling, the complete coupled system ((\ref{sys:vf-discrete}),(\ref{eq:practical-ALE-construct}),(\ref{eq:practical-w-construct})) may be linearized using the Newton's technique (proposed in \cite{soulaimani98}) which involves some tools from {\it shape optimization} theory. Alternatively, a semi--implicit approach, where geometrical coupling is formally decoupled from the fluid equations, may be used. The two systems (\ref{sys:vf-discrete}) and ((\ref{eq:practical-ALE-construct}),(\ref{eq:practical-w-construct})) are then solved iteratively until convergence is achieved. Such approach has been used in \cite{hecht17} for FSI problems. Numerical results presented below are obtained employing the latter, semi--implicit approach mainly to avoid dealing with the shape derivatives which exceed the topic of this paper. {\it Mini} finite element space is employed for the velocity--pressure pair, $\sV_h\times\sQ_h=[\fespaceminiP1]^3\times\fespaceP1$, while space of linear functions is employed for the ALE map and mesh velocity, $\sA_h=[\fespaceP1]^3$.
\begin{remark}
Note that the boundary of the domain, $\partial\Omega_h$, is evolving with material velocity $\bfv_h$ (see (\ref{eq:practical-w-construct})$_2$). This ensures the validity of assumption (\ref{prop:assumptions2}) in Proposition~\ref{prop:discrete-energy-balance}, which was essential for deriving the energy balance. However, although this is a very standard choice in droplet kinematics simulations, it does result in more often need for a re--meshing. An alternative choice of the domain boundary velocity which aims to avoid this issue may be studied in the future.
\end{remark}

As may be noted from the derivation of the energy estimates (Proposition~\ref{prop:discrete-energy-balance}), time step has no direct impact on the scheme stability (assuming that the quadrature rule is chosen such that the violation of discrete SCL is small). The main restriction on the time step comes from the non--linearity of the system to be solved (initial guess has to be close enough for the scheme to converge), and due to implicit involvement of the geometry itself into the system through the curvature term. For example, if the position of the free surface is updated with the fluid velocity, then the difference between two configurations may be too large in areas where velocity is large if the time step is not sufficiently small. Consequently, these areas may exhibit unphysically large curvature which, in term, enlarges the capillary forces and causes the scheme breakdown. This issue is related directly to the geometry updating and it is also characteristic for the {\it mean curvature flow} simulations. Hence, in practice, the time step has to be chosen sufficiently small in order to ensure the smooth geometry evolution and the scheme convergence. While it is clear that it somehow depends on the characteristic velocity and, consequently, on the dimensionless numbers of the system, the exact relationship is not known to the best of our knowledge. Finding the largest suitable time step for which  the geometry iterative updating converges usually requires some experimenting.

\section{Numerical validation}\label{sec:numerical-validation}
This section deals with validation of the newly proposed numerical method. First, we provide a short review on the non--dimensionalization process -- more specifically, we try to justify our choice of the characteristic velocity. Then we compare the numerical result of a simplified scenario with an existing analytical solution in 2D setup (in subsection~\ref{subsec:2d}). We also provide a short report on convergence with respect to the mesh size parameter $h$ and time step $\Delta t$ in both 2D and 3D setups (in subsections~\ref{subsec:2d} and \ref{subsec:3d}). This is done by comparing the steady state solutions of simplified scenarios obtained with different parameters $h$ and $\Delta t$. It is shown that the energy balance is independent of the time step and mesh size parameter confirming the theoretical predictions developed earlier. Finally, a realistic 3D case of a droplet on an inclined and non--homogeneous surface is investigated. We show that our scheme is able to simulate a complex droplet behaviors (sliding and rolling) while keeping its energy stability even for a long time simulations. For all cases reported in this paper, Crank--Nicolson quadrature rule is employed for handling the discrete SCLs on $\Sigma_h$ and $\Gamma_h$. Hence, the spurious energy is of second order in time and this proved to be sufficiently small to ensure the scheme stability for all cases reported below.

All of the simulations are performed with {\it FreeFem++} software (\cite{hecht12}).

\subsection{The choice of the characteristic length and velocity}\label{subsec:charac-length}
The initial configuration of the droplet is set up as a {\it spherical cap} of a given volume $V_0$ and where the angle between sphere and the {\it cut-off plane} equals to contact angle $\theta_s$. See Figure~\ref{fig:setup} for illustration. The characteristic length $L$ is then defined as the radius of such spherical cap.

To determine characteristic velocity $U$, let us consider droplet on a horizontal surface. Assuming the order of magnitude $\vert p \vert \sim \vert \sigma \Delta_\Sigma\bfx_\Sigma \vert$
and taking into account $p_c=\varrho U^2$, we obtain
\begin{equation}\label{eq:charact-vel}
\varrho U^2 \sim \frac{\gamma}{L}\textrm{ i.e. } U=\sqrt{\frac{\gamma}{\varrho L}}.
\end{equation}
Inserting this definition for characteristic velocity $U$ into the dimensionless numbers $\noFr$, $\noRe$ and $\noCa$, we obtain
\[
\noFr^2 = \noBo^{-1}\;,\:\noRe = \noLa^{1/2}\;,\:\noCa = \noLa^{-1/2}\;,\textrm{ where }\noBo = \frac{\varrho g L^2}{\gamma}\;,\:\noLa = \frac{\gamma\varrho L}{\mu^2}\;.
\] 

Alternative definitions for characteristic velocity would also make sense in certain scenarios. For example, on an inclined plane, assuming dominantly sliding dynamics of a droplet, the characteristic velocity can be derived from the balance between gravity acceleration and friction forces. Characteristic velocity for purely rolling droplets has already been investigated in \cite{mahadevan99}. Definition (\ref{eq:charact-vel}) for the characteristic velocity proved to be a decent choice for all of the numerical simulations presented in this paper.

Regarding the Navier slip coefficient, in its dimensional form it is given by $\varsigma=\mu/l_s$, where $l_s$ is the {\it slip length} (see \cite{qian03,qian06,ganesan09,rothstein10}). $l_s$ is typically of order $10^{-9}$ m (i.e. order of nanometers) but it can be greatly altered on manufactured surfaces by, e.g, applying hydrophobic coatings (\cite{nakajima11}). For example, in \cite{rothstein10}, surfaces with slip length greater than $25$ $\mu$m have been reported. Using the physical quantities typical for water, the slip coefficient in its dimensionless form (used in this paper) can be estimated to be of order $10^{-1}\sim 10^4$ (from surfaces with large to small slip length). For the purposes of this paper, hydrophobic surfaces with large slip length are of the main interest, on which droplets are able to depin and start moving.

\subsection{Validation in 2D: Comparison with analytical solution}\label{subsec:2d}
The exact solution of the Young--Laplace equation governing the shape of 2D droplets in the gravity field was derived in \cite{cunjing18}. On a horizontal surface the (symmetric) shape of a droplet of volume $V$ is given by parametric equations
\begin{equation}\label{eq:YL}
\begin{split}
x(\vartheta) = \pm \frac{\sqrt{2} a}{2} \int\limits_0^\vartheta \frac{\cos\xi}{\sqrt{A-\cos\xi}}\td\xi\;,\:
y(\vartheta) = -\sqrt{2}a\sqrt{A-\cos\vartheta}
\end{split}
\end{equation}
for $\vartheta\in[0,\theta_s]$, where $a^2=\noBo^{-1}$ and $A\in[1,\infty]$ is a constant determined by
\[
V = 2a^2\left[ \sqrt{A-\cos\theta_s}\int\limits_0^{\theta_s}\frac{\cos\xi}{\sqrt{A-\cos\xi}}\td\xi - \sin\theta_s \right].
\] 
To validate the proposed numerical method, we ran simulations with with $\varsigma=0$ (i.e. perfect slip condition) and for various choices of other dimensionless numbers until steady state solution was reached . At $T=16$ steady state is reached (or almost reached) for all the cases presented below. The initial configuration is chosen as described in the previous subsection. The obtained droplet shapes have been then compared with the analytical solution governed by Young--Laplace equation and excellent match has been observed. Simulations were run on (initially) uniform meshes of different mesh sizes and with various time steps. For the results presented below, $\noLa=1$ and $\theta_s=3\pi/4$, while $\noBo$ may vary and is specified for each case scenario. Dimensionless volume is $V_0=2.85$ in all cases. Relative volume oscillations defined by $\vert V_0 - V(t)\vert/{V_0}$, where $V(t)$ denotes droplet volume at time $t$, are conserved within order of $10^{-4}$ for all cases. The whole simulation does not require any mesh adaptation since the relative deformation is fairly small (more complex cases with necessary mesh adaptation are investigated later).

In Figure~\ref{fig:2Dshapes}, droplet shapes after simulations reached steady state are shown. For all cases in this scenario, time step was chosen as $\Delta t=0.1$ and mesh quality parameter $h=0.1$. This corresponds to $346$ mesh vertices for (initially) uniform mesh. In Figure~\ref{fig:varyBo} shapes governed by analytical solution (\ref{eq:YL}) are shown for $\noBo=0.2,0.4,0.8$. Figures~\ref{fig:Bo2e-1}, \ref{fig:Bo4e-1} and \ref{fig:Bo8e-1} show analytical versus numerically obtained droplet shapes. Excellent agreement can be observed.

\begin{figure}[h]
\subfloat[Exact droplet shapes for $\noBo=0.2,0.4,0.8$.]{\includegraphics[width=0.48\linewidth]{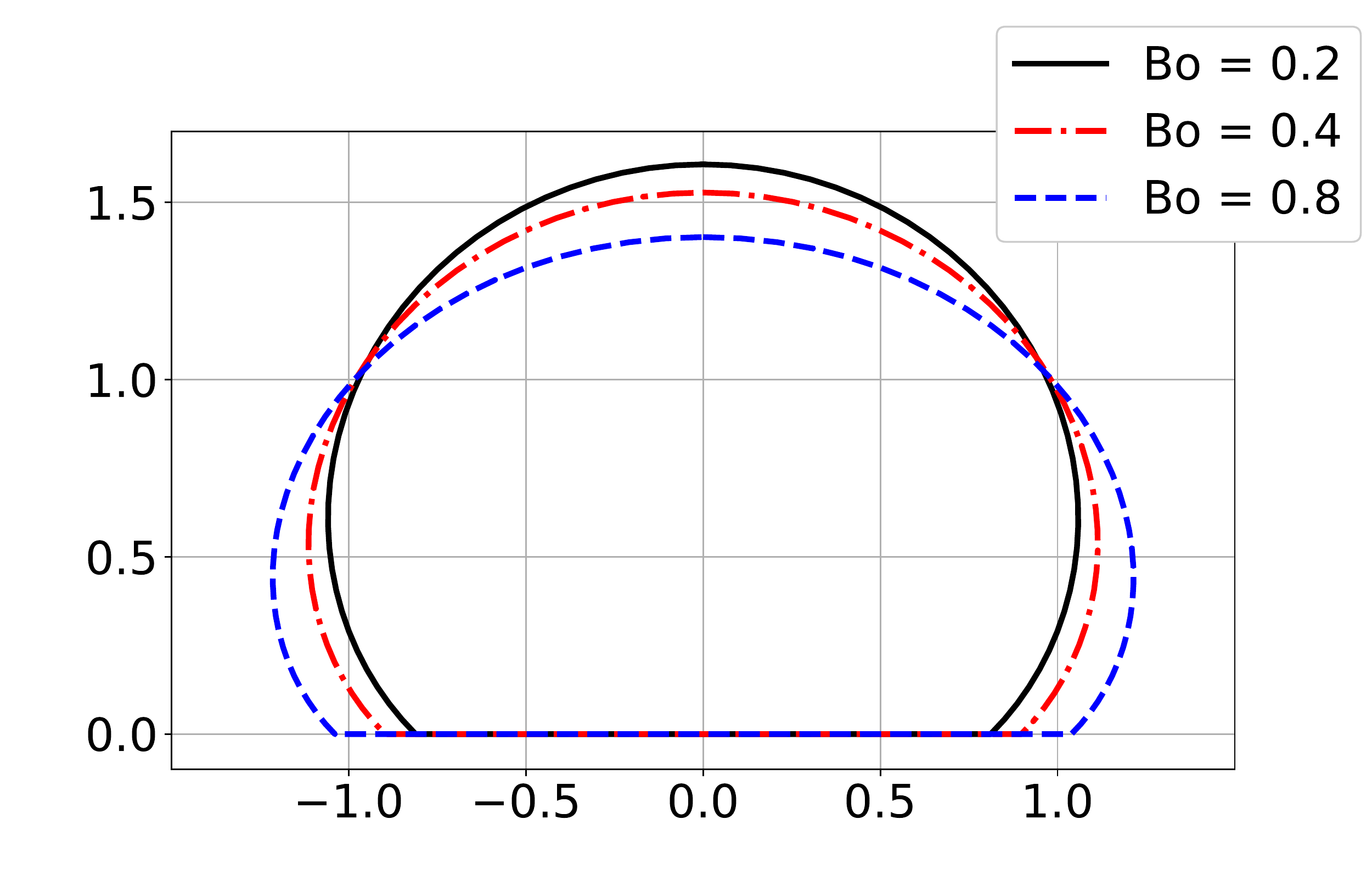}\label{fig:varyBo}}\hspace{0.02\linewidth}\subfloat[Exact versus numerically obtained (steady state) droplet shapes, $\noBo=0.2$.]{\includegraphics[width=0.48\linewidth]{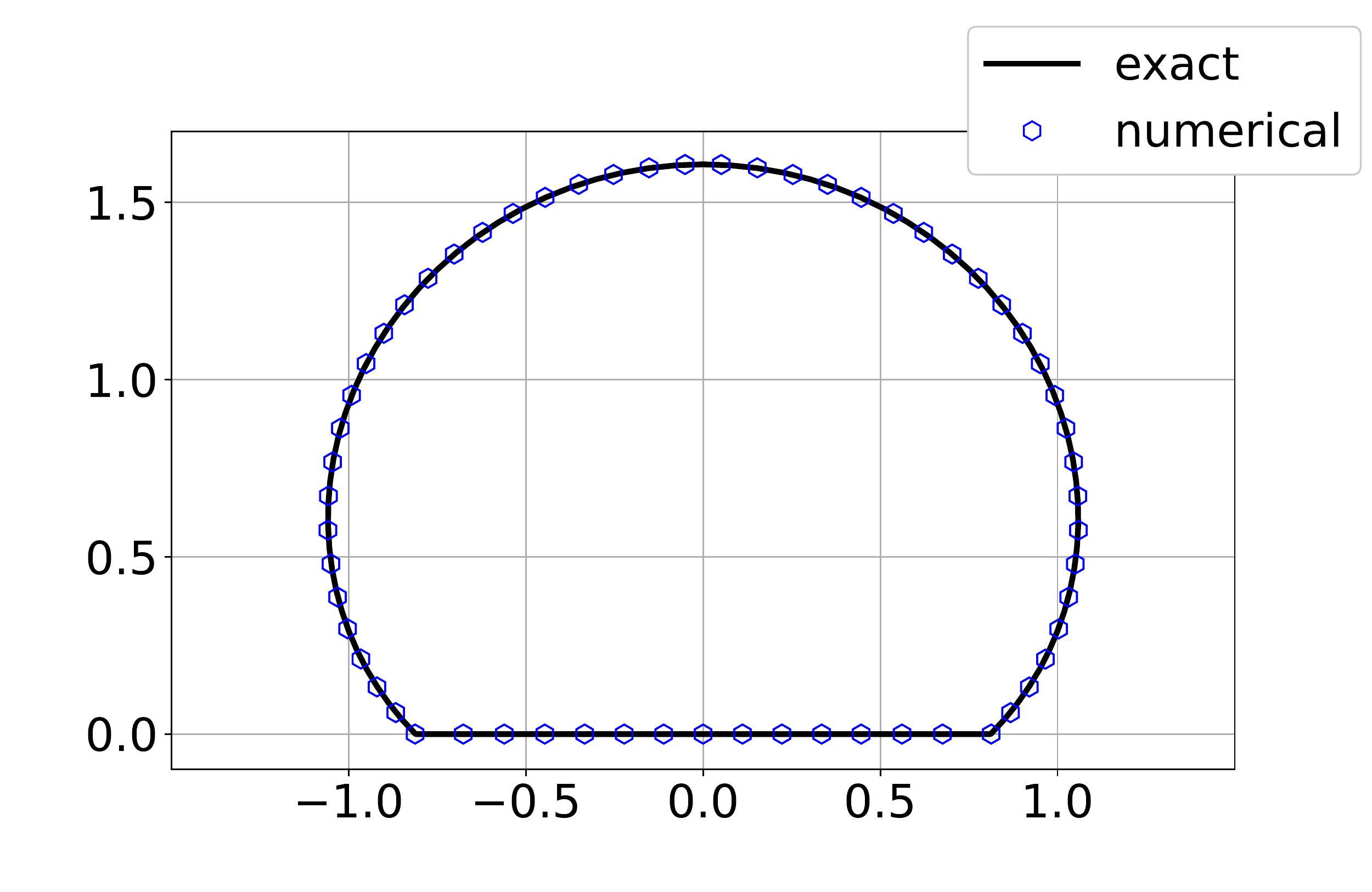}\label{fig:Bo2e-1}}

\subfloat[Exact versus numerically obtained (steady state) droplet shapes, $\noBo=0.4$.]{\includegraphics[width=0.48\linewidth]{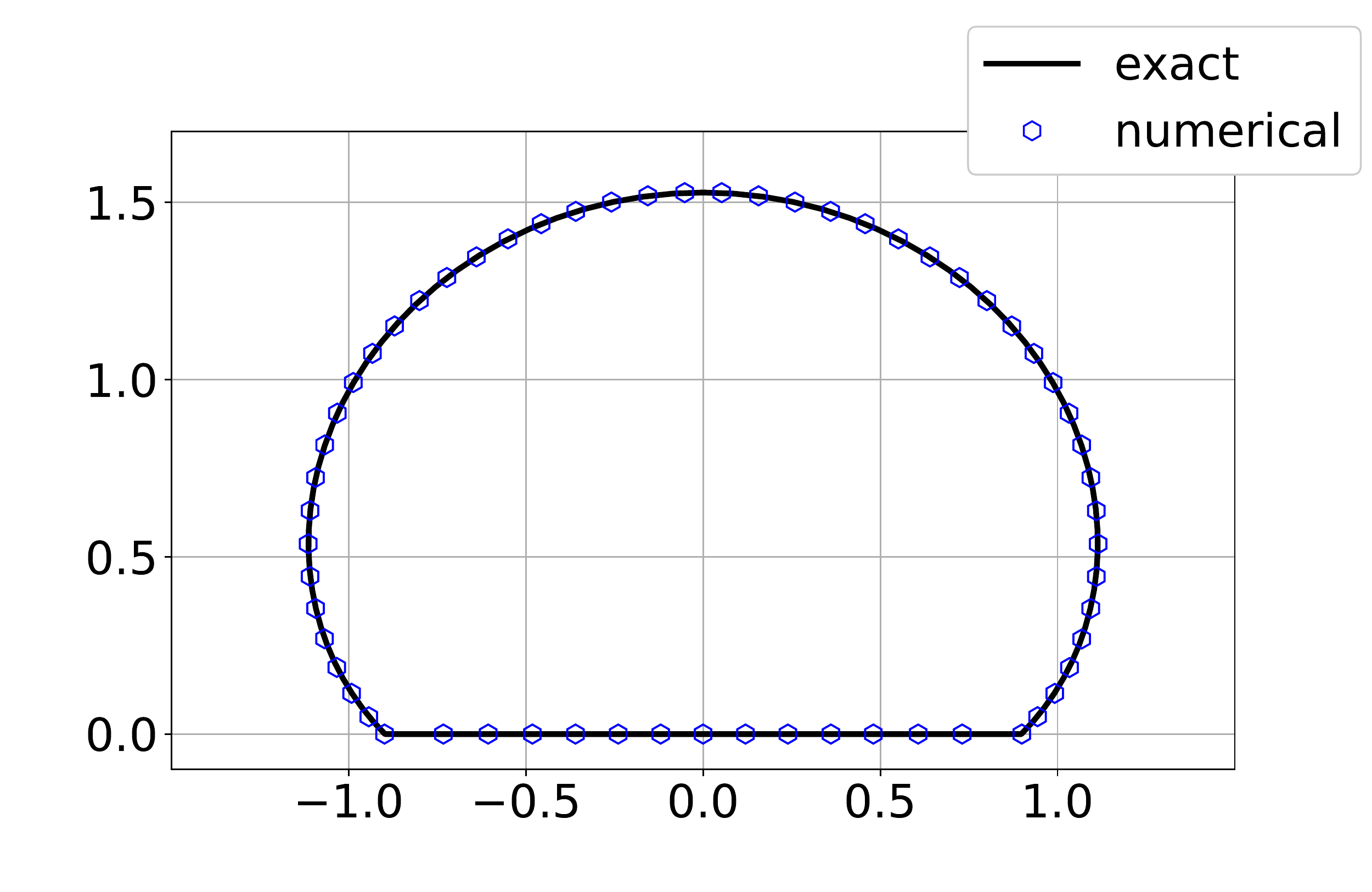}\label{fig:Bo4e-1}}\hspace{0.02\linewidth}\subfloat[Exact versus numerically obtained droplet shapes, $\noBo=0.8$.]{\includegraphics[width=0.48\linewidth]{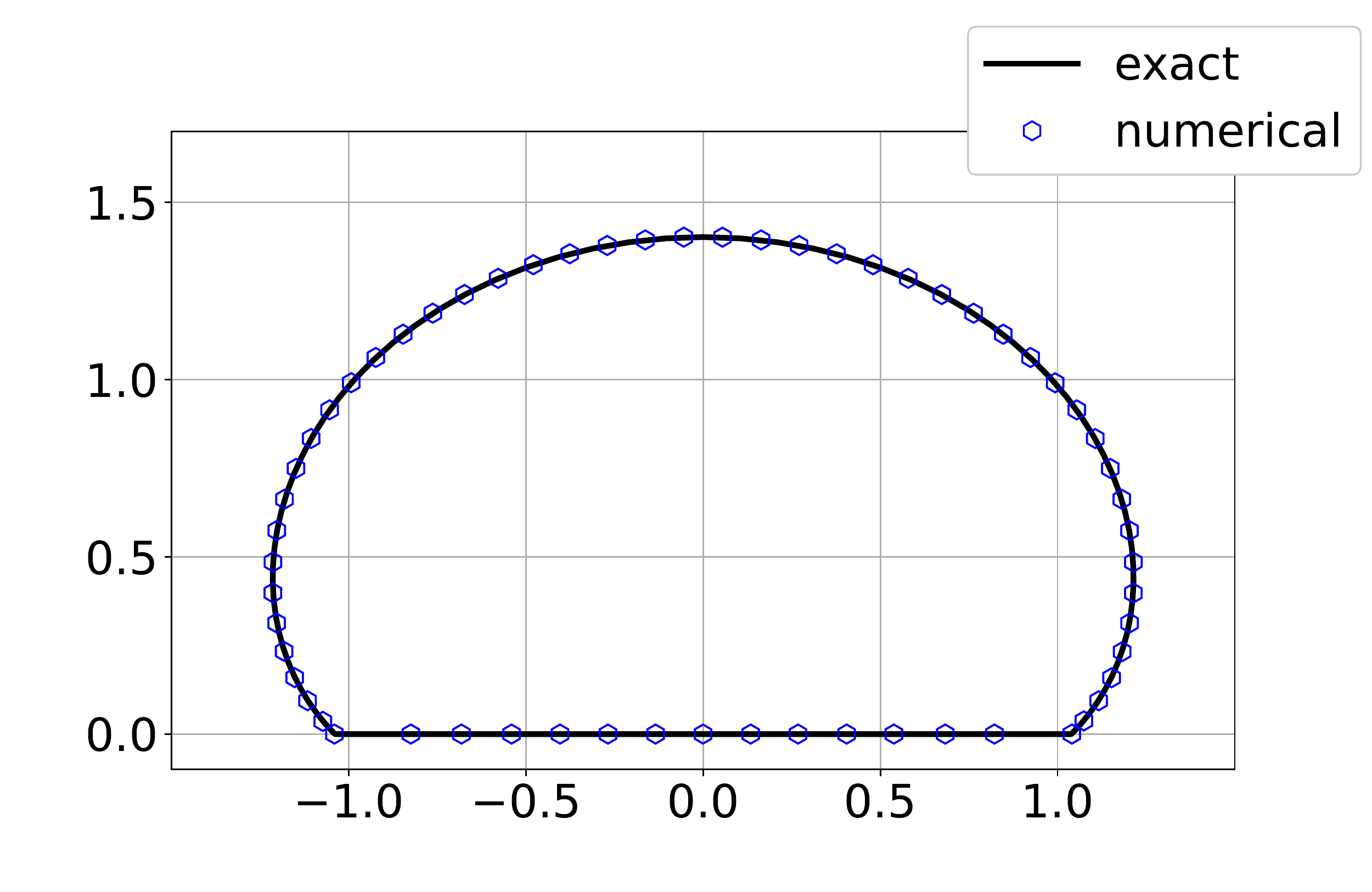}\label{fig:Bo8e-1}}

\caption{Comparison of the exact and numerically obtained (steady state) droplet shapes for the various choices of Bond numbers. For the presented numerical results $\Delta t=0.1$ and $h=0.1$.}
\label{fig:2Dshapes}
\end{figure}

In Figure~\ref{fig:2D-mesh-time} steady state droplet shapes numerically obtained for various choices of the time step $\Delta t$ and mesh quality parameter $h$ are shown. Bond number is kept fixed for all simulations, $\noBo=0.2$. It can be observed that for all cases numerically obtained steady state droplet shapes converge towards the analytical shape.

\begin{figure}[h]
\subfloat[Numerically obtained steady state droplet shapes for different choices of $h$ with $\Delta t=0.2$.]{\includegraphics[width=0.48\linewidth]{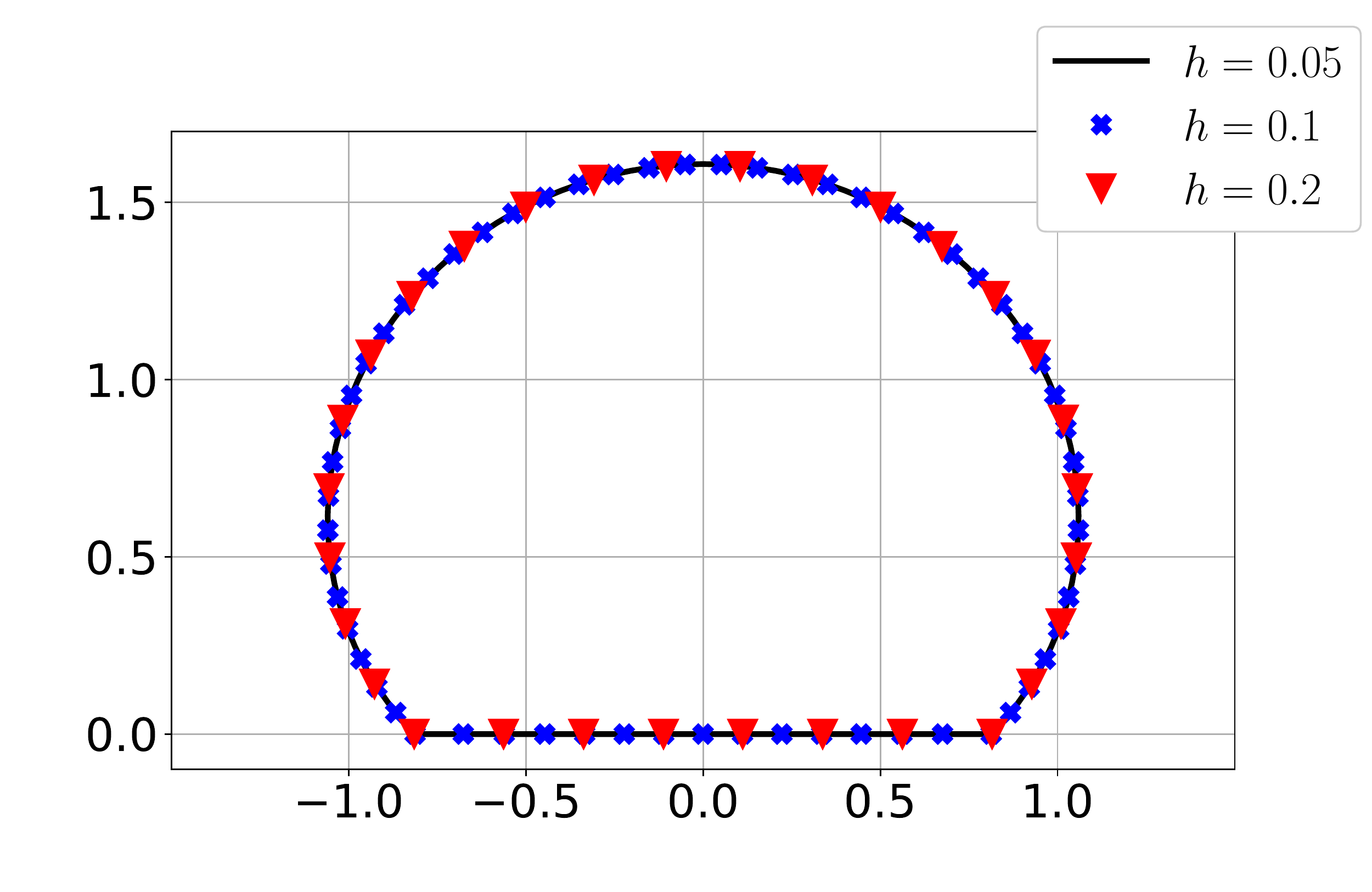}\label{fig:vary-mesh}}\hspace{0.04\linewidth}\subfloat[Numerically obtained steady state droplet shapes for different choices of $\Delta t$, with $h=0.1$.]{\includegraphics[width=0.48\linewidth]{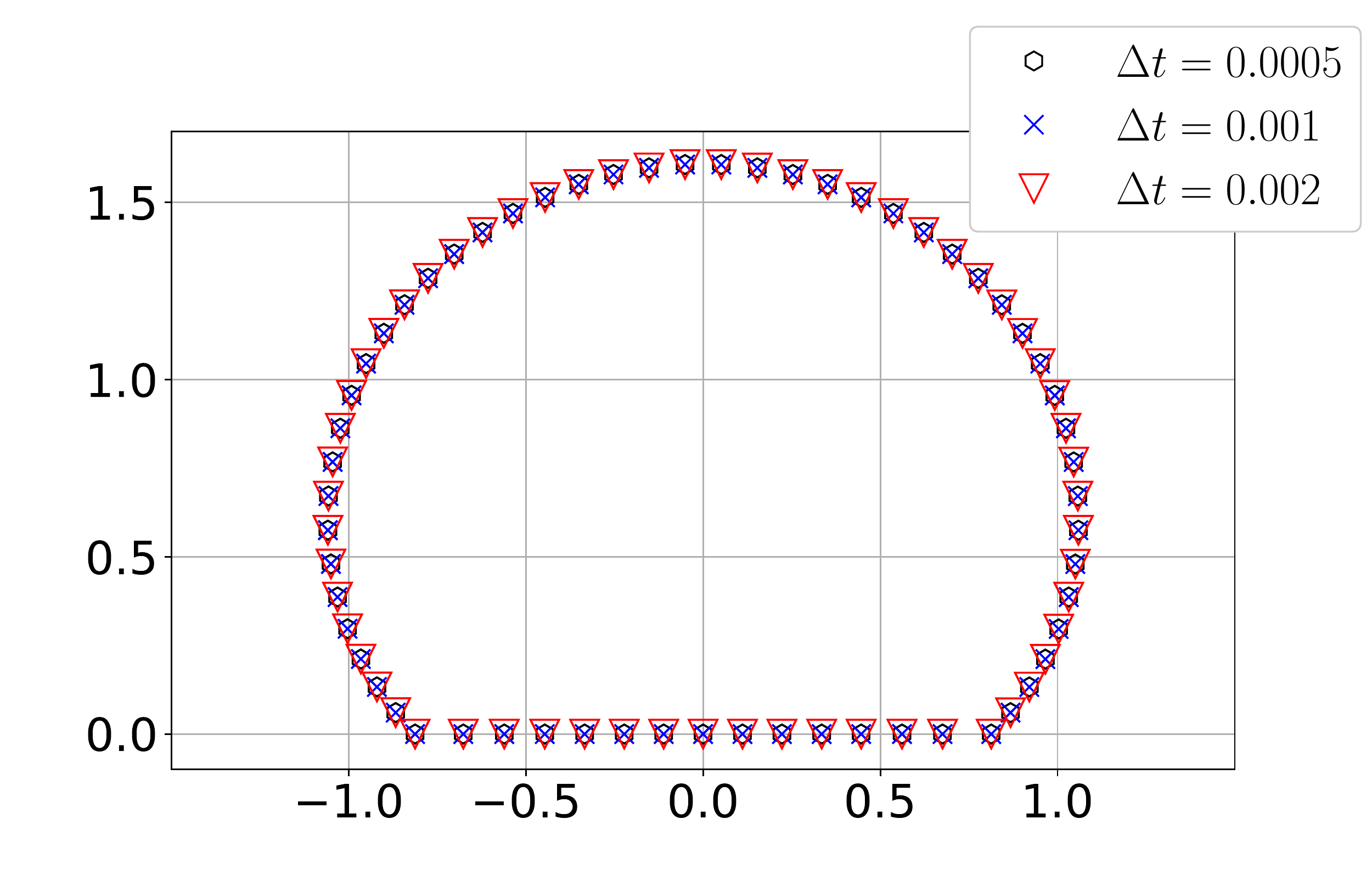}\label{fig:vary-time}}
\caption{Comparison of the numerically obtained (steady state) droplet shapes for the various choices of mesh parameter $h$ and time step $\Delta t$. Bond number is kept fixed, $\noBo=0.2$. For all cases, steady state droplet shapes converge towards the shape predicted by Young--Laplace equation.}
\label{fig:2D-mesh-time}
\end{figure}

Recall, by Proposition~\ref{prop:discrete-energy-balance} and Remark~\ref{rem:questionable-terms}, the discrete counterpart of the energy balance defined by (\ref{eq:discr-count-eb}), ${\cal E}_{\Delta t, h}$, has been estimated as
\[
{\cal E}_{\Delta t,h}^{n+1} + o(\Delta t^k) \leq 0,
\]
where $k$ depends on the choice of quadrature rule $\intInn1$ in the free-surface and contact line terms. Ideally, ${\cal E}_{\Delta t,h}^{n+1} = 0$. In Figure~\ref{fig:2D-energies} we plot discrete energy balance ${\cal E}_{\Delta t,h}$ with respect to two different time steps and two different mesh parameters. It can be observed that neither the time step nor the mesh parameter influence the energy balance (this is in agreement with the theoretical estimates). We mentioned that this has been confirmed on multiple tests -- we plot here only two different scenarios for clear visualization.

\begin{figure}[h]
\subfloat[Evolution of quantity ${\cal E}_{\Delta t,h}$ for different choices of $\Delta t$, with $h=0.1$.]{\includegraphics[width=0.48\linewidth]{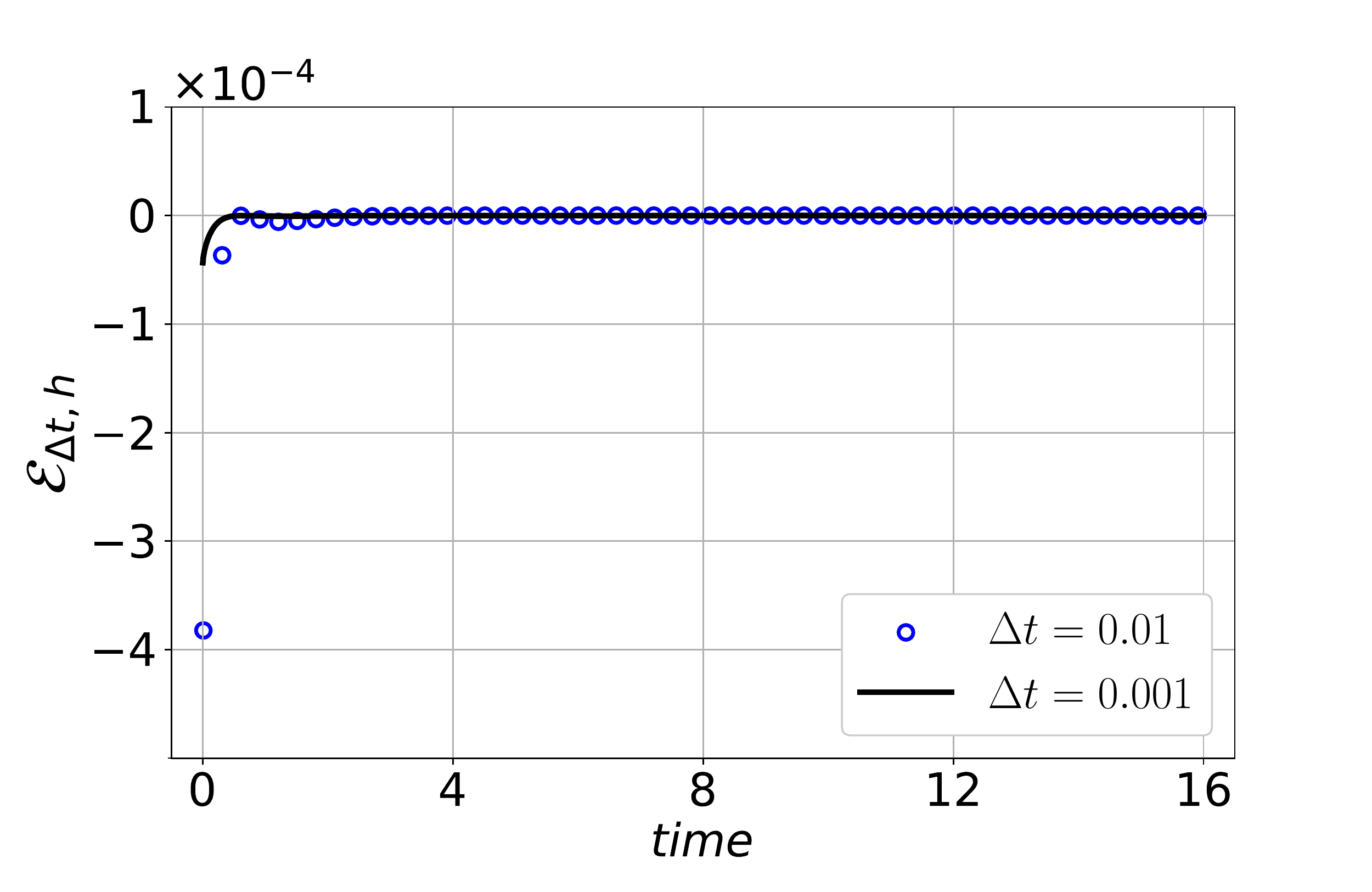}\label{fig:vary-t2d}}\hspace{0.04\linewidth}\subfloat[Evolution of quantity ${\cal E}_{\Delta t,h}$ for different choices of $h$, with $\Delta t=0.002$.]{\includegraphics[width=0.48\linewidth]{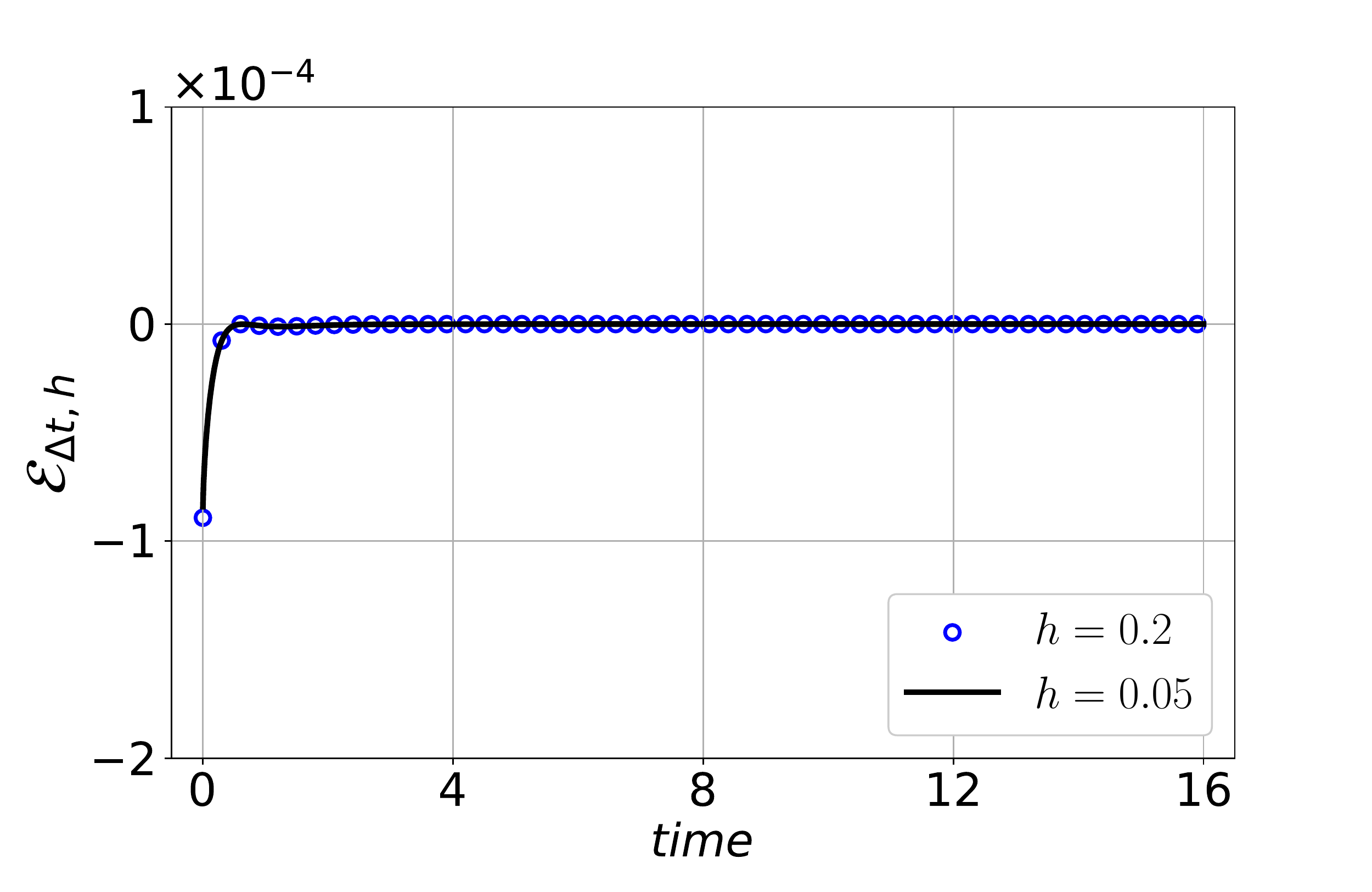}\label{fig:vary-h2d}}
\caption{Comparison of discrete energy balance ${\cal E}_{\Delta t,h}$ (\ref{eq:discr-count-eb}) for two different choices of mesh parameters $h$ and time steps $\Delta t$. Bond number is kept fixed, $\noBo=0.2$. Remark: circle markers are not plotted at each time step to improve visibility.}
\label{fig:2D-energies}
\end{figure}

\subsection{Validation in 3D}\label{subsec:3d}
To investigate our numerical scheme in a 3D setup, we performed analogous tests to those in 2D setup reported above. The scheme remains stable in the 3D setup regardlessly of the time step or mesh parameter (provided they are sufficiently small for the Newton method to converge). For the results presented below, $\varsigma=0$ (perfect slip), $\noLa=1$, $\noBo=0.4$, and the simulation final time is $T=16$ at which steady state is (almost) reached. Dimensionless volume is $V_0=3.88$ in all cases. Relative volume oscillations are conserved within order of $10^{-4}$ for all cases. Mesh adaptation is not required due to relatively small deformation.

In Figure~\ref{fig:3D-shapes}, numerically obtained steady state droplet shapes are shown for various time steps and mesh parameters. Figure~\ref{fig:finalmesh3d} shows the steady state droplet shape obtained on mesh with parameter $h=0.2$ and with time step $\Delta t=0.01$. To confirm that schemes with different time steps and mesh parameters do not converge towards different solutions, we compare the final droplet shapes obtained with different parameters. For clearer visualization, we extract curves obtained by intersecting the (steady state) droplet meniscus with planes perpendicular to $xy$--plane and passing through point $(0,0,0)$. Two such planes are shown in Figure~\ref{fig:3dplane-cut}. Figures \ref{fig:vary-h3d} and \ref{fig:vary-t3d} show the comparison of steady states obtained for two different mesh sizes ($h=0.2,0.1$) and two different time steps ($t=0.01,0.005$). It may be observed that an excellent agreement is achieved. In Figure~\ref{fig:3D-energies}, the profiles of discrete energy balance ${\cal E}_{\Delta t,h}$ defined by (\ref{eq:discr-count-eb}) and obtained for different choices of $h$ and $\Delta t$ are compared. It may be observed that the numerical results confirm the theoretical estimates.

\begin{figure}[h]
\subfloat[Steady state droplet shape obtained with scheme parameters $h=0.2$ and $\Delta t=0.01$.]{\includegraphics[width=0.45\linewidth]{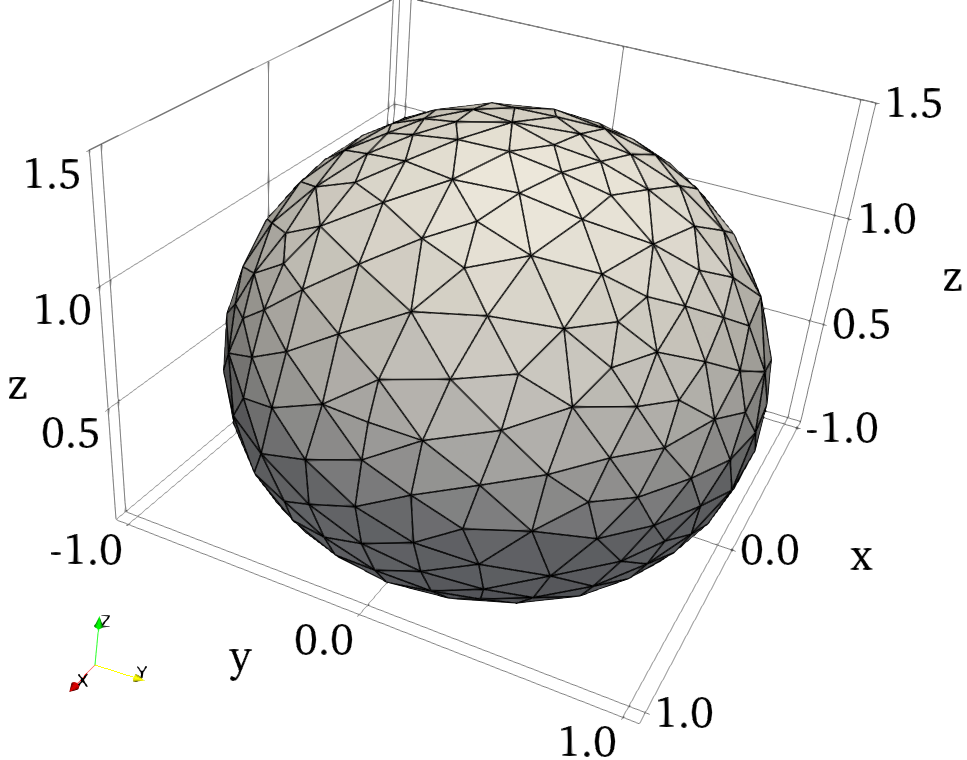}\label{fig:finalmesh3d}}\hspace{0.08\linewidth}\subfloat[Steady state droplet meniscus and two cutting planes passing through $(0,0,0)$.]{\includegraphics[width=0.45\linewidth]{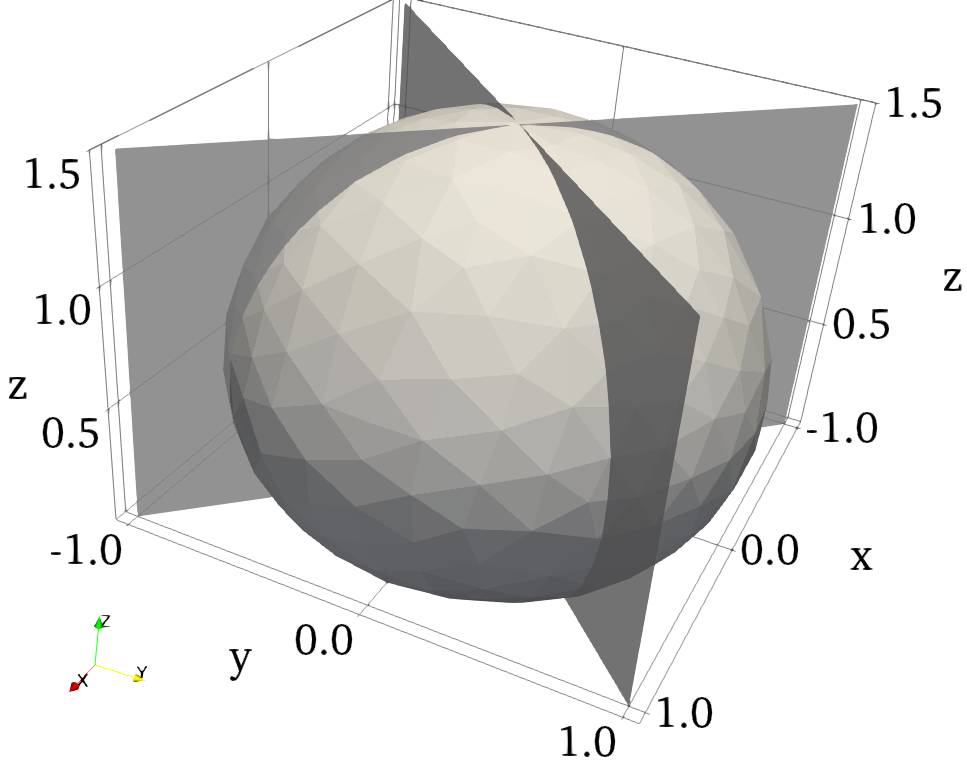}\label{fig:3dplane-cut}}

\subfloat[Intersection of numerically obtained steady state droplet menisci for different choices of $h$ (with $\Delta t=0.01$) and planes shown in Figure~\ref{fig:3dplane-cut}.]{\includegraphics[width=0.51\linewidth]{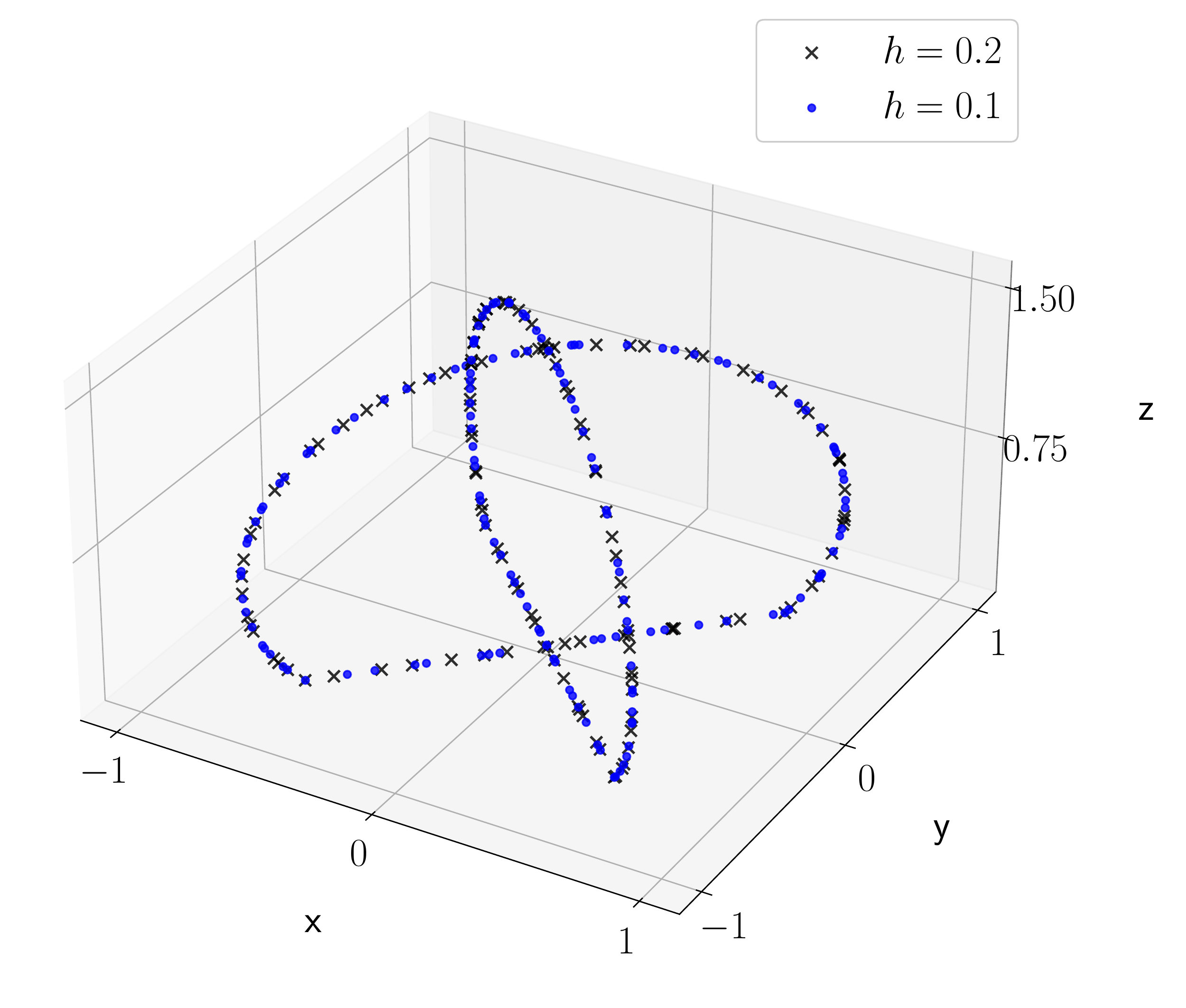}\label{fig:vary-h3d}}\hspace{0.02\linewidth}\subfloat[Intersection of numerically obtained steady state droplet menisci for different choices of $\Delta t$ (with $h=0.1$) and planes shown in Figure~\ref{fig:3dplane-cut}.]{\includegraphics[width=0.51\linewidth]{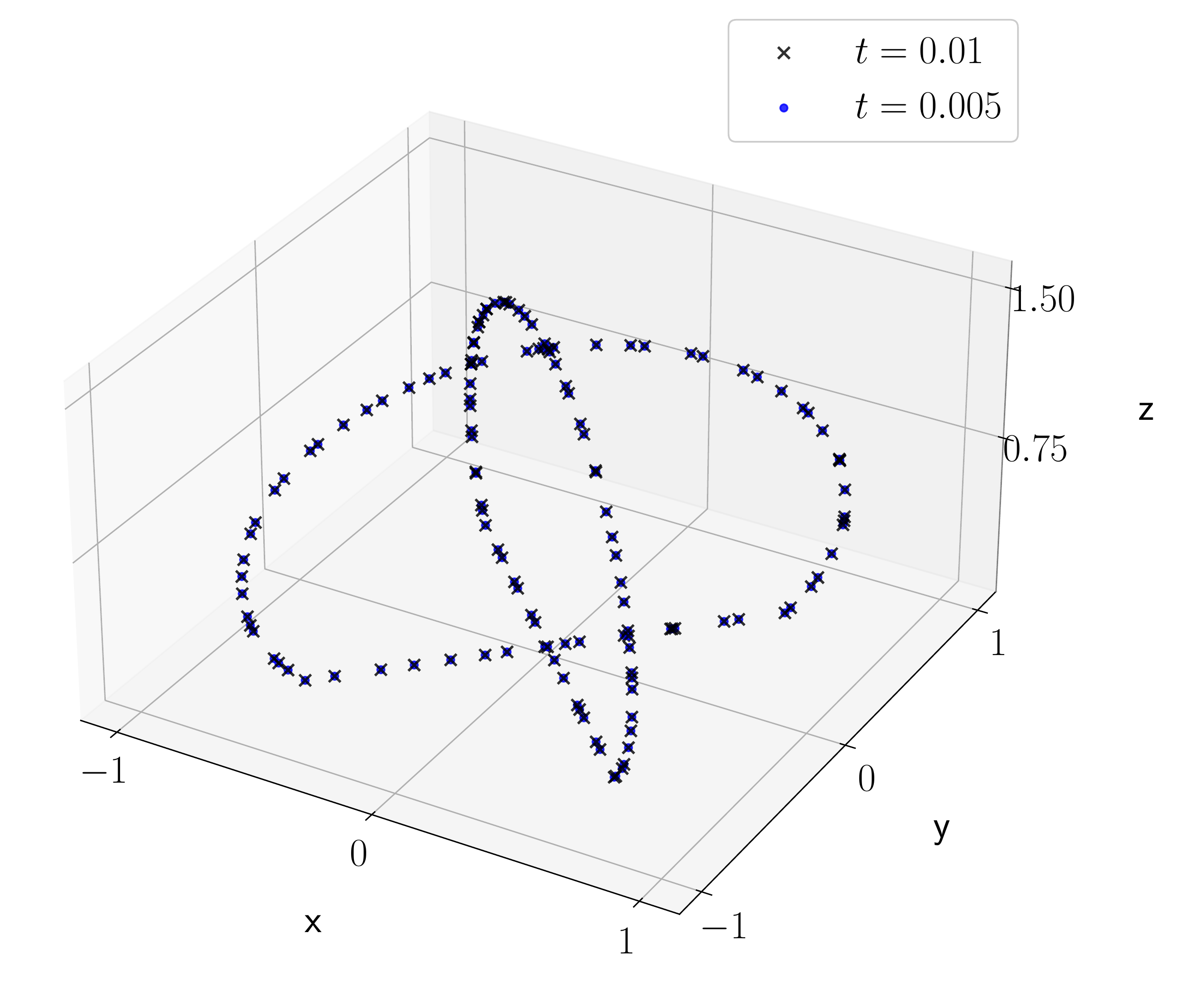}\label{fig:vary-t3d}}
\caption{Comparison of numerically obtained droplet menisci for the various choices of mesh parameters $h$ and time steps $\Delta t$. Bond number is kept fixed, $\noBo=0.4$. In all cases, steady state droplet shapes converge towards the same shape.}
\label{fig:3D-shapes}
\end{figure}

\begin{figure}[h]
\subfloat[Evolution of quantity ${\cal E}_{\Delta t,h}$ for different choices of $h$, with $\Delta t=0.01$.]{\includegraphics[width=0.45\linewidth]{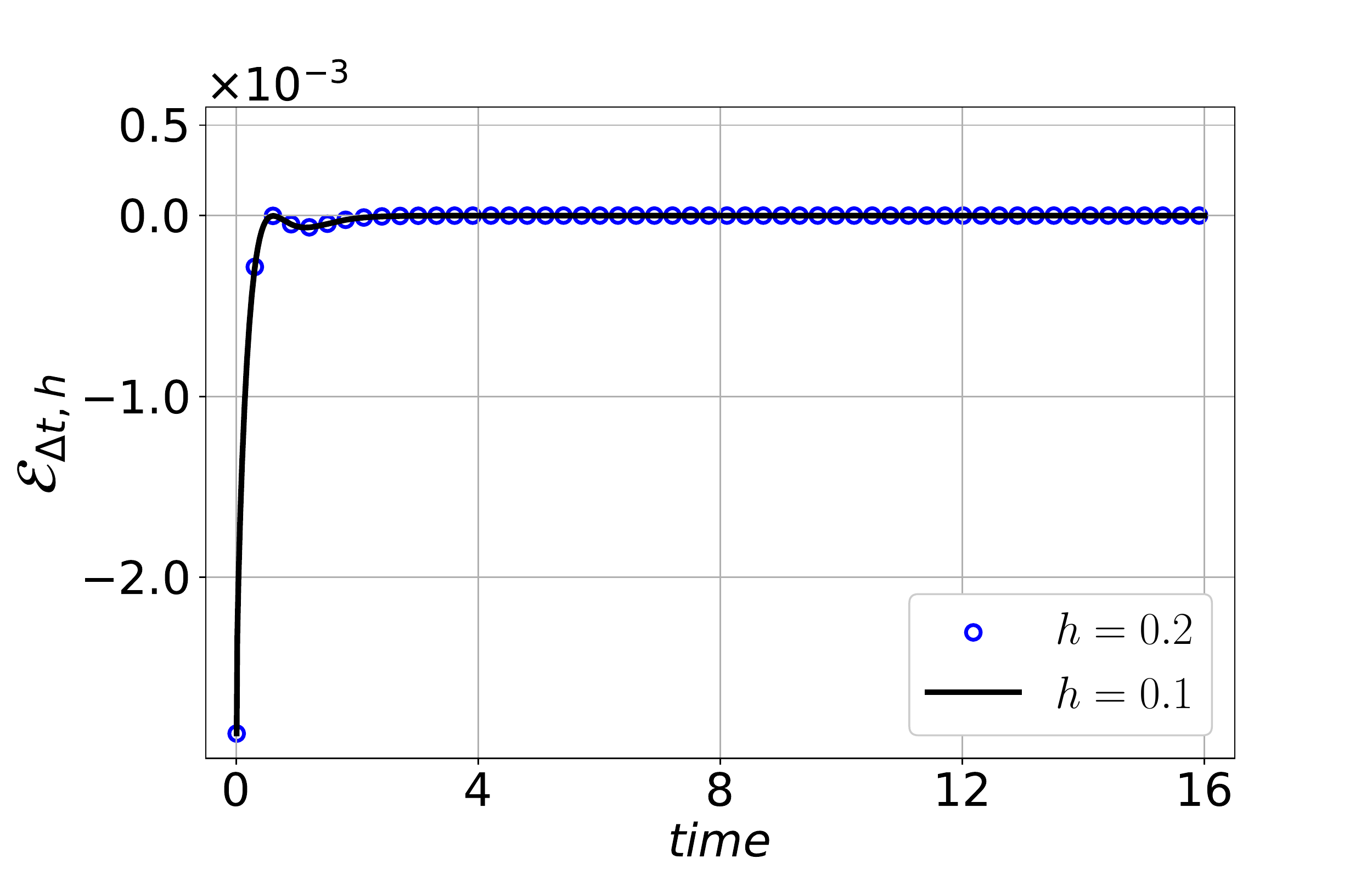}\label{fig:E-vary-m3d}}\hspace{0.08\linewidth}\subfloat[Evolution of quantity ${\cal E}_{\Delta t,h}$ for different choices of $\Delta t$, with $h=0.2$.]{\includegraphics[width=0.45\linewidth]{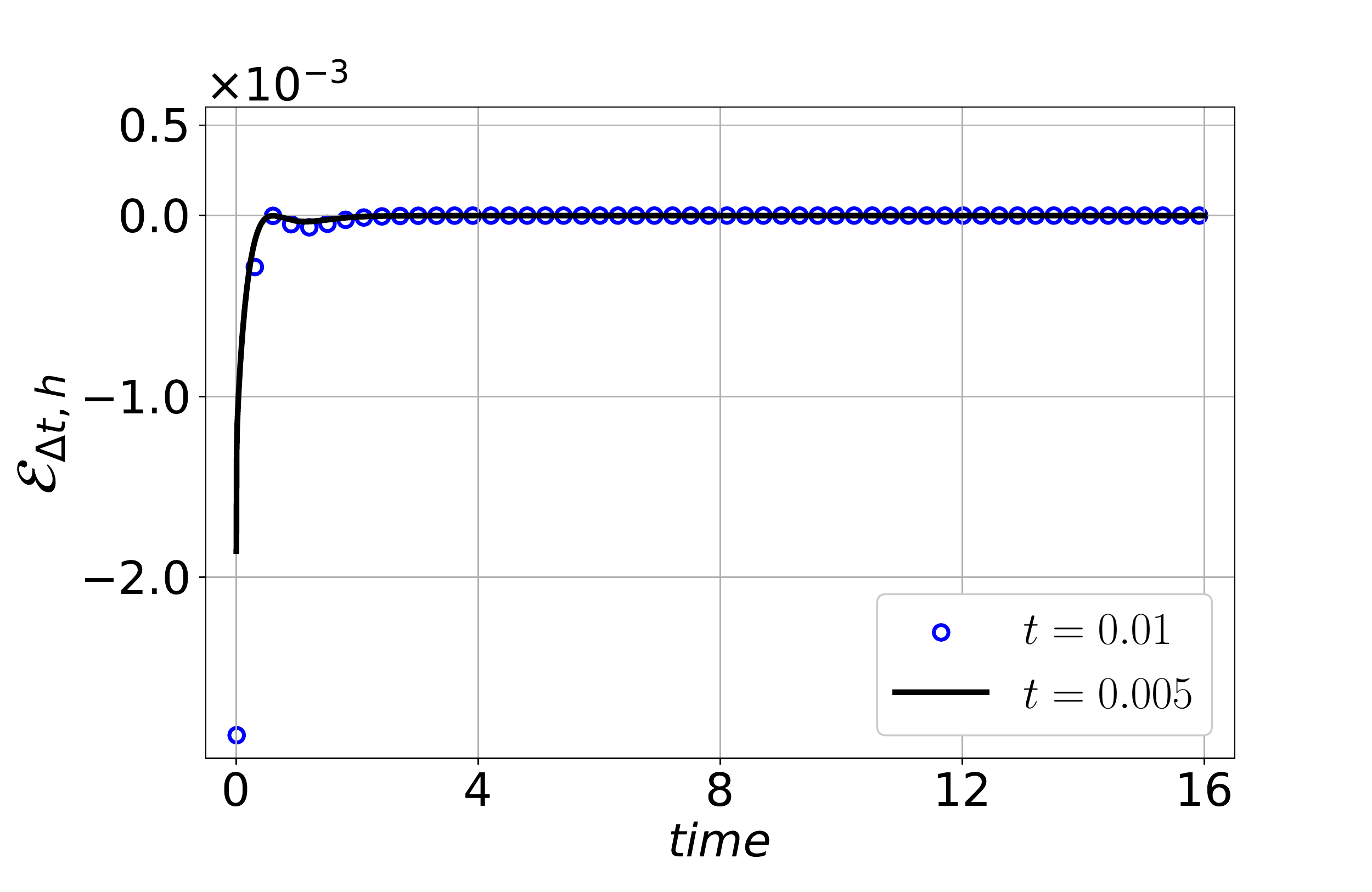}\label{fig:E-vary-t3d}}

\caption{Comparison of discrete energy balance ${\cal E}_{\Delta t,h}$ (\ref{eq:discr-count-eb}) for two different choices of mesh parameters $h$ and time steps $\Delta t$. Bond number is kept fixed, $\noBo=0.4$. Remark: circle markers are not plotted at each time step to improve visibility.}
\label{fig:3D-energies}
\end{figure}

\subsection{In--dept investigation of the scheme capabilities on realistic and complex 3D setup}\label{subsec:detail-drop}
For the final test, we consider a droplet of dimensionless volume $V_0=4.1$ on an inclined surface with the inclination angle $\alpha=\pi/4$. Recall, we chose the coordinate system in which $xy$--plane is aligned with the supporting surface and inclination is performed in $xz$--plane -- see definition (\ref{eq:inclination-k}). Furthermore, to increase the complexity, we consider a non--homogeneous supporting surface $\Gamma$ where the non--homogeneity is introduced through the varying static contact angle, $\theta_s=\theta_s(\bfx_\Gamma)$, defined by
\[
\theta_s(\bfx_\Gamma) = \frac{5\pi}{6} + \chi_{\{x<5\}}\frac{-x+1-y}{20} + \chi_{\{x>5\}}\frac{-4-y}{20}\;,\:\bfx_\Gamma=(x,y)\in[-1,9]\times[-2,2]\;,
\] 
where $\chi$ denotes the characteristic function on $\bbmR$ (see Figure~\ref{fig:cont_angle}). This example was partially motivated by work in \cite{mistura17}. The initial geometry is constructed as described in Subsection~\ref{subsec:charac-length} with the initial static contact angle $\theta_s=5\pi/6$. The rest of the dimensionless numbers are as follows: $\noLa^{1/2} = 10$, $\noBo = 0.3$ and $\varsigma = 1$. For the results presented below, $\Delta t = 0.02$ and initial mesh is chosen to be uniform with $h=0.15$ ($1253$ vertices). Simulation was run until $T=12$.

Figure~\ref{fig:init-final-complex3D} shows droplet states at times $t=0$, $6$, $9$ and $12$. It can be observed that within the time interval $[0,12]$, the droplet has traveled a significant distance. In Figure~\ref{fig:cont_angle}, droplet trajectory from the top view (onto $xy$--plane) and static contact angle are shown. It can be observed that the non--uniformity of the surface results in "pulling--off" the droplet out of the inclination direction. We confirm that this does not occur for the case of the homogeneous supporting surface in our simulations, although, such results are not reported here. This case also illustrates the importance of the full 3D simulations -- clearly, such scenario cannot occur in 2D.

\begin{figure}[h]
\subfloat[Droplet states at times $t=0$, $t=6$, $9$ and $t=12$.]{\includegraphics[width=0.48\linewidth]{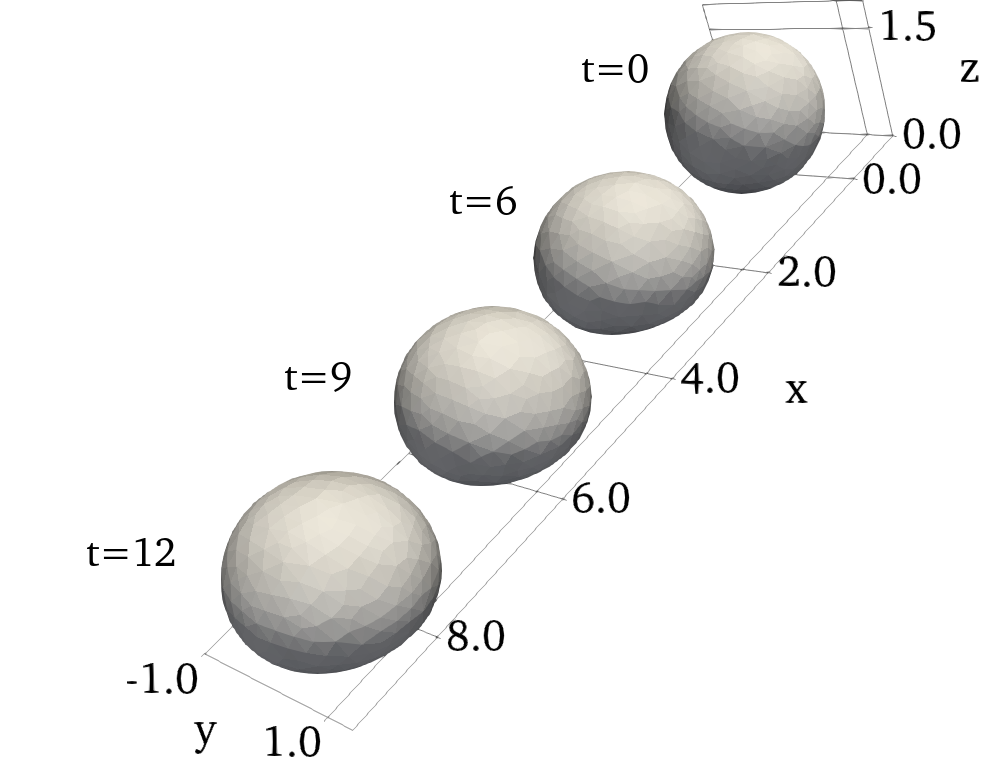}\label{fig:init-final-complex3D}}\hspace{0.04\linewidth}\subfloat[The droplet trajectory and inclination direction (solid line) are shown on the left. Static contact angle is shown on the right.]{\includegraphics[width=0.48\linewidth]{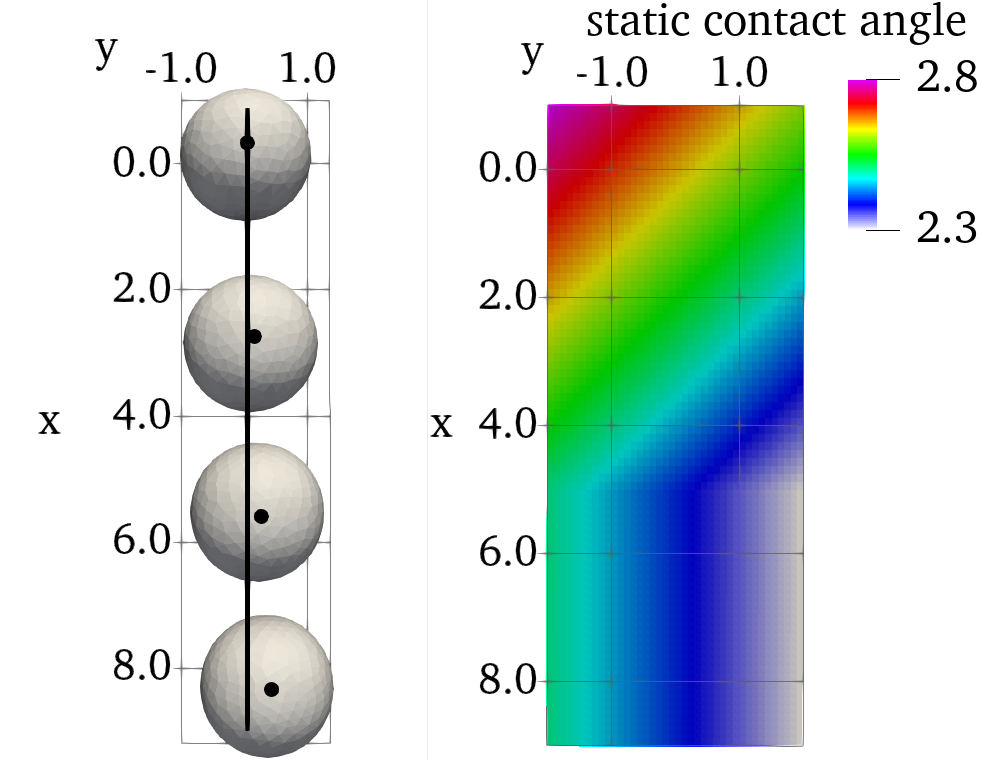}\label{fig:cont_angle}}

\caption{In Figure~\ref{fig:init-final-complex3D} droplet states at different time instants are shown, including initial ($t=0$) and final time ($t=12$). In Figure~\ref{fig:cont_angle} droplet trajectory (black dots denote the center of mass) and static contact angle are shown. It can be observed that the droplet is "pulled--off" off the inclination direction (see definition \ref{eq:inclination-k}).}
\label{fig:3dcomplex-initfinal}
\end{figure}

As it will be shown in a moment, the droplet dynamics exhibits (simultaneously) both sliding and rolling regimes. Since the interface is moved with the material velocity (recall equation (\ref{eq:practical-w-construct})), the mesh quality may be reduced after some time. Hence, we introduce a mesh quality test which is performed every five time steps: we test the mesh quality by evaluating the edge ratio $q_e$ and the aspect ratio $q_a$. The edge ratio of a single tetrahedron is defined as the ratio of its longest and shortest edge, and $q_e$ is the maximum value over all tetrahedrons in mesh. The aspect ratio of a single tetrahedron $K$ is the ratio of its longest edge $h_{\textrm{max}}(K)$ and the radius of an inscribed sphere $r(K)$, $q_a(K)=h_{\textrm{max}}(K) / (2\sqrt{6}\;r(K))$. The aspect ratio of the mesh is then the maximum value of $q_a(K)$ over all tetrahedrons $K$. The mesh adaptation criterion is then defined by: $q_e>3$ or $q_a>4$. In this case, an automatic mesh adaptation is performed with an aim of a uniform mesh with mesh parameter $h=0.15$, i.e. the aim is the mesh with the similar quality as the initial mesh. Automatic mesh adaptation is performed within FreeFem++ with a {\it built in} meshing tool {\it Mmg} (\cite{dapogny14}). Relative volume oscillations do not exceed the order of $10^{-4}$ throughout the whole simulation.

To understand the droplet dynamics better, we decompose the velocity at point $\bfp\in\Omega$, $\bfv_{\bfp}$, into the translational velocity of the {\it center of mass}, $\bfv_\textrm{cm}$, and the rotational velocity, $\bfv_{\bfp,\textrm{cm}}$, i.e.
\begin{equation}\label{eq;vel-decomp}
\bfv_{\bfp} = \bfv_\textrm{cm} + \bfv_{\bfp,\textrm{cm}}\;\textrm{ with }\: \bfv_{\textrm{cm}} = \frac{1}{\vert\Omega\vert}\int\limits_{\partial\Omega}\bfx \bfv\cdot\bfn\dS\;,
\end{equation}
where $\vert\Omega\vert$ denotes the volume of $\Omega$. In Figure~\ref{fig:rot-vel-glyp-stream} the rotational velocity $\bfv_{\bfp,\textrm{cm}}$ is shown with glyphs and streamlines representations. It may be observed that in coordinate system which follows the center of mass of the droplet, droplet exhibits solid body--like rotation around its center of mass. We mention that, throughout the simulations, we have observed that the rotational component of the velocity gets smaller by increasing the Laplace number $\noLa$. More specifically, for larger Laplace numbers the droplet dynamics regime is dominantly sliding. These results, however, are not reported here for the sake of conciseness and since such dynamics is less complex and attractive. 

\begin{figure}[h]
\subfloat[The rotational component of the fluid velocity, $\bfv_{\bfp,\textrm{cm}}$, at $t=12$. Glyphs representation.]{\includegraphics[width=0.48\linewidth]{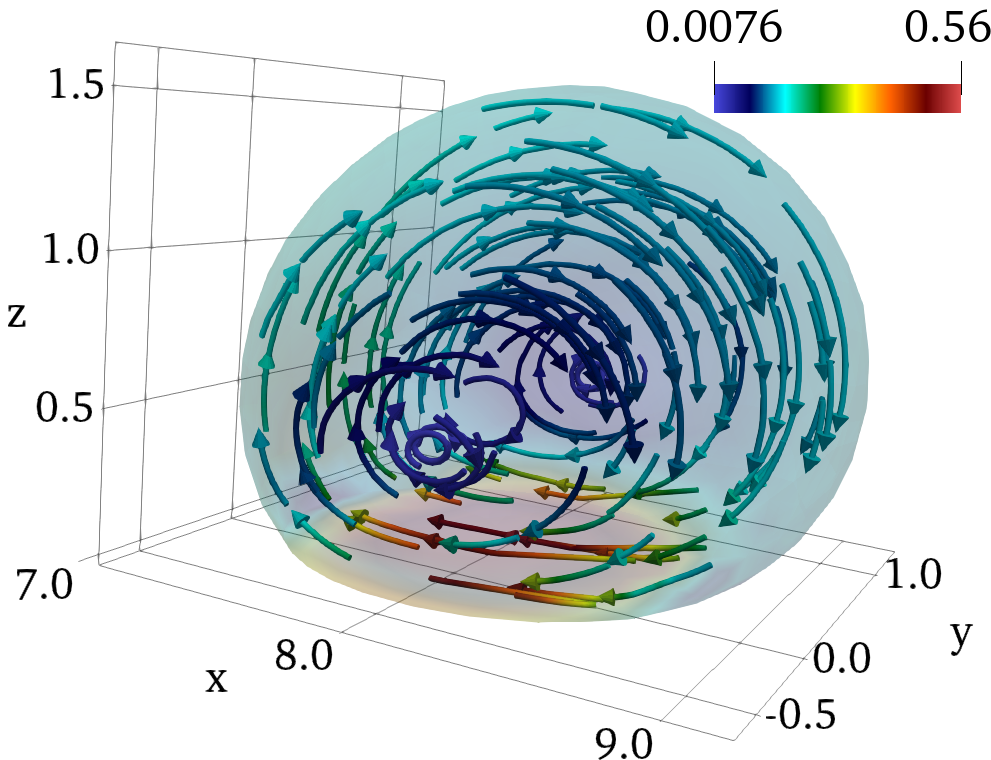}\label{fig:3Dglyphs}}\hspace{0.04\linewidth}\subfloat[Velocity streamlines of the rotational component of the fluid velocity, $\bfv_{\bfp,\textrm{cm}}$, at $t=12$.]{\includegraphics[width=0.48\linewidth]{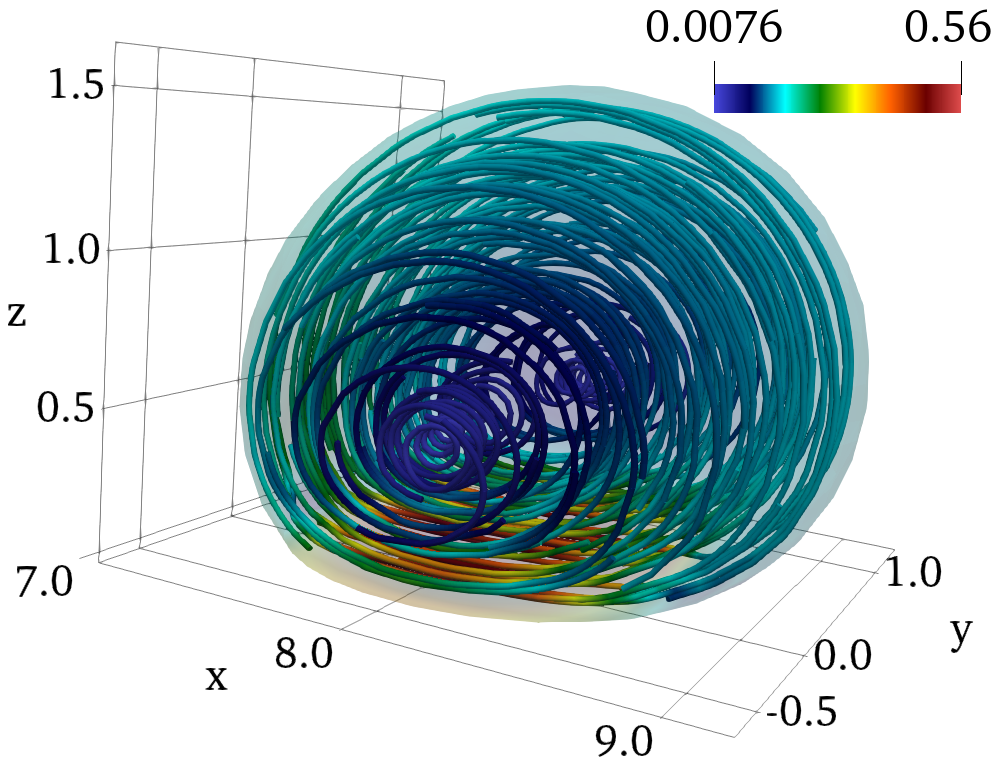}\label{fig:3Dstreams}}

\caption{Glyph and streamline representations of the rotational component of the fluid velocity, $\bfv_{\bfp,\textrm{cm}}$, defined by equation (\ref{eq;vel-decomp}). Color indicates the magnitude. The corresponding translational component of the velocity at $t=12$ equals $\bfv_\textrm{cm}={[0.89,0.05,2\times 10^{-3}]^T}$.}
\label{fig:rot-vel-glyp-stream}
\end{figure}

In Figure~\ref{fig:3Dstream-rot-xz}, streamlines of the field composed of the $x$ and $z$ components of the rotational velocity and the total $y$ velocity component ($[(\bfv_{\bfp,\textrm{cm}})_x,(\bfv_{\bfp})_y,(\bfv_{\bfp,\textrm{cm}})_z]^T$) is shown. The aim is to show that the flow direction is indeed off the inclination direction and pulled towards the area of $xy$--plane with stronger wetting properties (i.e. smaller contact angle).  Figure~\ref{fig:3Dstream-rot-xz} should be compared with Figure~\ref{fig:cont_angle} to put in perspective. Figure~\ref{fig:3Dside-view} shows the side view of the droplet at time $t=12$, i.e. the view perpendicular to $xz$--plane. Clear difference between advancing and receding contact angles can be observed.

\begin{figure}[h]
\subfloat[Velocity streamlines at time $t=12$ of the field ${[(\bfv_{\bfp,\textrm{cm}})_x,(\bfv_{\bfp})_y,(\bfv_{\bfp,\textrm{cm}})_z]^T}$. View onto $yz$--plane. Color indicates the magnitude.]{\includegraphics[width=0.48\linewidth]{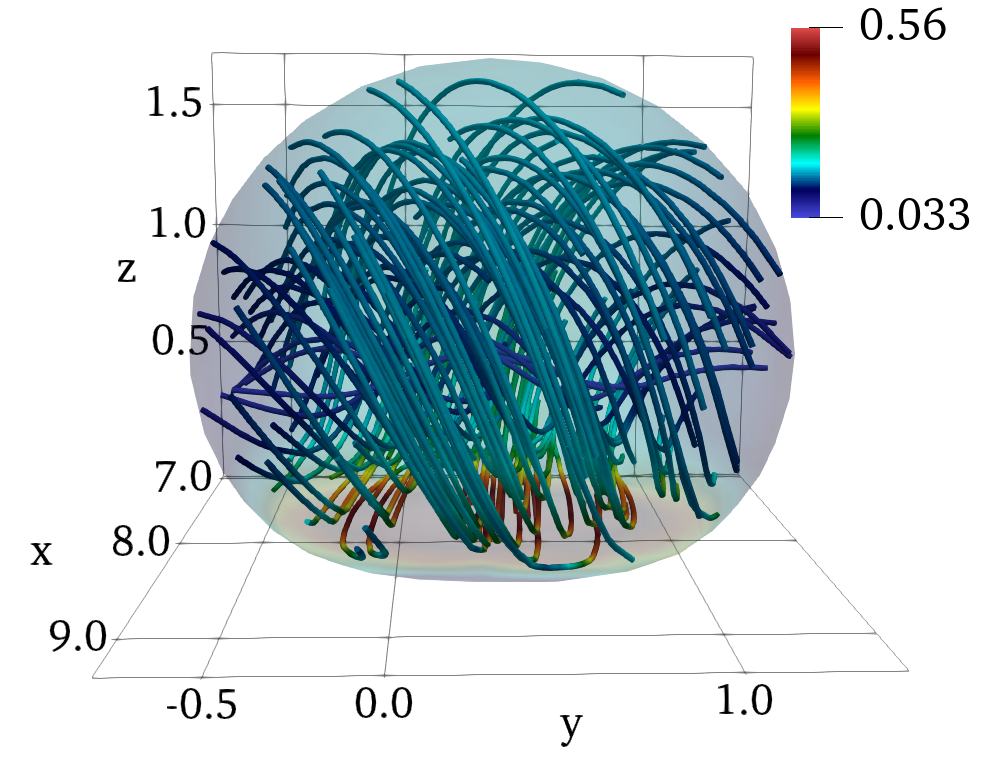}\label{fig:3Dstream-rot-xz}}\hspace{0.04\linewidth}\subfloat[Droplet profile at time $t=12$ from the side, view onto $xz$--plane. Difference between the advancing and receding contact angles may be observed.]{\includegraphics[width=0.48\linewidth]{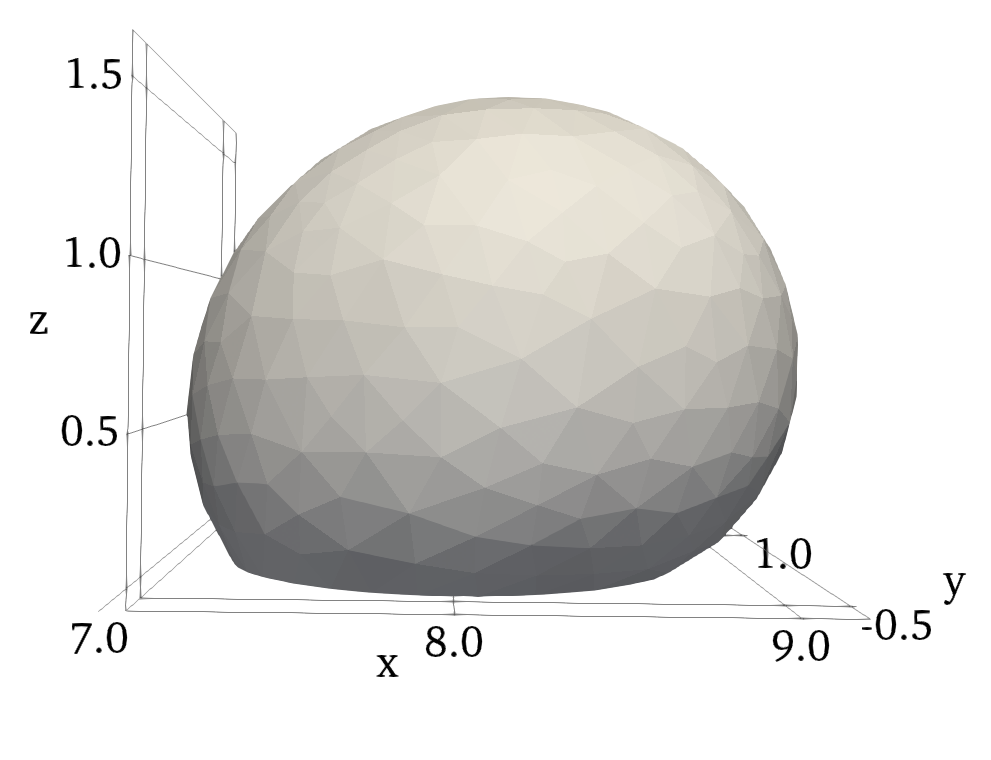}\label{fig:3Dside-view}}

\caption{Streamlines representation of the field composed of the $x$ and $z$ rotational components of the fluid velocity, and the total $y$ component, ${[(\bfv_{\bfp,\textrm{cm}})_x,(\bfv_{\bfp})_y,(\bfv_{\bfp,\textrm{cm}})_z]^T}$ (\ref{fig:3Dstream-rot-xz}). On the right, the profile of the droplet at $t=12$ from the side view (perpendicular to $xz$--plane) is shown. Note the difference between the advancing and receding contact angles.}
\label{fig:side-rot-vel-stream-stream}
\end{figure}

Finally, in Figure~\ref{fig:3d-energies-remesh}, the stability of the proposed scheme is analyzed. Figure~\ref{fig:3d-ddt-totalE} shows discrete energy balance ${\cal E}_{\Delta t,h}$ defined by (\ref{eq:discr-count-eb}), and it should be compared with Figure~\ref{fig:3d-euler-diss} which shows Euler dissipation defined by 
\begin{equation}\label{eq:euler-dissipation}
\frac{1}{\Delta t}\int\limits_{\Omega_h^n} \frac{1}{2}\vert\bfv_{h,n}^{n+1}-\bfv_{h}^{n}\vert^2\dbfx = \frac{1}{\Delta t}\left(
\int\limits_{\Omega_h^n}\frac{1}{2}\vert\bfv_{h,n}^{n+1}\vert^2\dbfx
- \int\limits_{\Omega_h^n}\bfv_{h,n}^{n+1}\cdot\bfv_{h}^{n}\dbfx
+ \int\limits_{\Omega_h^n}\frac{1}{2}\vert\bfv_{h}^{n}\vert^2\dbfx\right).
\end{equation}
Euler dissipation appears due to temporal discretization and it is a stabilizing property. It plays an essential role in the discrete energy estimate in Proposition~\ref{prop:discrete-energy-balance}, specifically, during the estimation of the transient term (see inequality (\ref{ineq:transient-term})). By definition, Euler dissipation is always non--negative and the lowest value achieved during the simulation is of order $10^{-4}$. This explains the stability of the scheme. Indeed, applying the Crank--Nicolson quadrature rule for SCLs on $\Sigma_h$ and $\Gamma_h$ in (\ref{prop:assumptions1-odt}), which are the only possible sources of spurious energy, the errors do not exceed the order of $10^{-6}$. Thus, the errors due to the SCLs on $\Sigma_h$ and $\Gamma_h$ are sufficiently small to be absorbed by Euler dissipation.

The evolution of discrete kinetic, potential ,surface tension and wetting energies is shown in Figure~\ref{fig:3d-energies-all}, and the discrete viscous and friction powers in Figure~\ref{fig:3d-power}. Carefully comparing these two figures, the process of transforming the potential energy into the kinetic energy and viscous work (and the rest) can be observed. Specially, one may note how the increase in wetting energy (i.e. in wetting area) and kinetic energy increases the friction power. This is also in agreement with the physical intuition.

The mesh adaptation frequency can be reduced in practice by relaxing the mesh quality criteria -- we imposed quite strict quality criteria in order to ensure that re--meshing will not affect the energy stability in a negative way. Better mesh adaptation procedures, which would result in better mesh quality after each adaptation and consequently reduce the re--meshing frequency, could be also implemented. As mentioned above, we used general FreeFem++ built--in algorithms for automatic mesh adaptation to reduce the implementation complexity.

We have also run the same simulation with a smaller time step ($\Delta t=0.01$) which resulted in essentially the same configuration at the final time $t=12$. Energy balance was improved as expected since error due to SCLs on $\Sigma_h$ and $\Gamma_h$ in (\ref{prop:assumptions1-odt}) was reduced.

\begin{figure}[h]
\subfloat[Evolution of quantity ${\cal E}_{\Delta t,h}$ and the re--meshing occurrences.]{\includegraphics[width=0.48\linewidth]{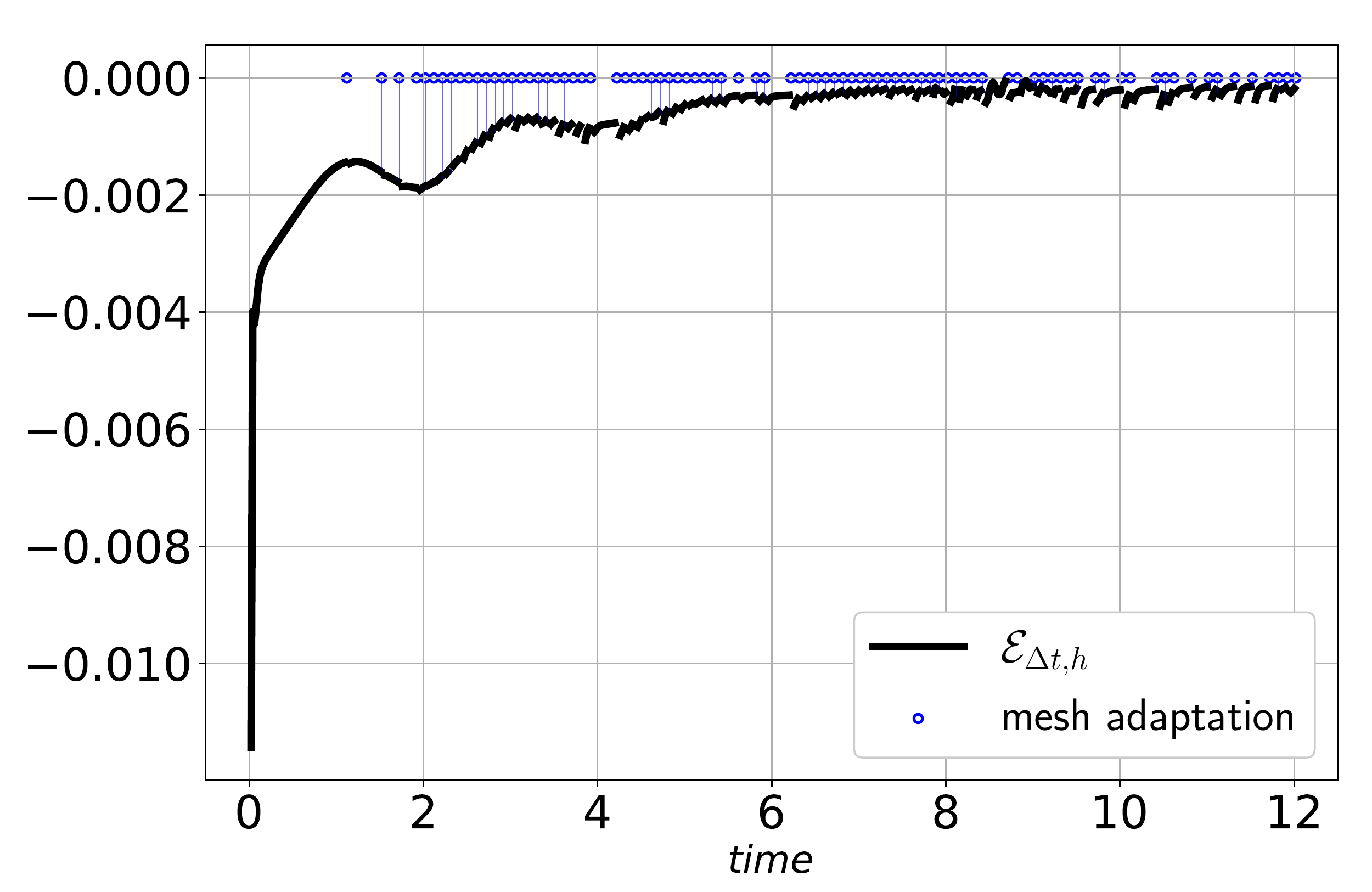}\label{fig:3d-ddt-totalE}}\hspace{0.04\linewidth}\subfloat[Evolution of discrete kinetic, potential, wetting and surface tension energies; $E_{k,h}$, $E_{p,h}$, $E_{w,h}$, $E_{fs,h}$.]{\includegraphics[width=0.48\linewidth]{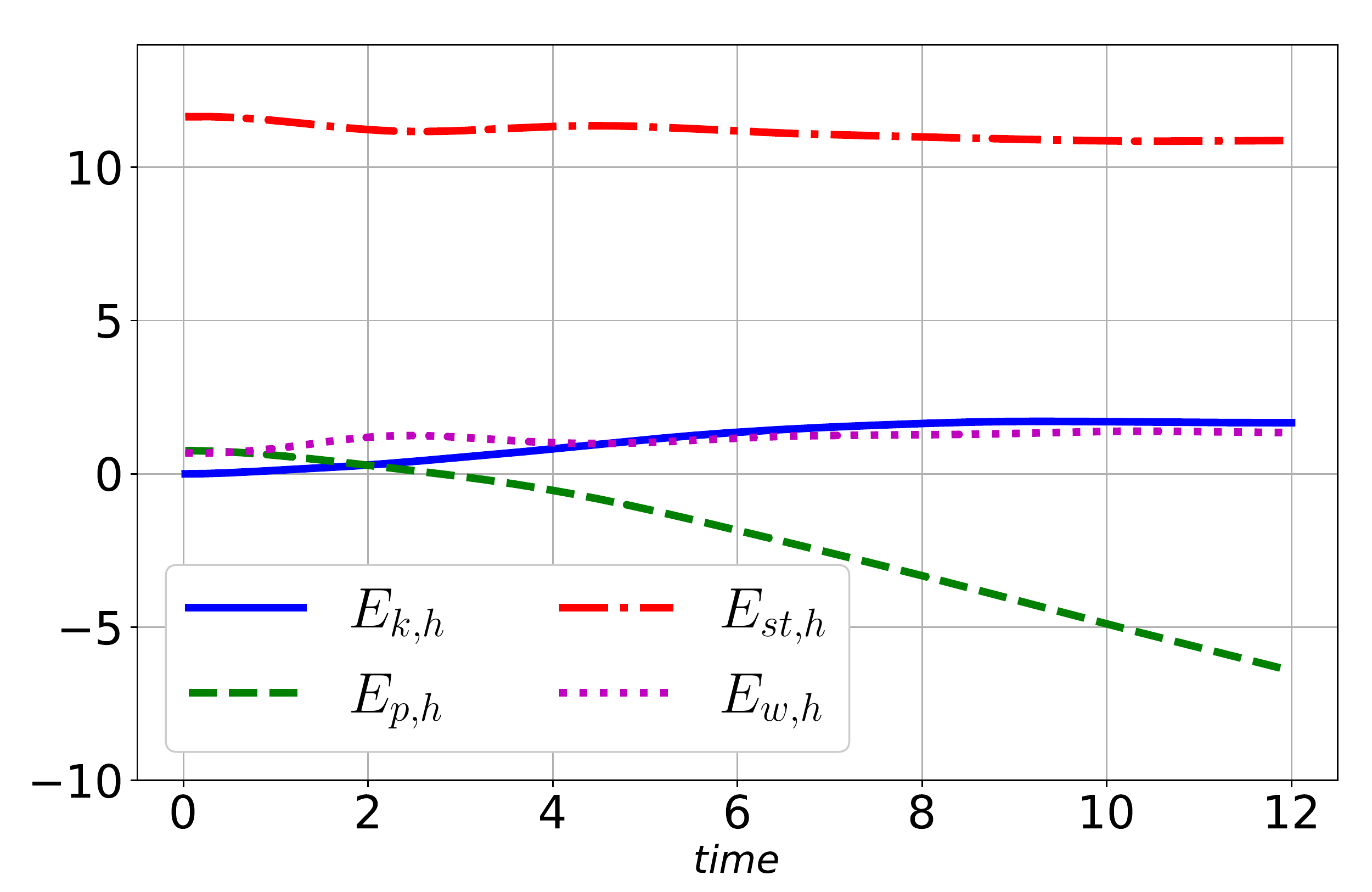}\label{fig:3d-energies-all}}

\subfloat[Evolution of Euler dissipation and the re--meshing occurrences.]{\includegraphics[width=0.48\linewidth]{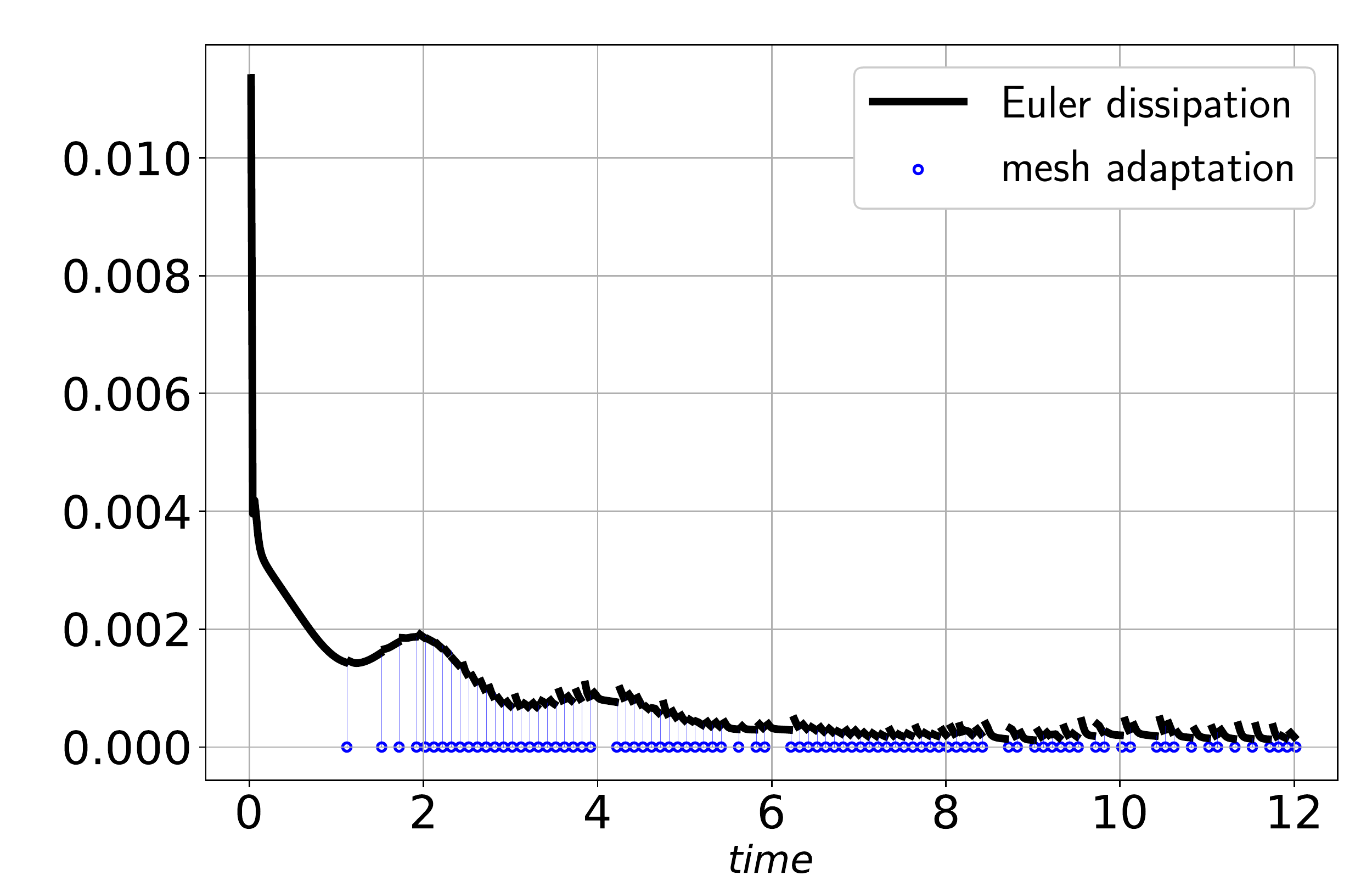}\label{fig:3d-euler-diss}}\hspace{0.04\linewidth}\subfloat[Evolution of discrete friction and viscous powers; $P_{fr}$, $P_v$.]{\includegraphics[width=0.48\linewidth]{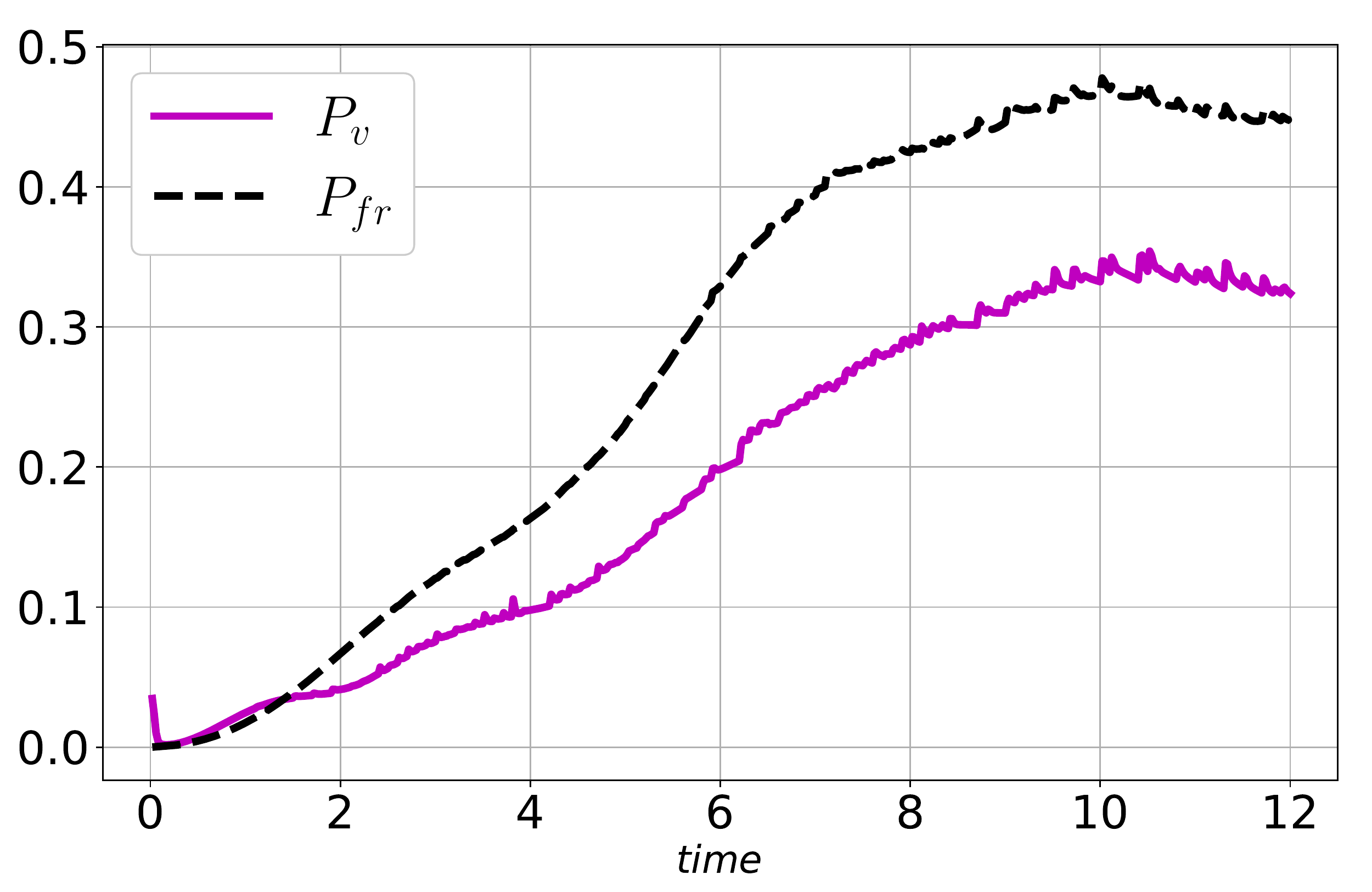}\label{fig:3d-power}}

\caption{Discrete energy balance ${\cal E}_{\Delta t,h}$ (\ref{eq:discr-count-eb}) in (\ref{fig:3d-ddt-totalE}). The evolution of discrete kinetic, potential, wetting and surface tension energies in (\ref{fig:3d-energies-all}). The evolution of friction and viscous powers in (\ref{fig:3d-power}). The evolution of Euler dissipation in (\ref{fig:3d-euler-diss}). The discontinuities in ${\cal E}_{\Delta t,h}$ and Euler dissipation profiles (\ref{fig:3d-ddt-totalE}) occur at the mesh adaptation steps (indicated by circles). $\Delta t=0.02$ and $h=0.15$.}
\label{fig:3d-energies-remesh}
\end{figure}

\section{Conclusion}\label{sec:conclusion}
We presented a novel ALE FEM scheme for the free-surface flows with moving contact line modeled through generalized Navier boundary conditions. The total energy balance equation has been derived from the governing system. It has been shown theoretically (and confirmed numerically) that the newly proposed scheme preserves the energy balance on the discrete level. Moreover, the discrete energy balance is preserved irrespectively of the choice of the time step $\Delta t$ (provided that $\Delta t$ is sufficiently small for the Newton's linearization method to converge). This ensures that no spurious energy, which may potentially cause the scheme blowup and non--physical behavior in long time simulations, is brought into the system. Numerical validation of theoretical results has been successfully performed. The proposed scheme may be applied to quite general cases of the free surface flows. The focus in this work was, however, exclusively on millimetric droplets in contact with solid surfaces. The theoretical results are valid for the cases of supporting surfaces which exhibit non--homogeneous properties and may be inclined. In this paper we only focused on cases where supporting surface is a plane in three-dimensional space. However, the scheme may be extended for more general, curved supporting surfaces. We have investigated the scheme capabilities on a fairly complex scenario which resulted in both sliding and rolling droplet dynamics.

The main drawback of the developed methodology in this paper is that the interface has to be evolved with the material velocity in order to guarantee the energy stability. In this scenario, when droplet dynamics starts to exhibit rolling regime the mesh quality reduces quickly and frequent mesh adaptation procedures are necessary. Generalization of the method proposed in this paper with a more general choice of the interface velocity aiming to reduce the remeshing frequency is currently under the investigation.

\section*{Acknowledgment}
This work was supported by National Health Research Institutes, Institute of Biomedical Engineering and Nanomedicine, Zhunan, Taiwan; project number BN--110--PP--08.

\appendix
\section{Transport theorems}
The following Reynolds transport theorem and its generalization for smooth manifolds of co--dimension $1$ are (often implicitly) used throughout this paper.
\begin{thm}[Reynolds transport theorem]\label{thm:RTT} 
Let $\Omega\subset\bbmR^d$ with sufficiently smooth boundary and let $\phi$ be a smooth spatial field, scalar or vector valued. Then, for any co--moving control volume $B\subset\Omega$
\begin{equation}
\begin{split}
& \frac{\td}{\dt}\int\limits_B\phi\dbfx =  \int\limits_B\Big( \frac{\tD \phi}{\Dt} + \phi\div\bfv \Big)\dbfx\textrm{ , and}\\
& \frac{\td}{\dt}\int\limits_B\phi\dbfx =  \int\limits_B\frac{\partial\phi}{\partial t}\dbfx +  \int\limits_{\partial B} \phi\bfv\cdot\bfn\dS,
\end{split}
\end{equation}
where $\bfv$ denotes the material velocity of $\Omega$ and $\tD/\Dt$ the Lagrangian derivative with respect to $\bfv$.
\end{thm}
For intuition behind Theorem~\ref{thm:RTT} and proof, see \cite{gurtin81}.

\begin{thm}\label{thm:generalizedRTT}
Let $\Sigma\subset\bbmR^d$ be a smooth manifold of co--dimension 1, and let $\phi$ be a smooth surface field on $\Sigma$, scalar or vector valued. Then, for any co--moving control submanifold (of co--dimension 1) $\omega\subset\Sigma$
\begin{equation}
\begin{split}
&\frac{\td}{\dt}\int\limits_\omega\phi\dS = \frac{\td}{\dt}\int\limits_\omega\Big( \frac{\tD \phi}{\Dt} + \phi\div_\Sigma\bfv \Big)\dS\textrm{ , and}\\
&\frac{\td}{\dt}\int\limits_\omega\phi\dS = \int\limits_\omega \Big( \frac{\partial\phi}{\partial t} + V\bfn\cdot\nabla\phi - 2H V\phi \Big) \dS + \int\limits_{\partial\omega}\phi\bfv\cdot\bfm_{\partial\omega}\ds,
\end{split}
\end{equation}
where $\bfm_{\partial\omega}$ denotes the co--normal to $\partial\omega$, $\bfv$ is the material velocity of $\Sigma$, $V=\bfv\cdot\bfn$ its normal component, $2H=-\div_\Sigma\bfn$ is the mean curvature of $\Sigma$, and $\tD/\Dt$ is the $\Sigma$--Lagrangian derivative with respect to $\bfv$.
\end{thm}
For the definition of partial and material derivatives with respect to $t$ variable for surface fields in Theorem~\ref{thm:generalizedRTT}, we refer to \citep{dziuk07,murdoch05,murdoch90,gurtin02}. Indeed, such derivatives are not straightforward to define since the domain is $(d-1)$--dimensional which is embedded and evolves in $d$--dimensional space.
\section{Discrete space conservation laws}\label{apx:dSCL}
Let us briefly discuss the discrete space conservation laws. They play a central role in the construction of the energy stable scheme in Section~\ref{sec:energy-stable-scheme}. In what follows, we intensively exploit the notation introduced in Section~\ref{sec:ALE}. The following two definitions define space conservation laws within FEM.
\begin{definition}[Space conservation law (SCL)]
Let $\Omega_h^n$ and $\Omega_h^{n+1}$ be two configurations at times $t_n$ and $t_{n+1}$. Furthermore, assume $\Omega_h^{n+1}$ is obtained from $\Omega_h^n$ via deformation map $\ALE_h^{[n,n+1]}$ re-constructed from velocity $\bfw_h^{n+1}$. We say that quadrature rule in variable $t$, $\intInn1$, satisfies discrete space conservation law if the following holds exactly:
\begin{equation}\label{defn:scl0}
\int\limits_{\Omega_h^{n+1}}\psi_h\dbfx - \int\limits_{\Omega_h^{n}}\psi_h\dbfx = \intInn1 \int\limits_{\Omega_h(t)} \psi_h\div\bfw_h^{n+1} \dbfx\;,
\end{equation}
for all $\psi_h$ belonging to a FEM space constructed over triangulation of $\kOmega_h$.
\end{definition}

\begin{definition}[Generalized SCL for discrete manifolds of co-dimension 1]
Let $\Sigma_h^n$ and $\Sigma_h^{n+1}$ be two $(d-1)$-dimensional configurations at times $t_n$ and $t_{n+1}$, embedded in $\bbmR^d$. Furthermore, assume $\Sigma_h^{n+1}$ is obtained from $\Sigma_h^n$ via deformation map $\ALE_h^{[n,n+1]}$ re-constructed from velocity $\bfw_h^{n+1}$. We say that quadrature rule in variable $t$, $\intInn1$, satisfies discrete space conservation law if the following holds exactly:
\begin{equation}\label{defn:scl1}
\int\limits_{\Sigma_h^{n+1}}\psi_h\dbfx - \int\limits_{\Sigma_h^{n}}\psi_h\dbfx = \intInn1 \int\limits_{\Sigma_h(t)} \psi_h\div_{\Sigma_h}\bfw_h^{n+1} \dbfx\;,
\end{equation}
for all $\psi_h$ belonging to a FEM space constructed over triangulation of $\widehat{\Sigma}_h$.
\end{definition}

Note that identities given in (\ref{prop:assumptions1}) represent various forms of SCLs. They may be derived from (\ref{defn:scl0}) and (\ref{defn:scl1}). 

We concisely review the methodology for constructing schemes with vanishing discrete SCL derived in \cite{ivancic19}. However, alternative SCL satisfying schemes may easily be integrated within the energy stable method derived in this paper. We illustrate the methodology derived in \cite{ivancic19} on the gravity term: for $t\in[0,\Delta t]$ and piecewise linear in time reconstruction of domain deformation, i.e. $\ALE_h^{[n+1,n]}(\bfx_h^{n+1})=\bfx_h^{n+1}-t\bfw_h^{n+1}(\bfx_h^{n+1})$,
\[\begin{split}
\int\limits_{\partial\Omega_h} \Phi_h\bfvarphi_h\cdot\bfn\dS =  \int\limits_{\partial\Omega_h^{n+1}}\frac{1}{\noFr^2}\Big( \bfk\cdot\bfx_h^{n+1} + (t-\Delta t)\bfk\cdot\bfw_h^{n+1} \Big)\bfvarphi\cdot\Cof\ALEF_{h,t}^{[n+1,n]}\bfn\dS.
\end{split}
\]
Above, $\Cof$ denotes the cofactor matrix. Noticing that right--hand integral is over a fixed--in--time domain, $\Omega_h^{n+1}$, and that the term under the integral sign is polynomial in time, it may be integrated exactly in time, i.e. we may take $\intInn1=\int_{t_n}^{t_{n+1}}$. An analogous approach may be taken in the rest of the SCL involving terms. However, using this approach for SCL on surfaces or curves in 3D space, the under integral term that appears is non--polynomial function in time and cannot be integrated exactly. Hence, $\intInn1$ has to be taken as some quadrature formula in single variable. While we only tried with Crank--Nicolson quadrature formula in numerical section, higher order formulas are straightforward to implement.

\section{}\label{apx:nomenclature}

\nomenclature[L,v]{\(\bfv\)}{fluid velocity}
\nomenclature[L,u]{\(\bfu\)}{supporting surface (\(\Gamma\)) velocity}
\nomenclature[L,w]{\(\bfw\)}{ALE velocity}
\nomenclature[L,A]{\(\ALE\)}{ALE map}
\nomenclature[L,p]{\(p\)}{fluid pressure}
\nomenclature[L,h]{\(\bfh\)}{mean curvature vector}
\nomenclature[L,H]{\(H\)}{mean curvature}
\nomenclature[G,k]{\(\kappa_{1,2}\)}{principal curvatures}
\nomenclature[L,k]{\(\bfk\)}{gravity vector}
\nomenclature[G,r]{\(\varrho\)}{fluid density}
\nomenclature[G,m]{\(\mu\)}{dynamic viscosity}
\nomenclature[G,g]{\(\gamma\)}{surface tension}
\nomenclature[G,s]{\(\varsigma\)}{friction slip coefficient}
\nomenclature[G,d]{\(\delta\)}{Dirac's delta distribution}
\nomenclature[G,F]{\(\Phi\)}{gravity potential}
\nomenclature[L,g]{\(g\)}{gravity acceleration}
\nomenclature[L,pc]{\(p_c\)}{characteristic pressure}
\nomenclature[L,U]{\(U\)}{characteristic velocity}
\nomenclature[L,L]{\(L\)}{characteristic length}
\nomenclature[L,tc]{\(t_c\)}{characteristic time}
\nomenclature[G,t]{\(\theta\)}{dynamic contact angle}
\nomenclature[G,ts]{\(\theta_s\)}{static contact angle}
\nomenclature[G,a]{\(\alpha\)}{supporting surface inclination angle}
\nomenclature[G,s]{\(\bfsigma\)}{fluid stress tensor}
\nomenclature[L,D]{\(\bbmD\)}{deviatoric stress tensor}
\nomenclature[G,O]{\(\Omega\)}{domain occupied by the liquid}
\nomenclature[G,G]{\(\Gamma\)}{solid--liquid interface}
\nomenclature[G,S]{\(\Sigma\)}{liquid--gas interface}
\nomenclature[G,e]{\(\eta\)}{solid--liquid--gas interface}
\nomenclature[L,P]{\(P_{\Sigma,\Gamma}\)}{projection operator}
\nomenclature[L,x]{\(\bfx_\Sigma\)}{embedding $\Sigma\hookrightarrow\bbmR^d$}
\nomenclature[L,d]{\(d\)}{space dimension ($d=2,3$)}
\nomenclature[L,n]{\(\bfn\)}{unit normal to boundary of $\Omega$}
\nomenclature[L,t]{\(\bft_{\eta}\)}{unit tangent to $\eta$}
\nomenclature[L,m]{\(\bfm\)}{unit co--normal to embedded surface}
\nomenclature[L,t]{\(t\)}{time variable}
\nomenclature[L,x]{\(\bfx\)}{space variable}
\nomenclature[L,E]{\({\cal E}_{\Delta t,h}\)}{discrete energy balance}
\nomenclature[G,Dt]{\(\Delta t\)}{time step}
\nomenclature[L,h]{\(h\)}{spatial discretization, mesh parameter}
\nomenclature[L,Ek]{\(E_k\)}{kinetic energy}
\nomenclature[L,Efs]{\(E_{fs}\)}{free surface energy}
\nomenclature[L,Ew]{\(E_w\)}{wetting energy}
\nomenclature[L,Ep]{\(E_p\)}{potential energy}
\nomenclature[L,Pv]{\(P_v\)}{viscous power}
\nomenclature[L,Pfr]{\(P_{fr}\)}{friction power}
\nomenclature[O,V]{\(\sV(\Omega)\)}{function space (velocity)}
\nomenclature[O,Q]{\(\sQ(\Omega)\)}{function space (pressure)}
\nomenclature[O,A]{\(\sA(\Omega)\)}{function space (ALE map)}

\begin{multicols}{2}
\printnomenclature
\end{multicols}

\newpage
\bibliography{references}
\end{document}